\documentclass[11pt,reqno]{amsart}
\usepackage{}
\usepackage{amssymb}
\usepackage{amsfonts}
\usepackage{txfonts}
\usepackage{mathrsfs}
\usepackage{latexsym}
\usepackage{graphicx}
\usepackage{amscd,amssymb,amsmath,amsbsy,amsthm,amsfonts}
\usepackage[all]{xy}
\usepackage[colorlinks,plainpages,urlcolor=blue]{hyperref} 
\usepackage{verbatim}

\hypersetup{
    colorlinks=true,       			
    linkcolor=blue, 
    citecolor=blue,        			
    filecolor=blue,      				
    urlcolor=blue           			
}

\usepackage{tikz}
\usetikzlibrary{arrows,calc}
\tikzset{
>=stealth',
help lines/.style={dashed, thick},
axis/.style={<->},
important line/.style={thick},
connection/.style={thick, dotted},
}

\newtheorem{theorem}{Theorem}[section]
\newtheorem{lemma}[theorem]{Lemma}

\newtheorem{corollary}[theorem]{Corollary}
\newtheorem{prop}[theorem]{Proposition}

\theoremstyle{definition}
\newtheorem{definition}[theorem]{Definition}
\newtheorem{example}[theorem]{Example}
\newtheorem{remark}[theorem]{Remark}

\newtheorem{thm}{Theorem}

\newcommand{\N}{\mathbb{N}}
\newcommand{\Z}{\mathbb{Z}}
\newcommand{\Q}{\mathbb{Q}}
\newcommand{\R}{\mathbb{R}}
\newcommand{\C}{\mathbb{C}}
\newcommand{\M}{\mathcal {M}}
\newcommand{\E}{\mathcal {E}}

\newcommand{\inj}{\hookrightarrow}
\newcommand{\h}{\mathfrak h}

\newcommand{\g}{\mathfrak g}

\DeclareMathOperator{\ad}{ad}
\DeclareMathOperator{\Hom}{Hom}

\DeclareMathOperator{\gr}{gr}

\DeclareMathOperator{\SL}{SL}
\DeclareMathOperator{\Der}{Der}

\DeclareMathOperator{\reg}{reg}

\DeclareMathOperator{\Hol}{Hol}
\DeclareMathOperator{\Lie}{Lie}

\DeclareMathOperator{\trig}{trig}

\DeclareMathOperator{\KZB}{KZB}
\DeclareMathOperator{\AKZ}{AKZ}

\DeclareMathOperator{\Spec}{Spec}
\DeclareMathOperator{\rank}{rank}
\DeclareMathOperator{\orb}{orb}

\newcommand {\Omit}[1]{}

\newcommand {\bcomment}[1]{}

\newcommand {\wh}[1]{\widehat{#1}}

\newcommand {\elp}{{\scriptscriptstyle{\operatorname{ell}}}}
\newcommand {\Aell}{{\mathfrak t}_\elp^\Phi}
\newcommand {\Aelltwo}{{\mathfrak t}_\elp^{\mathsf{A_1}}}
\newcommand {\ttrig}{{\mathfrak t}_{\trig}^\Phi}
\newcommand {\Bell}{B_\elp}
\newcommand {\Pell}{P_\elp}

\renewcommand {\Im}{\operatorname{Im}}
\newcommand {\rreg}{_{\scriptscriptstyle{\reg}}}
\newcommand {\etaa}[1]{\frac{d#1}{e^{#1}-1}}

\newenvironment{romenum}
{

\begin{enumerate}}{\end{enumerate}}

\newcommand {\CEE}{Calaque--Enriquez--Etingof }

\newcommand {\sfA}{\mathsf A}
\newcommand {\IR}{\mathbb R}
\newcommand {\IZ}{\mathbb Z}
\newcommand {\IC}{\mathbb C}
\newcommand {\wrt}{with respect to }

\newcommand {\mm}{\mathfrak m}
\newcommand {\aand}{\qquad\text{and}\qquad}
\newcommand {\wt}[1]{\widetilde{#1}}

\newcommand {\IN}{\mathbb N}

\title[Universal KZB Equations for arbitrary root systems]%
{Universal KZB Equations for arbitrary root systems}
\keywords{}
\author[V.~Toledano Laredo]{Valerio~Toledano Laredo}
\address{Department of Mathematics,
Northeastern University,
360 Huntington Ave.,
Boston, MA 02115, USA}
\email{V.ToledanoLaredo@neu.edu}

\author[Y.~Yang]{Yaping~Yang}
\address{
School of Mathematics and Statistics
The University of Melbourne,
813 Swanston Street, Parkville VIC 3010}
\email{yaping.yang1@unimelb.edu.au}

\begin{document}
\begin{abstract}
Generalising work of Calaque--Enriquez--Etingof \cite{CEE}, we construct a universal
KZB connection $\nabla_\elp$ for any finite (reduced, crystallographic) root system
$\Phi$. $\nabla_\elp$ is a flat connection on the regular locus of the elliptic configuration
space associated to $\Phi$, with values in a graded Lie algebra $\Aell$ with a
presentation with relations in degrees 2, 3 and 4 which we determine explicitly. The
connection $\nabla_\elp$ also extends to a flat connection over the moduli space
of pointed elliptic curves. We prove that its monodromy induces an
isomorphism between the Malcev Lie algebra of the elliptic pure braid group $\Pell^\Phi$
corresponding to $\Phi$ and $\Aell$, thus showing that $\Pell^\Phi$ is not $1$--formal
and extending a result of Bezrukavnikov valid in type $\sfA$ \cite{Bez}. 
We then study one concrete incarnation of our KZB connection, which is obtained by
mapping $\Aell$ to the rational Cherednik algebra $H_{\hbar, c}$ of the corresponding
Weyl group $W$. Its monodromy gives rise to an isomorphism between appropriate
completions of the double affine Hecke algebra of $W$ and $H_{\hbar, c}$. 
\end{abstract}
\Omit{ArXiv abstract Jan 2018
Generalising work of Calaque-Enriquez-Etingof, we construct a universal KZB
connection D_R for any finite (reduced, crystallographic) root system R. D_R is
a flat connection on the regular locus of the elliptic configuration space associated
to R, with values in a graded Lie algebra t_R with a presentation with relations
in degrees 2, 3 and 4 which we determine explicitly. The connection D_R also
extends to a flat connection over the moduli space of pointed elliptic curves.
We prove that its monodromy induces an isomorphism between the Malcev
Lie algebra of the elliptic pure braid group P_R corresponding to R and t_R,
thus showing that P_R is not 1-formal and extending a result of Bezrukavnikov
valid in type A. We then study one concrete incarnation of our KZB connection,
which is obtained by mapping t_R to the rational Cherednik algebra H_{h,c} of
the corresponding Weyl group W. Its monodromy gives rise to an isomorphism
between appropriate completions of the double affine Hecke algebra of W and
H_{h,c}. 
Comments: 50 pages
}
\maketitle
\tableofcontents

\section{Introduction}

The Knizhnik--Zamolodchikov--Bernard (KZB) connection was constructed by Felder--Wieczerkowski
in \cite{FW}. It is a flat connection on the vector bundle of WZW conformal blocks on the moduli space
$\M_{1,n}$ of elliptic curves with $n$ marked points. Calaque--Enriquez--Etingof later constructed a
universal KZB connection on $\M_{1,n}$ \cite{CEE}, which has coefficients in an arbitrary vector bundle.
The main goal of the current paper is to generalise the construction of \CEE to the elliptic configuration
space associated to an arbitrary root system.

\subsection{The universal KZB connection}

\bcomment{
2. is the trigonometric/
elliptic Lie algebra of $\Phi$ isomorphic to that for the dual root system
$\Phi^\vee$? Probably not, but check. If not, make a remark about this in the main
text (not in the introduction).}

\subsubsection{}

Let $E$ be a Euclidean vector space, and $\Phi\subset E^*$ a finite, reduced,
crystallographic root system. Let $Q\subset E^*$, $Q^\vee\subset E$ be the
root and coroot lattices, and $P\subset E^*,P^\vee\subset E$ the corresponding
weight and coweight lattices dual to $Q^\vee$ and $Q$ respectively. Denote the
complexification of $E$ by $\h$.

Let $\tau$ be a point in the upper half plane $\mathfrak{H}=\{z \in \C \mid
\Im(z)> 0\}$, and set $\Lambda_\tau:=\Z+\Z\tau\subset \C$. Consider the
elliptic curve $\mathcal{E}_\tau:=\C/{\Lambda_\tau}$ with modular parameter
$\tau$. Let $T:=\h/(P^\vee+\tau P^\vee)\cong\Hom_{\IZ}(Q,\E_\tau)$, which
is non--canonically isomorphic to $\E_\tau^n$, where $n=\rank(P^{\vee})$. Any root $\alpha
\in\Phi$ induces a map $\chi_\alpha:T\to\E_\tau$, with kernel $T_\alpha$.
We refer to  $T_{\reg}=T\setminus \bigcup_{\alpha\in \Phi}T_{\alpha}$, as
the \textit{elliptic configuration space} associated to $\Phi$.
If $E=\IR^n$ with standard coordinates $\{\epsilon_i\}$, and $\Phi=\{\epsilon
_i-\epsilon_j\}_{1\leq i\neq j\leq n}\subset E^*$ is the root system of type $\sfA_{n-1}$,
$T_{\reg}$ is the configuration space of $n$ ordered points on the elliptic
curve $\E_\tau$. 

\subsubsection{}

Let $\theta(z| \tau )$ be the Jacobi theta function,  which is a holomorphic function $\C\times
\mathfrak{H} \to \C$, whose zero set is $\{z\mid\theta(z| \tau )=0\}=\Lambda_\tau$, and such
that its residue at $z=0$ is $1$. (See Section
\S\ref{theta function}). Let $x$ be another complex variable, and set
\begin{equation*}
k(z, x|\tau):=\frac{\theta(z+x| \tau)}{\theta(z| \tau)\theta(x| \tau)}-\frac{1}{x}.
\end{equation*}
The function $k(z, x|\tau)$ has only simple poles at $z\in \Lambda_\tau$, and $k(z, x|\tau)$ is holomorphic in $x$. In other words, $k(z, x|\tau)$ belongs to $\Hol(\C-\Lambda_\tau)[\![x]\!]$. 

\subsubsection{}

Let $A$ be an algebra endowed with the following data: a set of elements $\{t_{\alpha}\}_{\alpha
\in\Phi}$, such that  $t_{-\alpha}=t_{\alpha}$, and two linear maps $x: \h\to A$, $y:\h\to A$. We
define an $A$--valued connection on $T_{\reg}$. It takes the following form. 
\begin{equation*}
\nabla_{\KZB, \tau}=d-\sum_{\alpha \in \Phi^+} k(\alpha, \ad(\frac{x_{\alpha^\vee}}{2})|\tau)(t_\alpha)d\alpha+\sum_{i=1}^{n}y(u^i)du_i,
\end{equation*}
where $\Phi_+\subset \Phi$ is a chosen system of positive roots, $\{u_i\}$, and $\{u^i\}$ are
dual basis of $\h^*$ and $\h$ respectively.\footnote{We assume further that $A$ is endowed
with a topology such that the expressions $k(\alpha, \ad(\frac{x_{\alpha^\vee}}{2})|\tau)(t_\alpha)$
converge. For example, $A$ could be complete \wrt a grading for which the elements
$x_\alpha$ have positive degree.} When $\Phi$ is the root system of type $\sfA_n$,
the connection above coincides with the universal KZB connection introduced by
Calaque--Enriquez--Etingof.


\begin{thm}[Theorem~\ref{thm:connection flat}]
\label{intro:thm:flat}
The connection  $\nabla_{\KZB, \tau}$ is flat if and only if the following relations
hold in $A$
\begin{enumerate}
\item For any rank 2 root subsystem $\Psi$ of $\Phi$, and $\alpha\in\Psi$,
\[[t_\alpha, \sum_{\beta \in \Psi^+}t_\beta]=0. \]
\item For any $u, v\in \h$
\[[x(u), x(v)]=0=[y(u), y(v)].\]
\item For any $u, v\in \h$,
\[[y(u), x(v)]=\displaystyle\sum_{\gamma\in \Phi^+}\langle v, \gamma\rangle \langle u, \gamma\rangle t_\gamma.\]
\item For any $\alpha\in\Phi$ and $u\in\h$ such that $\alpha(u)=0$,
\[[t_\alpha, x(u)]=0=[t_\alpha, y(u)].\]
\end{enumerate}
If, moreover, the Weyl group $W$ of $\Phi$ acts on $A$, then $\nabla_{\KZB, \tau}$ is
$W$--equivariant if and only
\begin{enumerate}
\item[(5)] For any $w\in W$, $\alpha\in\Phi$, and $u,v\in\h$,
\[w(t_\gamma)=t_{w\,\gamma},\qquad
w(x(u))=x(w\,u),
\aand w(y(v))=y(w\,v)\]
\end{enumerate}
\end{thm}

Define $\Aell$ to be the Lie algebra generated by a set of elements $\{t_\alpha\}_{\alpha\in \Phi}$,
such that $t_\alpha=t_{-\alpha}$, and two linear maps $x: \h\to\Aell$, $y:\h\to \Aell$, satisfying
the relations in Theorem \ref{intro:thm:flat}. The Weyl group $W$ acts on $\Aell$ according
to relation (5) in Theorem \ref{intro:thm:flat}. 

\subsection{The Malcev Lie algebra of $\Pell^\Phi$}

Let $\Pell^\Phi=\pi_1(T_{\reg}, x_0)$ be the fundamental group of 
$T_{\reg}$. We refer to $\Pell^\Phi$ as the {\it elliptic pure braid group} corresponding to $\Phi$.
The flatness of the universal KZB connection $\nabla_{\KZB, \tau}$ gives rise to the monodromy
map
\[\mu: \Pell^\Phi \to \exp(\widehat{\Aell}),\]
where $\widehat{\Aell}$ is the (pro--nilpotent) completion of $\Aell$ with respect to the grading
given by $\deg(x(u))=1=\deg(y(u))$, and $\deg(t_\alpha)=2$, and $\exp(\widehat{\Aell})$ is the
corresponding pro--unipotent group.

Let $J \subseteq \C \Pell^\Phi$ be the augmentation ideal of the group ring $\C \Pell^\Phi$,
and $\widehat{\C \Pell^\Phi}$ the completion of $\C \Pell^\Phi$ with respect to $J$. Let $U
(\Aell)$ be the universal enveloping algebra of $\Aell$, and $\widehat{U(\Aell)}$ its completion
\wrt the grading on $\Aell$. The monodromy map $\mu$ extends to the completions

\begin{thm}[Theorem~\ref{thm:formal}]
\label{intro:thm formal}\hfill
\begin{enumerate}
\item The monodromy map extends to an isomorphism of Hopf algebras
$\wh{\mu}: \widehat{\C \Pell^\Phi} \to \widehat{U(\Aell)}$.
\item The restriction of $\wh{\mu}$ to primitive elements is an isomorphism of
the Malcev Lie algebra of $\Pell^\Phi$ to the graded completion of $\Aell$.
\end{enumerate}
\end{thm}

Let $\Gamma$ be an abstract group. 
Recall the Malcev Lie algebra $\mm_\Gamma$ of $\Gamma$ is the subspace of
primitive elements in the completion of $\IC\Gamma$ \wrt its augmentation ideal. By definition, $\Gamma$ is 1--formal if $\mm_\Gamma$ is isomorphic, as a filtered
Lie algebra, to the graded completion of a quadratically presented Lie algebra (see,
e.g. \cite{DPS}).

Let $g\geq 1$, and $\Gamma_{g,n}$ be the fundamental group of the configuration space
of $n$ ordered points on an oriented surface of genus $g$. Using the theory of minimal models,
Bezrukavnikov gave an explicit presentation of the Malcev Lie algebra of $\Gamma_{g,n}$,
and proved in particular that $\Gamma_{g,n}$ is 1--formal if and only if $g>1$ or $g=1$
and $n\leq 2$ \cite{Bez}. \CEE rederived Bezrukavnikov's result in the case of genus $g=1$
by using Chen's iterated integrals associated to the universal KZB connection which they
introduced \cite{CEE}. A similar construction for surfaces of higher genus was carried by
Enriquez \cite{E}. A Tannakian interpretation of the construction of \CEE was recently
given by Enriquez--Etingof \cite{EE}.

The proof of Theorem \ref{intro:thm formal} is modelled on the approach of \cite{CEE, E}.

\begin{thm}[Theorem~\ref{th:not quadratic}]
\item The Lie algebra $\Aell$ is not quadratic. In particular, the pure elliptic
braid group $\Pell^\Phi$ is not 1--formal.
\end{thm}

\subsection{Extension to the moduli space}

Let $\M_{1, n}$ be the moduli space of pointed elliptic curves associated to a root system $\Phi$. 
More explicitly, let $\mathfrak{H}\ni\tau$ be the upper half plane.
The semidirect product $(P^\vee\oplus P^\vee)\rtimes \SL_2(\Z)$ acts on $\h \times \mathfrak{H}$. 
For $(\bold{n}, \bold{m}) \in (P^\vee\oplus  P^\vee)$ and $(z, \tau)\in \h \times \mathfrak{H}$, the action is given by translation:
$(\bold{n}, \bold{m})*(z, \tau):=
(z+\bold{n}+\tau\bold{m}, \tau)$. For $\left(\begin{smallmatrix}
a & b \\
c & d
\end{smallmatrix}\right)\in \SL_2(\Z)$, the action is given by
$\left(\begin{smallmatrix}
a & b \\
c & d
\end{smallmatrix}\right)*(z, \tau):=(\frac{z}{c\tau+d}, \frac{a\tau+b}{c\tau+d})$. 
 Let $\alpha(-): \h\to \C$ be the map induced by the root $\alpha\in \Phi$. 
 We define $\widetilde{H}_{\alpha, \tau}\subset \h\times \mathfrak{H}$ to be
\[
\widetilde{H}_{\alpha, \tau}=\{(z, \tau)\in \h\times\mathfrak{H}\mid \alpha(z) \in \Lambda_\tau=\Z+\tau \Z\}.
\]
The elliptic moduli space $\M_{1, n}$ is defined to be the quotient of 
$\h \times \mathfrak{H}\setminus\bigcup_{\alpha\in \Phi^+, \tau\in \mathfrak{H}} \widetilde{H}_{\alpha, \tau}$
by the action of $(P^\vee \oplus P^\vee) \rtimes \SL_2(\Z)$ action.

In the type $\sfA$ case, $\M_{1, n}$ is the moduli space of elliptic curves with $n$ marked points. We extend the KZB connection $\nabla_{\KZB, \tau}$ on $\mathcal{E}_\tau$ to a flat connection on $\M_{1, n}$. 

\subsubsection{}
To begin with, we have the following derivation of $\Aell$. 
Let $\mathfrak{d}$ be the Lie algebra with generators $\Delta_0, d, X$, and $\delta_{2m}( m\geq 1)$, and relations
\begin{align*}
&[d, X]=2X, \,\ [d, \Delta_0]=-2\Delta_0, \,\ [X, \Delta_0]=d,\\
&[\delta_{2m}, X]=0, \,\ [d, \delta_{2m}]=2m\delta_{2m}, \,\ (\ad\Delta_0)^{2m+1}(\delta_{2m})=0.
\end{align*}
The Lie subalgebra generated by $\Delta_0, d, X$ is a copy of $\mathfrak{sl}_2$ in $\mathfrak{d}$. We can decompose $\mathfrak{d}$ as $\mathfrak{d}=\mathfrak{d}_+\rtimes \mathfrak{sl}_2$, where $\mathfrak{d}_+$ is the subalgebra generated by $\{\delta_{2m}\mid m\geq 1\}$. 
In Proposition \ref{prop:d is deriv}, we show $\mathfrak{d}$ acts on $\Aell$ by derivation. 
 (See Proposition \ref{prop:d is deriv} for the action of $\mathfrak{d}$ on $\Aell$).

The Lie algebra $\Aell\rtimes \mathfrak{d}$ is $\Z^2-$ graded. The grading is given by 
\begin{align*}
&\deg(\Delta_0)=(-1, 1), \,\ \deg(d)=(0, 0), \,\  \deg(X)=(1, -1), \,\ \deg(\delta_{2m})=(2m+1, 1)\\
\text{and}\,\ &\deg(x(u))=(1, 0)\,\ \deg(y(u))=(0, 1), \,\ \deg(t_\alpha)=(1, 1).
\end{align*}

Let $G_n:=\exp(\widehat{\Aell\rtimes \mathfrak{d}_+})\rtimes \SL_2(\C)$
be the semiproduct of $\exp(\widehat{\Aell\rtimes \mathfrak{d}_+})$ and $\SL_2(\C)$, where the former is the completion of $\Aell\rtimes \mathfrak{d}_+$ with respect to the grading above. Following \cite{CEE}, in Proposition \ref{prop: G_n bundle}, we construct a principal bundle $P_n$ on $\M_{1, n}$ with structure group $G_n$, which is unique under certain conditions.

\subsubsection{}
We extend the connection $\nabla_{\KZB, \tau}$ to a connection $\nabla_{\KZB}$ on the principal bundle $P_n$.  We now describe the extension $\nabla_{\KZB}$. 

\Omit{Let $\widetilde{H}_{\alpha, \tau}\subset \h\times \mathfrak{H}$ be the hyperplane
$\widetilde{H}_{\alpha, \tau}=\{(z, \tau)\in \h\times\mathfrak{H}\mid \alpha(z) \in \Lambda_\tau=\Z+\tau \Z\}$. }Let $g(z, x|\tau):=k_x(z, x|\tau)$ be the derivative of function $k(z, x|\tau)$ with respect to variable $x$. 
Set $a_{2n}:=-\frac{(2n+1)B_{2n+2}(2i\pi)^{2n+2}}{(2n+2)!}$, where $B_n$ is the Bernoulli numbers and let $E_{2n+2}(\tau)$ be the Eisenstein series. 
Consider the following function on $\h\times \mathfrak{H}$ 
\begin{align*}
\Delta:=\Delta(\underline{\alpha}, \tau)=
&-\frac{1}{2\pi i}\Delta_0-\frac{1}{2\pi i}\sum_{n\geq 1}a_{2n}E_{2n+2}(\tau)\delta_{2n}+\frac{1}{2\pi i}\sum_{\beta \in \Phi^+}g(\beta, \ad\frac{ x_{\beta^\vee}}{2}|\tau)(t_\beta).
\end{align*}
This is a meromorphic function on $\h \times \mathfrak{H}$ valued in $\widehat{(\Aell\rtimes \mathfrak{d}_+)\rtimes \mathfrak{n}_+}\subset \Lie(G_n)$, (where $\mathfrak{n}_+=\C\Delta_0\subset \mathfrak{sl}_2$). It has only poles along the hyperplanes $\bigcup_{\alpha\in \Phi^+, \tau\in \mathfrak{H}} \widetilde{H}_{\alpha, \tau}$.

\begin{thm}[Theorem~\ref{thm:moduli}]
The following $\widehat{\Aell\rtimes\mathfrak{d}}$-valued KZB connection on $\M_{1, n}$ is flat.
\begin{align*}
\nabla_{\KZB}&=\nabla_{\KZB, \tau}-\Delta d\tau\\
&=
d-\Delta d\tau-\sum_{\alpha \in \Phi^+} k(\alpha, \ad\frac{x_{\alpha^\vee}}{2}|\tau)(t_\alpha)d\alpha+\sum_{i=1}^{n}
y(u^i)du_i.
\end{align*}
\end{thm}

\subsection{Trigonometric degeneration}

\subsubsection{}

Let $H=\h/P^\vee\cong\Hom_{\IZ}(Q,\IC^\times)$ be the complex algebraic torus with Lie algebra
$\h$ and coordinate ring given by the group algebra $\IC Q$. We denote the function corresponding
to $\lambda\in Q$ by $e^{2\pi\iota\lambda}$, and set $H_{\reg}=H\setminus\bigcup_{\alpha\in\Phi}
\{e^{2\pi\iota\alpha}=1\}$. The first named author introduced a universal trigonometric connection
$\nabla_{\trig}$ on $H$, with logarithmic singularities on $H\setminus H_{\reg}$ \cite{TL1}. The
connection is flat, $W$--equivariant, and takes values in a Lie algebra $\ttrig$ which is
described as follows.

Let $A$ be an algebra endowed with a set of elements $\{ t_\alpha\}_{\alpha\in \Phi}$ such
that $t_{-\alpha}=t_\alpha$, and a linear map $X: \h\to A$. The trigonometric connection
$\nabla_{\trig}$ is the $A$--valued connection on $H_{\reg}$ given by
\begin{equation}\label{eq:nabla trig in intro}
\nabla_{\trig}=d-
\sum_{\alpha\in\Phi_+}\frac{2\pi i d\alpha}{e^{2\pi i \alpha}-1}t_{\alpha}-
du_i\,X(u^i).
\end{equation}
where $\Phi_+\subset \Phi$ is a chosen system of positive roots, $\{u_i\}$ and $\{u^i\}$
are dual bases of $\h^*$ and $\h$ respectively, and the summation over $i$ is implicit.
The connection $\nabla_{\trig}$ is flat if, and only if the following relations hold \cite{TL1}
\begin{enumerate}
  \item For any rank 2 root subsystem $\Psi\subset \Phi$, and $\alpha\in \Psi$,
$[t_\alpha, \sum_{\beta\in \Psi_+}t_\beta]=0.$
  \item For any $u, v\in \h$,
$[X(u), X(v)]=0.$
  \item For any $\alpha\in \Phi_+$, $w\in W$ such that $w^{-1}\alpha$ is a simple root and $u\in \h$, such that $\alpha(u)=0$,
$[t_\alpha, X_w(u)]=0,$
      where $ X_w(u)=X(u)-\sum_{\beta\in \Phi_+\cap w \Phi_-}\beta(u)t_\beta$.
\end{enumerate}
By definition, the Lie algebra $\ttrig$ is the Lie algebra presented by the above relations. 

\subsubsection{}

The connection $\nabla_{\KZB, \tau}$ degenerates to a trigonometric
connection as the imaginary part of the modular parameter $\tau$ tends to $\infty$.

\begin{thm}[Proposition~\ref{prop:degeneration}]
\label{intr:degeration}
As $\Im\tau\to+\infty$, the connection $\nabla_{\KZB, \tau}$ degenerates to the following connection
$\nabla^{\deg}$ on $H_{\reg}$
\begin{align*}
\nabla^{\deg}
=&d-\sum_{\alpha \in \Phi^+} \frac{2\pi i t_{\alpha}}{e^{2\pi i \alpha}-1} d \alpha
+\sum_{i=1}^{n} \left(y(u^i) -\sum_{\alpha \in \Phi^+}(\alpha, u^i) 
\Big(\frac{2\pi i e^{2\pi i \ad(\frac{x_{\alpha^\vee}}{2})}}{e^{2\pi i \ad(\frac{x_{\alpha^\vee}}{2})}-1 }  -\frac{1}{\ad(\frac{x_{\alpha^\vee}}{2})}\Big) t_\alpha \right) du_i
\end{align*}

By universality of $\ttrig$, the above degeneration gives rise to a map $\ttrig\to \widehat{\Aell}$
given by
\begin{align*}
t_\alpha	\mapsto t_\alpha
\aand
X(u)	&\mapsto -y(u) +\sum_{\alpha \in \Phi^+}(\alpha, u) 
\Big(\frac{2\pi i e^{2\pi i \ad(\frac{x_{\alpha^\vee}}{2})}}{e^{2\pi i \ad(\frac{x_{\alpha^\vee}}{2})}-1 }  -\frac{1}{\ad(\frac{x_{\alpha^\vee}}{2})}\Big)t_\alpha.
\end{align*}

 \bcomment{
5. Is the degeneration isomonodromic in an appropriate sense?}
\end{thm}

\Omit{The map $A_{\trig}\to \widehat{\Aell}$ preserves the following gradings of $A_{\trig}$ and $\Aell$.
The grading on $A_{\trig}$ is given by $\deg(t_\alpha)=\deg(X(u))=1$.
The grading on $\Aell$ is given by  $\deg(t_\alpha)=\deg(y(u))=1$ and $\deg(x(u))=0$.
Valerio: I am unsure about giving x(u) degree zero since then the sums in ad(x) may not converge.
Perhaps the statement one can give is that the map trigo to elliptic preserves (descending) filtrations
given by elements of degree $\geq m$.
}

\Omit{
Such degeneration in
Theorem \ref{intr:degeration} gives a link between the elliptic connection on $T_{\reg}$ and trigonometric connection on $H_{\reg}$. 
It not only produces algebra homomorphisms, but also useful to compute the monodromy of elliptic connections. 
The monodromy of trigonometric connections can be computed using reduction method to rank 1 case in \cite{Ch1}. Theorem \ref{intr:degeration} provides a method to compute the monodromy of the elliptic connection, see \cite{TLY3}. 
}

\subsection{Rational Cherednik algebras and elliptic Dunkl operators}

For the last part of this paper, we study one concrete incarnation of our KZB connection
$\nabla_{\KZB}$, which is obtained by mapping $\Aell\rtimes \mathfrak{d}$ to the rational
Cherednik algebra of the Weyl group $W$. In special cases, this specialisation coincides
with the elliptic Dunkl operators. 

\Omit{
Let $\mathbb{H}_{\g}$ be the double affine Hecke algebra defined by Cherednik in \cite
[Definition 2.5]{C}. By definition, $\mathbb{H}_{\g}$ is the quotient of the group algebra
of the orbifold fundamental group of $T_{\reg}/W$ by the additional relations
\[(S_i-q t_i)(S_i+qt_i^{-1})=0,\]
where $S_i$ is the element of $\pi_1^{\orb}(T_{\reg}/W)$ homotopic to a small loop
around the root hypertori corresponding to root $\alpha_i$. See \cite[Definition 2.1]{C}
for the explicit presentation of $\mathbb{H}_{\g}$.
}

\subsubsection{} 

Let $H_{\hbar, c}$ be the rational Cherednik algebra of $W$ introduced in \cite{EG}.
$H_{\hbar, c}$ is generated by the group algebra $\C W$, together with a copy of
$S\h$ and $S\h^*$, and depends on two sets of parameters (see \cite{EG}, or
Section \ref{ss:rca} for the defining relations). Specifically, let $S\subset W$ be
the set of reflections, $K$ the vector space of $W$--invariant functions $S\to\C$,
and $\wt{K}=K\oplus\C$. $H_{\hbar, c}$ is an algebra over $\C[\wt{K}]$ which is $\IN$--graded, provided
the standard linear functions $\{c_s\}_{s\in S/W}$ and $\hbar$ on $K$ and $\IC$ are
given degree $2$.

\begin{thm}[Prop. \ref{prop:map to H_{h,c}}]\label{th:maps to RCA}
For any $a, b\in \C$, there is a Lie algebra homomorphism
$\Aell \to H_{\hbar, c}$,  defined as follows
\begin{gather*}
x(v)\mapsto a\pi(v),\qquad
y(u)\mapsto bu,\aand
t_\gamma\mapsto ab\left(\frac{\hbar}{h^\vee}-\frac{2c_{s_\gamma}}{(\gamma|\gamma)}s_{\gamma}\right),
\end{gather*}
for $u, v\in \h$ and $\gamma\in \Phi^+$, 
where $\pi: \h\to \h^*$ is the isomorphism induced by the non-degenerate bilinear form $(\cdot|\cdot)$ on $\h$,
$s_{\gamma}$ is the reflection corresponding to the root $\gamma$, and $h^\vee$ is the dual Coxeter
number of $\Phi$.
\end{thm}

Set $a=b=1$ for definiteness. Theorem \ref{th:maps to RCA} implies the
following

\begin{thm}[Theorem \ref{thm:KZB for rca}] 
The universal KZB connection $\nabla_{\KZB, \tau}$ specializes to the following 
elliptic KZ connection valued in the rational Cherednik algebra 
\[
\nabla_{H_{\hbar, c}, \tau}=d+\sum_{\alpha \in \Phi^+}\frac{2c_{\alpha}}{(\alpha|\alpha)}k(\alpha, \ad(\frac{ \alpha^\vee}{2})|\tau) s_{\alpha} d\alpha
   -\sum_{\alpha \in \Phi^+} \frac{\hbar}{h^\vee} \frac{\theta'(\alpha|\tau)}{\theta(\alpha|\tau)} d\alpha
   +\sum_{i=1}^{n} u^i du_i.
\] The elliptic connection $\nabla_{H_{\hbar, c}, \tau}$ is flat and $W$-equivariant.
\Omit{\item
The above elliptic KZ connection extends to a flat connection $\nabla_{H_{\hbar, c}}$ on $\mathcal{M}_{1, n}$ valued in the rational Cherednik algebra. It takes the following form
\[
\nabla_{H_{\hbar, c}}=\nabla_{H_{\hbar, c}, \tau}+\Delta d\tau,
\]
where 
\[
\Delta
=-\frac{1}{2\pi i}\frac{\textbf{F}}{\hbar}
+\frac{1}{\hbar\pi i}\sum_{n\geq 1}a_{2n}E_{2n+2}(\tau)\Big(\sum_{\alpha\in \Phi^+}\frac{c_\alpha^2}{(\alpha, \alpha)}(x_{\alpha^\vee})^{2n}\Big)+\frac{1}{2\pi i}\sum_{\beta \in \Phi^+}g(\beta, \ad\frac{ x_{\beta^\vee}}{2}|\tau)\left(\frac{\hbar}{h^\vee}-\frac{2c_{s_\beta}}{(\beta|\beta)}s_{\beta}\right). 
\]
Here $\textbf{F}$, $x_{\alpha}$ and $s_\beta$ are elements in $H_{\hbar, c}$. 
\end{enumerate}}
\end{thm}

Similar results are obtained in Section \ref{sec:rca} for the extension of $\nabla_{\KZB, \tau}$ to the
moduli space $\mathcal{M}_{1, n}$ by extending the homomorphism $\Aell \to H_{\hbar, c}$ to the
semidirect product $\Aell\rtimes \mathfrak{d}$.

\subsubsection{}

Let $\Bell$ be the elliptic braid group, that is the orbifold fundamental group $\Bell=\pi_1^{\orb}(T_{\reg}/
W)$, and $\mathbb{H}_{\g}$ the double affine Hecke algebra defined by Cherednik in \cite[Definition 2.5]
{C}. By definition, $\mathbb{H}_{\g}$ is the quotient of the group algebra of $\Bell$ by the additional relations
\[(S_i-q t_i)(S_i+qt_i^{-1})=0,\]
where $S_i$ is the element of $\pi_1^{\orb}(T_{\reg}/W)$ homotopic to a small loop around the divisor
corresponding to the root $\alpha_i$. See \cite[Definition 2.1]{C} for the explicit presentation of $\mathbb{H}_{\g}$.

The connection $\nabla_{H_{\hbar, c}, \tau}$ is flat and $W$-equivariant. Its monodromy yields a
homomorphism of $\mathbb{H}_{\g}$ onto the graded completion of $H_{\hbar, c}$ which becomes
an isomorphism after $\mathbb{H}_{\g}$ is also completed.

\Omit{
a one parameter
family of representations of $\Bell$, which factors through $\mathbb{H}_{\g}$,
where $q=e^{\pi i\frac{\hbar}{h^\vee}}$ and $t_i=e^{-\pi i \frac{2 c_{\alpha_i}}{(\alpha_i, \alpha_i)}}$.
Thus, the monodromy induces a functor from the category of finite-dimensional representations of rational Cherednik algebra to the category of finite-dimensional representations of the double affine Hecke algebra.}

\subsubsection{Relation with the affine Hecke algebra}
We produce an algebra homomorphism from the degenerate affine Hecke algebra to the completion of the rational Cherednik algebra using the degeneration of Theorem \ref{intr:degeration}. 

Cherednik in  \cite{Ch1} constructed the affine KZ connection. It is a flat and $W$-equivariant connection on $H\rreg$ valued in the degenerate affine Hecke algebra ${\mathcal H}'$. The degenerate affine Hecke algebra $\mathcal{H}'$ is the associative algebra generated by $\C W$ and the symmetric algebra $S\h$. 
Let  $\{s_\alpha,  x(u) \mid s_\alpha \in W,  u\in \h\}$ be the generators of $\mathcal{H}'$. 
The affine KZ connection can be obtained by specializing the universal trigonometric KZ connection $\nabla_{\trig}$ \eqref{eq:nabla trig in intro}. More precisely, we have a map $A_{\trig}\to \mathcal{H}'$, by $t_{\alpha} \mapsto k_\alpha s_\alpha$, and  $X(u)\mapsto x(u)$, for $\alpha\in \Phi$, and $u\in \h$. This gives the  affine KZ connection
\[
\nabla_{\AKZ}=d-\sum_{\alpha \in \Phi^+} \frac{2\pi i t_{\alpha}}{e^{2\pi i \alpha}-1} d \alpha k_{\alpha}s_\alpha
-\sum_{i} x(u^i) d u_i. \]

As $\Im \tau\to \infty$, by Proposition \ref{prop:degeneration} and \S\ref{sec:AKZ part}, the elliptic KZ connection $\nabla_{H_{\hbar, c}, \tau}$ degenerates to the following affine KZ connection.
\begin{align*}
\nabla
=&d-\sum_{\alpha \in \Phi^+}
-\frac{2c_\alpha}{(\alpha|\alpha)} \frac{2\pi i s_{\alpha}}{e^{2\pi i\alpha}-1}d\alpha
-\sum_{\alpha\in \Phi^+}
\left( 
\frac{e^{2\pi i \alpha}+1}{e^{2\pi i \alpha}-1}\frac{\pi i \hbar}{h^\vee}
+\frac{2c_\alpha}{(\alpha|\alpha)} 
\Big(\frac{2\pi i e^{2\pi i x_{\alpha^\vee}}}{e^{2\pi i x_{\alpha^\vee}}-1 }  -\frac{1}{x_{\alpha^\vee}}\Big) s_{\alpha} \right)d\alpha
+\sum_{i=1}^{n}y(u^i)du_i.
\end{align*}

By the universality of the affine KZ connection $\nabla_{\AKZ}$, we have
\begin{thm}
There is an algebra homomorphism from the degenerate affine Hecke algebra $\mathcal{H}'$ to $\wh{H}_{\hbar, c}$ by
\begin{align*}
k_\alpha &\mapsto -\frac{2c_\alpha}{(\alpha|\alpha)},  \,\ w\mapsto w, \,\ \text{for $w \in W$}, \\
 x(u) &\mapsto -y(u)+\sum_{\alpha\in \Phi^+} \alpha(u)\left( 
\frac{e^{2\pi i \alpha}+1}{e^{2\pi i \alpha}-1}\frac{\pi i \hbar}{h^\vee}
+\frac{2c_\alpha}{(\alpha|\alpha)} 
\Big(\frac{2\pi i e^{2\pi i x_{\alpha^\vee}}}{e^{2\pi i x_{\alpha^\vee}}-1 }  -\frac{1}{x_{\alpha^\vee}}\Big) s_{\alpha} \right).
\end{align*}
\end{thm}
\subsubsection{Relation with the elliptic Dunkl operator}
In \cite{BFV}, Buchstaber-Felder-Veselov defined elliptic Dunkl operators for Weyl groups. Etingof and Ma
in \cite{EM2} generalised these operators to an abelian variety $A$ with a finite group action, and defined
an elliptic Cherednik algebra as a sheaf of algebras on $A$. They also constructed certain representations
of the elliptic Cherednik algebra.

We show that those representations arise from the flat connection valued in the rational Cherednik algebra
$H_{0, c}$, with parameter $\hbar=0$. 

\begin{thm}[Proposition \ref{prop:dunkl operator}]
The flat connection $\nabla_{H_{0, c}, \tau}$ specialized on the vector bundle 
\[\h_{\C}\times_{(P^\vee \oplus \tau P^{\vee})} \C W\] coincides with the elliptic Dunkl operator
 \begin{align*}
\nabla=d-\sum_{w\in W}\sum_{\alpha \in \Phi^+}
\frac{2c_{\alpha}}{(\alpha|\alpha)} \frac{ \theta(\alpha+\alpha^\vee(w\rho)}{\theta(\alpha)\theta(\alpha^\vee(w\rho))} s_{\alpha} d\alpha
\end{align*}
in \cite{BFV} and \cite{EM2}. 
\end{thm}

\subsection{The elliptic Casimir connection}
In \cite{TLY2}, we also give another concrete incarnation,  the \textit{elliptic Casimir connection}, of the universal connections $\nabla_{\KZB, \tau}$ and $\nabla_{\KZB}$. This is an elliptic analogue of the rational Casimir connection \cite{TL0} and of the trigonometric Casimir connection \cite{TL1}. 
This elliptic Casimir connection takes values in the deformed double current algebra $D_{\hbar}(\g)$,
a deformation of the universal central extension of the double current algebra $\g[u, v]$ recently introduced by Guay \cite{G1, G2}.
The construction of a map from $\Aell$ to $D_{\hbar}(\g)$ relies crucially on a double loop presentation of $D_{\hbar}(\g)$ in \cite{G1, GY}.  
\subsection*{Acknowledgments}
We are grateful to Pavel Etingof for numerous helpful discussions. Part of the
work was done when Y.Y. visited the Universit\"at Duisburg-Essen. During the
revision of this paper, Y.Y. was hosted by Perimeter Institute for Theoretical
Physics. She is grateful for the hospitality and excellent working conditions
for those institutes. Y.Y. thanks Gufang Zhao for useful discussions.

\section{The universal connection for arbitrary root systems}
\label{sec:conn def}

\subsection{The Lie algebra $\Aell$}

Let $E$ be a Euclidean vector space, and $\Phi\subset E^*$ a finite, reduced,
crystallographic root system. Let $Q\subset E^*$ (resp. $Q^\vee\subset E$)
be the root lattice generated by the roots $\{\alpha| {\alpha\in\Phi}\}$ (resp.
coroots), and $P\subset E^*,P^\vee\subset E$ the corresponding weight and
coweight lattices dual to $Q^\vee$ and $Q$ respectively. Denote the complexification
of $E$ by $\h$.

For any subset $\Psi \subset \Phi$, and subring $R\subset \R$ of real numbers, let
$\langle \Psi\rangle_R \subset \h^*$ be the $R-$span of $\Psi$. By definition, a root
subsystem of $\Phi$ is a subset $\Psi\subset \Phi$ such that $\langle \Psi \rangle_\Z
\cap \Phi=\Psi$.
If $\Psi \subset \Phi$ is a root subsystem, we set $\Psi_+=\Psi\cap \Phi_+$.

\begin{definition}\label{def:holonomy Aell}
Let $\Aell$ be the Lie algebra generated by a set of elements $\{t_\alpha\}_{\alpha\in \Phi}$, such that $t_\alpha=t_{-\alpha}$, and two linear maps $x: \h\to \Aell$, $y:\h\to \Aell$. Those generators satisfy the following relations
\begin{description}
\item[(tt)] For any root subsystem $\Psi$ of $\Phi$, we have 
$[t_\alpha, \sum_{\beta \in \Psi^+}t_\beta]=0. $
\item[(xx)]
$[x(u), x(v)]=0$, for any $u, v\in \h$;
\item[(yy)]
$[y(u), y(v)]=0$, for any $u, v\in \h$;
\item[(yx)]
$[y(u), x(v)]=\displaystyle \sum_{\gamma\in \Phi^+}\langle v, \gamma\rangle \langle u, \gamma\rangle t_\gamma$.
\item[(tx)]
$[t_\alpha, x(u)]=0$, if $\langle \alpha, u\rangle=0$.
\item[(ty)]
$[t_\alpha, y(u)]=0$, if $\langle \alpha, u\rangle=0$.
\end{description}
\end{definition}
The Lie algebra $\Aell$ is bigraded, with grading $\deg(x(u))=(1, 0), \deg(y(v))=(0, 1)$, and $\deg(t_\alpha)=(1, 1)$, for any $u, v\in \h$ and $\alpha\in \Phi$. 
\begin{remark}
Unlike the rational and the trigonometric cases, there is no embedding of the Lie algebra $\Aelltwo$ of rank $1$ into the Lie algebra $\Aell$ of higher rank, such that, $x(u)\mapsto x(u), y(v)\mapsto y(v)$, and $t_{\alpha}\mapsto t_{\alpha}$. This fails since the defining relation
\[
[y(u), x(v)]=\sum_{\gamma\in \Phi^+}\langle v, \gamma\rangle \langle u, \gamma\rangle t_\gamma
\] is not preserved under the map. 
\end{remark}

\subsection{Theta functions}
\label{theta function}
In this subsection, we recall some basic facts about theta functions.

Let $\Lambda_\tau:=\Z+\Z\tau \subset \C$ and $\mathfrak{H}$ be the upper half plane, i.e. $\mathfrak{H}:=\{z \in \C \mid \Im(z)> 0\}$.
The following properties of $\theta(z| \tau)$ uniquely characterize the theta function $\theta(z| \tau )$  (see also \cite{CEE})
\begin{enumerate}
\item
$\theta(z| \tau )$ is a holomorphic function $\C\times \mathfrak{H} \to \C$, such that $\{z\mid\theta(z| \tau )=0\}=\Lambda_\tau$.
\item
$\frac{\partial \theta}{\partial z}(0| \tau)=1$.
\item
$\theta(z+1| \tau )=-\theta(z| \tau )=\theta(-z| \tau )$, and\,\  $\theta(z+\tau| \tau )=-e^{-\pi i \tau}e^{-2 \pi i z}
\theta(z| \tau )$.
\item
$\theta(z| \tau+1)=\theta(z| \tau )$, while $\theta(-z/\tau|-1/\tau)=-(1/\tau)e^{(\pi i/\tau)z^2}\theta(z| \tau )$.
\item Let $q:=e^{2\pi i \tau}$ and $\eta(\tau):=q^{1/24}\prod_{n\geq 1}(1-q^n)$. If we set $\vartheta(z|\tau):=\eta(\tau)^3 \theta(z|\tau)$, then $\vartheta(z|\tau)$ satisfies the differential equation
\[
\frac{\partial \vartheta(z, \tau)}{\partial \tau}
=\frac{1}{4\pi i}\frac{\partial^2 \vartheta(z, \tau)}{\partial z^2}.\]
\end{enumerate}
We have the following product formula of $\theta(z|\tau)$. Let $u=e^{2\pi i z}$, we have
\begin{equation}\label{prod formula}
\theta(z| \tau)=u^{\frac{1}{2}}\prod_{s>0}(1-q^su)\prod_{s\geq0}(1-q^su^{-1})\frac{1}{2\pi i}\prod_{s>0}(1-q^s)^{-2}.
\end{equation} 
\begin{remark}
The theta function $\theta(z|\tau)$ is normalized such that $\frac{\partial \theta}{\partial z}(0| \tau)=1$. One could take the  Jacobi theta function $\vartheta_1(z, q)=2\sum_{n=0}^{\infty}q^{\frac{n^2}{8}}\sin(nz)$ in \cite[Page 463-465]{WW}, then we have
\[
\theta(z| \tau ):= \frac{\vartheta_1(\pi z, q)}{\frac{\partial \vartheta_1}{\partial z}(0, q)}
=\frac{\vartheta_1(\pi z, q)}{2 q^{\frac{1}{8}}\prod_{n-1}^{\infty}(1-q^n)^3}. \]
\end{remark}
\subsection{Principal bundles on elliptic configuration space}
\label{subsec:bundles}
Consider the elliptic curve $\mathcal{E}_\tau:=\C/{\Lambda_\tau}$ with modular parameter $\tau\notin \R$.
Let $T:=\h/(P^\vee \oplus \tau P^\vee)$ be the adjoint torus, which non-canonically isomorphic to $\E_\tau^n$, for $n=\rank(P^{\vee})$. For any root $\alpha\in Q \subset \h^*$, the linear map $\alpha: \h=P^{\vee}\otimes_{\Z} \C \to \C$ induces a natural map 
\[
\chi_\alpha: T=\h/(P^\vee \oplus \tau P^\vee) \to \E_\tau.\]
Denote kernel of $\chi_\alpha$ by $T_\alpha$, which is a divisor of $T$.
Denote $T_{\reg}=T\setminus \bigcup_{\alpha\in \Phi}T_{\alpha}$, which we will refer as \textit{the elliptic configuration space}. In the type A case, when $\Phi=\{\epsilon_i-\epsilon_j \mid 1\leq i\neq j\leq n\}$, $T_{\reg}$ is the configuration space of $n$-points on elliptic curve $\E_\tau$. 

The Lie algebra $\Aell$ is positively  bi-graded. Let $\widehat{\Aell}$ be the degree completion of $\Aell$ and $\exp(\widehat{\Aell})$ be the corresponding group of 
$\widehat{\Aell}$. We now describe a principal bundle $\mathcal P_{\tau, n}$ on the elliptic configuration space $T_{\reg}$ with structure group $\exp(\widehat{\Aell})$. 

The lattice $\Lambda_\tau\otimes P^\vee$ acts on $\h=\C\otimes P^\vee$ by translations, whose quotient is $T$. For any $g\in \exp(\widehat{\Aell})$, and the standard basis $\{\lambda_i^\vee\}_{1\leq i \leq n}$ of $P^{\vee}$, we define an action of $\Lambda_\tau\otimes P^\vee$ on the group $\exp(\widehat{\Aell})$ by 
\[
\lambda_i^\vee (g)=g \,\ \text{and} \,\  \tau \lambda_i^\vee (g)=e^{-2\pi i x(\lambda^\vee_i)}g.\] We can then form the twisted product 
$\widetilde{\mathcal{P}}:=\h\times_{\Lambda_\tau \otimes P^\vee}\exp(\widehat{\Aell})$. It is a principal bundle on $T$ with structure group $\exp(\widehat{\Aell})$. Denote by $\mathcal P_{\tau, n}$ the restriction of the bundle $\widetilde{\mathcal{P}}$ over $T_{\reg} \subset T$.

Equivalently, let $\pi: \h \to \h/(P^\vee \oplus \tau P^\vee)$ be the natural projection. For an open subset $U\subset T_{\reg}$, the sections of $\mathcal P_{\tau, n}$ on $U$ are 
\[
\{ f: \pi^{-1}(U) \to \exp(\widehat{\Aell})\mid f(z+\lambda^\vee_i)=f(z),
f(z+\tau \lambda^\vee_i)=e^{-2\pi i x(\lambda^\vee_i)}f(z)\}. 
\]
\subsection{The universal KZB connection}
In this subsection, we construct the universal KZB connection for root system $\Phi$. 
As in \cite{CEE}, we set
\begin{equation}\label{function k}
k(z, x|\tau):=\frac{\theta(z+x| \tau)}{\theta(z| \tau)\theta(x| \tau)}-\frac{1}{x}.
\end{equation}
When $\tau$ is fixed, the element $k(z, x|\tau)$ belongs to $\Hol(\C-\Lambda_\tau)[\![x]\!]$. 
For $x(u)=x_u\in \Aell$, substituting $x=\ad x_u$ in \eqref{function k}, we get a linear map 
$\Aell\to (\Aell\otimes \Hol(\C-\Lambda_\tau))^{\wedge}$, where $(-)^{\wedge}$ is taking the completion.

We consider the $\widehat{\Aell}-$valued connection on $T_{\reg}$
\begin{equation}\label{conn any type}
\nabla_{\KZB, \tau}=d-\sum_{\alpha \in \Phi^+} k(\alpha, \ad(\frac{x_{\alpha^\vee}}{2})|\tau)(t_\alpha)d\alpha+\sum_{i=1}^{n}y(u^i)du_i,
\end{equation}
where $\Phi_+\subset \Phi$ is a chosen system of positive roots, $\{u_i\}$, and $\{u^i\}$ are dual basis of $\h^*$ and $\h$ respectively.

Note that the form \eqref{conn any type} is independent of the choice of $\Phi^+$. It follows from the equality $k(z, x|\tau)=-k(-z, -x|\tau)$, which is a direct consequence of the fact that theta function $\theta(z|\tau)$ is an odd function.

We now show that the KZB connection $\nabla_{\KZB, \tau}$ \eqref{conn any type} is a connection on the principal bundle $\mathcal{P}_{\tau, n}$. In order to do this, 
we rewrite $\nabla_{\KZB, \tau}$ \eqref{conn any type} as the form
\[
\nabla_{\KZB}=d-\sum_{i=1}^n K_i d\lambda_i^\vee=d-\sum_{i=1}^n \left(\sum_{\alpha \in \Phi^+} (\alpha, \alpha_i) k(\alpha, \ad(\frac{x_{\alpha^\vee}}{2})|\tau)(t_\alpha)-y(\alpha_i)\right) d\lambda_i^\vee. 
\]
\begin{prop}\label{quasi peri}
For any $1\leq i\leq n$, the function $K_i$  satisfies the conditions
\[K_i(z+\lambda_j^\vee)=K_i(z) \,\ \text{and}
\,\ K_i(z+\tau \lambda_j^\vee)=e^{-2\pi i x(\lambda_j^\vee)}K_i(z).\]
As a consequence, the KZB connection $\nabla_{\KZB, \tau}$ \eqref{conn any type} is a connection on the bundle $\mathcal{P}_{\tau, n}$. 
\end{prop}
\begin{proof}
Using the properties of the theta function $\theta(z|\tau)$, we have,
for any integer $m\in \Z$, 
\begin{align*}
&k(z+m, x|\tau)=k(z, x|\tau)\\
&k(z+\tau m, x|\tau)=e^{-2\pi i m x}k(z, x|\tau)+\frac{e^{-2\pi i m x}-1}{x}.
\end{align*}
Therefore, it is obvious that $K_i(\alpha+\lambda^\vee_j)=K_i(\alpha)$, for any $\alpha\in \Phi, \lambda_j^{\vee}\in P^{\vee}$.
We now check the second equality. 
It follows from the following computation. 
\begin{align*}
&K_i(\alpha+\tau \lambda^\vee_j)\\
=&\sum_{\alpha \in \Phi^+} (\alpha, \alpha_i) k(\alpha+\tau (\alpha, \lambda_j^\vee), \ad(\frac{x_{\alpha^\vee}}{2})|\tau)(t_\alpha)-y(\alpha_i)\\
=&\sum_{\alpha \in \Phi^+} (\alpha, \alpha_i) \Bigg(\exp\left(-2 \pi i (\alpha, \lambda_j^\vee)\frac{x(\alpha^\vee)}{2}\right)
k(\alpha, \ad(\frac{x_{\alpha^\vee}}{2})|\tau)+\frac{\exp\left(-2 \pi i (\alpha, \lambda_j^\vee)x_{\alpha^\vee}/2\right)-1}{x_{\alpha^\vee}/2}\Bigg)(t_\alpha)-y(\alpha_i)\\
=&\sum_{\alpha \in \Phi^+} (\alpha, \alpha_i) \Bigg(\exp\left(-2 \pi i x(\lambda_j^\vee)\right)
k(\alpha, \ad(\frac{x_{\alpha^\vee}}{2})|\tau)+\frac{-\exp\left(2 \pi i x(\lambda_j^\vee)\right)-1}{x(\lambda_j^\vee)}(\alpha, \lambda_j^\vee)\Bigg)(t_\alpha)-y(\alpha_i)\\
&\phantom{12345678901234567890}
\text{by the relation $[-(\alpha, \lambda_j^\vee)\frac{x(\alpha^\vee)}{2}+x(\lambda^\vee_j), t_\alpha]=0$.}
\end{align*}
\begin{align*}
=&\sum_{\alpha \in \Phi^+} (\alpha, \alpha_i) \exp\left(-2 \pi i x(\lambda_j^\vee)\right)
k(\alpha, \ad(\frac{x_{\alpha^\vee}}{2})|\tau)(t_\alpha)+\frac{\exp\left(-2 \pi i x(\lambda_j^\vee)\right)-1}{x(\lambda_j^\vee)}(\sum_{\alpha \in \Phi^+} (\alpha, \alpha_i)(\alpha, \lambda_j^\vee)t_\alpha)-y(\alpha_i)\\
=&\sum_{\alpha \in \Phi^+} (\alpha, \alpha_i) \exp\left(-2 \pi i x(\lambda_j^\vee)\right)
k(\alpha, \ad(\frac{x_{\alpha^\vee}}{2})|\tau)(t_\alpha)-(\exp\left(-2 \pi i x(\lambda_j^\vee)\right)-1)(y(\alpha_i))-y(\alpha_i)\\
&\phantom{12345678901234567890}\text{by the relation $[y(\alpha_i), x(\lambda^\vee_j)]=\sum_{\alpha \in \Phi^+} (\alpha, \alpha_i)(\alpha, \lambda_j^\vee)t_\alpha$.}\\
=&\sum_{\alpha \in \Phi^+} (\alpha, \alpha_i) \exp\left(-2 \pi i x(\lambda_j^\vee)\right)
k(\alpha, \ad(\frac{x_{\alpha^\vee}}{2})|\tau)(t_\alpha)-\exp\left(2 \pi i x(\lambda_j^\vee)\right)(y(\alpha_i))\\
=&\exp(-2 \pi i x(\lambda_j^\vee))K_i(\alpha).
\end{align*}
This completes the proof. 
\end{proof}
Let $W$ be the Weyl group generated by reflections $\{s_{\alpha} \mid \alpha\in \Phi \}$. Assume $W$ acts on the algebra $\Aell$. 
\begin{theorem}\label{thm:connection flat}
\begin{enumerate}
\item
The connection $\nabla_{\KZB, \tau}$ \eqref{conn any type} is flat if and only if the defining relations (tt), (xx), (yy), (yx), (tx), (ty) of $\Aell$ hold.
\item
The connection $\nabla_{\KZB, \tau}$ is $W$-equivariant if and only the action of $W$ on $\Aell$ is given by \[
s_\alpha(t_\gamma)=t_{s_\alpha(\gamma)}, \,\ 
s_\alpha(x(u))=x(s_\alpha u), \,\ s_\alpha(y(v))=y(s_\alpha v). \]
\end{enumerate}
\end{theorem}
\begin{proof}
The necessity of claim (1) follows from Corollary \ref{cor:ness}. 
We postpone the proof of  sufficiency of part (1) to the next section. We now show part (2) of the claim. The connection $\nabla_{\KZB, \tau}$ is $W$-equivariant if and only if $s_i^* \nabla_{\KZB, \tau}=\nabla_{\KZB, \tau}$, for simple reflection $s_i\in W$. 
Note that $s_i$ permutes the set $\Phi^+\setminus \{\alpha_i\}$, and we have the equality 
$k(-\alpha, -\ad\frac{x_{\alpha^\vee}}{2})=k(\alpha, \ad\frac{x_{\alpha^\vee}}{2})$.
Therefore, 
\begin{align*}
s_i^*\nabla_{\KZB, \tau}&=d-\sum_{\alpha \in \Phi^+} k(s_i\alpha, \ad(\frac{s_i x_{\alpha^\vee}}{2})|\tau)(s_i t_\alpha)d(s_i \alpha)+\sum_{j=1}^{n-1}(s_iy(u^j))d(s_i u_j)\\
&=d-\sum_{\beta \in \Phi^+} k(\beta, \ad(\frac{s_i x_{(s_i^{-1}\beta)^\vee}}{2})|\tau)(s_i t_{s_i^{-1}\beta})d\beta+\sum_{j=1}^{n-1}(s_iy(u^j))d(s_i u_j)\\
&=d-\sum_{\beta \in \Phi^+} k(\beta, \ad(\frac{x_{\beta^\vee}}{2})|\tau)( t_{\beta})d\beta+\sum_{j=1}^{n-1}(y((s_iu)^j))d(s_i u_j)=\nabla_{\KZB, \tau}. 
\end{align*}
This completes the proof of part (2). 
\end{proof}
\begin{example}
In the type A case, the root system is  $\Phi=\{\alpha=\epsilon_i-\epsilon_j \mid 1\leq i\neq j \leq n\}$, where $\{\epsilon_1, \dots, \epsilon_n\}$ is a standard orthonormal basis of $\C^n$.
The connection $\nabla_{\KZB, \tau}$ \eqref{conn any type} can be rewritten as $\nabla_{\KZB}:=d-\sum_{i=1}^n K_i(z|\tau)dz_i$,
where
\begin{equation}\label{eq:KZB CEE}
K_i(z|\tau)  = -y_i + \sum_{\{j|j\neq i\}} K_{ij}(z_{ij}|\tau): = -y_i + \sum_{\{j|j\neq i\}} k(z_{ij}, \ad x_i|\tau)(t_{ij}).
\end{equation}
The above follows from the relation  
\[
(\ad x_i)^k(t_{ij})=(-\ad x_j)^k(t_{ij})=(\ad \frac{x_i-x_j}{2})^k(t_{ij}),
\]
for any $k>0$, since $[x_i+x_j, t_{ij}]=0$.
This form of KZB  connection \eqref{eq:KZB CEE} is constructed in \cite{CEE}. \end{example}

\section{Flatness of the universal KZB connection}
In this section, we prove Theorem \ref{thm:connection flat} (1). 
Set
\[A:=\sum_{\alpha \in \Phi^+} k(\alpha, \ad(\frac{x_{\alpha^\vee}}{2})|\tau)(t_\alpha)d\alpha-\sum_{i=1}^{n}y(u^i)du_i.\]
Since $dA=0$, the connection $\nabla_{\KZB, \tau}$ \eqref{conn any type} is flat if and only if the curvature $\Omega:=A \wedge A$ is zero. We write the  curvature as $\Omega=\Omega_1-\Omega_2+\Omega_3$, where
\begin{align}
\Omega_1&=\frac{1}{2}\sum_{\alpha\neq \beta}\left[k(\alpha, \ad(\frac{x_{\alpha^\vee}}{2})|\tau)(t_\alpha), k(\beta, \ad(\frac{x_{\beta^\vee}}{2})|\tau)(t_\beta)\right]d\alpha\wedge d\beta.  \label{Omega1}\\
\Omega_2&=\sum_{\alpha, i}\left[k(\alpha, \ad(\frac{x_{\alpha^\vee}}{2})|\tau)(t_\alpha), y(u^i)\right]d\alpha\wedge du_i.
\label{Omega2}\\
\Omega_3&=\frac{1}{2}\sum_{i\neq j}\left[y(u^i), y(u^j)\right]du_i\wedge du_j.  \label{Omega3}
\end{align}
The defining relation $[y(u^i), y(u^j)]=0$ of $\Aell$ implies that $\Omega_3=0$.
In the rest of this section, we show that  $\Omega_1=\Omega_2$. This gives the flatness of the KZB connection $\nabla_{\KZB, \tau}$, which finishes the proof of  Theorem  \ref{thm:connection flat} (1).
\subsection{}
In this subsection, we give another expressions of the two terms $\Omega_1$ \eqref{Omega1} and $\Omega_2$ \eqref{Omega2}. 

We first introduce some notations. 
For any two nonparallel non-zero vectors $u, v \in \h^*$, we associate a new element $\omega(u^\vee, v)\in \h$
\[
\omega(u^\vee, v)=\frac{-2(v, v)}{(u, u)(v, v)-(u, v)^2}u +\frac{2(u, v)}{(u, u)(v, v)-(u, v)^2}v, 
\] where $u^\vee:=\frac{2u}{(u, u)}$. 
The vector $\omega(u^\vee, v)$ in $\h$ can be characterized by the following property. 
\begin{itemize}
\item
The vector $\omega(u^\vee, v)$ is a linear combination of vectors $u$ and $v$.
\item
One has
$
(u^\vee+ \omega(u^\vee, v)) \bot u, \text{ and $\omega(u^\vee, v) \bot v.$} 
$
\end{itemize}
The following identities follow from a direct calculation.
\begin{lemma}\label{lemma:omeg}
The following identities of  $\omega(u^\vee, v)$ hold. 
\begin{enumerate}
\item $\omega(u^\vee, v)-\omega(v^\vee, u)=\omega(u^\vee, u+v)$,  \,\ and \,\ $\omega((au+bv)^\vee, cu)=\frac{1}{b}\omega(v^\vee, u)$, for $b\neq 0$;
\item $\frac{u^\vee+\omega(u^\vee, v)}{(u^\vee, v)}=\frac{u^\vee+\omega(u^\vee, u+v)}{(u^\vee, u+v)}$, and 
$\frac{u^\vee+\omega(u^\vee, v)}{(u^\vee, v)}=\frac{\omega(v^\vee, u)}{-2}$.
\end{enumerate}
\end{lemma}
\begin{prop}\label{prop:Ome1}
Modulo the relations (tx), (xx) of $\Aell$, we have, for any $\alpha, \beta\in \Phi^+$
\[\left[k(\alpha, \ad(\frac{x_{\alpha^\vee}}{2})|\tau)(t_\alpha), k(\beta, \ad(\frac{x_{\beta^\vee}}{2})|\tau)(t_\beta)\right]=
k(\alpha, \ad(\frac{x_{\omega(\alpha^\vee, \beta)}}{-2})|\tau)k(\beta, \ad(\frac{x_{\omega(\beta^\vee, \alpha)}}{-2})|\tau)[t_\alpha, t_\beta].\]
\end{prop}
\begin{proof}
If $\alpha=\beta$, it is clear that both sides are the same as zero. Therefore, the identity holds. 

We now assume $\alpha \neq \beta$.  Note that the function $k(z, x)$ is a formal power series in $x$, with coefficient in $\C[\![z]\!]$. To show the above identity, it suffices to show the identify with $k(z, x)$ replaced by $x^n$, for any $n\in \N$.
We have
\begin{align*}
[(\ad x_{\alpha^\vee})^n(t_\alpha), (\ad x_{\beta^\vee})^m(t_\beta)]
=&[(-\ad x_{\omega(\alpha^\vee, \beta)})^n(t_\alpha), (-\ad x_{\omega(\beta^\vee, \alpha)})^m(t_\beta)]\\
=&(-\ad x_{\omega(\alpha^\vee, \beta)})^n(-\ad x_{\omega(\beta^\vee, \alpha)})^m[t_\alpha, t_\beta]. 
\end{align*}
The first equality follows from the relation $[x_{\alpha^\vee}, t_{\alpha}]=-[x_{\omega(\alpha^\vee, \beta)}, t_{\alpha}]$, since
$\alpha^{\vee}+\omega(\alpha^\vee, \beta)$ is perpendicular to $\alpha$.  The second equality follows from 
$[x_{\omega(\alpha^\vee, \beta)}, t_{\beta}]=[x_{\omega(\beta^\vee, \alpha)}, t_{\alpha}]=0$. 
This completes the proof. 
\end{proof}
\begin{corollary} \label{cor:Omega1}
The curvature term $\Omega_1$ \eqref{Omega1} is equal to
\begin{equation}
\Omega_1=\frac{1}{2}\sum_{\alpha\neq \beta}k(\alpha, \ad(\frac{x_{\omega(\alpha^\vee, \beta)}}{-2})|\tau)k(\beta, \ad(\frac{x_{\omega(\beta^\vee, \alpha)}}{-2})|\tau)[t_\alpha, t_\beta]d\alpha\wedge d\beta.
\end{equation}
\end{corollary}

\begin{prop}\label{prop:Ome2}
Suppose $u \bot \alpha$, modulo the relations (ty),(yx),(tx),(xx) of $\Aell$,
the following identity holds for any $\alpha\in\Phi^+, u\in \h^*$.
\[\left[y(u), k(\alpha, \ad(\frac{x_{\alpha^\vee}}{2})|\tau)(t_\alpha)\right]=
\sum_{\gamma \in \Phi^+}(\alpha^\vee, \gamma)(u,\gamma)
\frac{k(\alpha, \ad(\frac{x_{\alpha^\vee}}{2})|\tau)-k(\alpha, \ad(\frac{x_{\omega(\alpha^\vee, \gamma)}}{-2})|\tau)}{ad x_{\alpha^\vee} +ad x_{\omega(\alpha^\vee, \gamma)}}[t_\gamma, t_\alpha].
\]
\end{prop}
\begin{proof}
As before, note that the function $k(z, x)$ is a formal power series in $x$, with coefficient in $\C[\![z]\!]$. To show the above identity, it suffices to show the identify with $k(z, x)$ replaced by $x^n$, for any $n\in \N$.
We have
\begin{align*}
&[y(u),(\ad x_{\alpha^\vee})^n(t_\alpha)]\\
=&\sum_{s=0}^{n-1}(\ad x_{\alpha^\vee})^s(\ad[y(u), x_{\alpha^\vee}])(\ad x_{\alpha^\vee})^{n-1-s}(t_\alpha)\\
=&\sum_{\gamma\in \Phi^+}(\alpha^\vee, \gamma)(u,\gamma)\sum_{s=0}^{n-1}(\ad x_{\alpha^\vee})^s(\ad t_\gamma)(\ad x_{\alpha^\vee})^{n-1-s}(t_\alpha)  \,\ \,\  \text{by the relation (yx)}\\
=&\sum_{\gamma\in \Phi^+}(\alpha^\vee, \gamma)(u,\gamma)\sum_{s=0}^{n-1}(\ad x_{\alpha^\vee})^s(\ad t_\gamma)(-\ad x_{\omega(\alpha^\vee, \gamma)})^{n-1-s}(t_\alpha) \,\ \,\  \text{by $[x_{\alpha^\vee} , t_{\alpha}]=-[x_{\omega(\alpha^\vee, \gamma)}, t_{\alpha}]$}\\
=&\sum_{\gamma\in \Phi^+}(\alpha^\vee, \gamma)(u,\gamma)\sum_{s=0}^{n-1}(\ad x_{\alpha^\vee})^s(-\ad x_{\omega(\alpha^\vee, \gamma)})^{n-1-s}(\ad t_\gamma)(t_\alpha)  \,\ \,\  \text{by $[x_{\omega(\alpha^\vee, \gamma)}, t_{\gamma}]=0$}\\
=&\sum_{\gamma\in \Phi^+}(\alpha^\vee, \gamma)(u,\gamma)f(\ad x_{\alpha^\vee}, -\ad x_{\omega(\alpha^\vee, \gamma)})([t_\gamma, t_\alpha]), \,\ \,\ \text{where $f(u, v)=\frac{u^n-v^n}{u-v}$.}
\end{align*}
Therefore, the assertion follows.
\end{proof}

\begin{corollary}\label{cor:Omega2}
The curvature term $\Omega_2$ \eqref{Omega2} is equal to
\[\Omega_2=
\sum_{\alpha\neq\gamma\in\Phi^+}
(\alpha^\vee, \gamma)
\frac{k(\alpha, \ad(\frac{x_{\alpha^\vee}}{2})|\tau)-k(\alpha, \ad(\frac{x_{\omega(\alpha^\vee, \gamma)}}{-2})|\tau)}{ad x_{\alpha^\vee} +ad x_{\omega(\alpha^\vee, \gamma)}}[t_\alpha,t_\gamma]d\alpha\wedge d\gamma.\]
\end{corollary}
\begin{proof}
By definition,
\[\Omega_2=
\sum_\alpha\sum_i\left[k(\alpha, \ad(\frac{x_{\alpha^\vee}}{2})|\tau)(t_\alpha), y(u^i)\right]d\alpha\wedge du_i. \]
Choose the basis $\{u_i\}\subset\h^*$ to be such that $u_1=\alpha$ and $u_i\bot\alpha$ for $i\geq 2$.
Then since $d\alpha\wedge du_1=0$, by Proposition \ref{prop:Ome2}, the above is
equal to
\begin{multline*}
-\sum_{\alpha,i\geq 2}
\sum_{\gamma \in \Phi^+}(u^i,\gamma)(\alpha^\vee, \gamma)
\frac{k(\alpha, \ad(\frac{x_{\alpha^\vee}}{2})|\tau)-k(\alpha, \ad(\frac{x_{\omega(\alpha^\vee, \gamma)}}{-2})|\tau)}{ad x_{\alpha^\vee} +ad x_{\omega(\alpha^\vee, \gamma)}}[t_\gamma, t_\alpha]
d\alpha\wedge du_i\\
=\sum_{\alpha,i\geq 1}
\sum_{\gamma \in \Phi^+}(u^i,\gamma)(\alpha^\vee, \gamma)
\frac{k(\alpha, \ad(\frac{x_{\alpha^\vee}}{2})|\tau)-k(\alpha, \ad(\frac{x_{\omega(\alpha^\vee, \gamma)}}{-2})|\tau)}{ad x_{\alpha^\vee} +ad x_{\omega(\alpha^\vee, \gamma)}}[ t_\alpha, t_\gamma]
d\alpha\wedge du_i, 
\end{multline*}
which is equal to the claimed result since $\gamma=(\gamma,u^i)u_i$.
\end{proof}
\subsection{} \label{flatness,part2}
In this subsection, we use the expressions of $\Omega_1$ and $\Omega_2$ from previous subsection to compute the difference $\Omega_1-\Omega_2$.
By Corollary \ref{cor:Omega1} and Corollary \ref{cor:Omega2}, we write
\begin{align*}
&\Omega_1-\Omega_2\\
=&\sum_{\alpha\neq \beta}\left(\frac{1}{2}k(\alpha, \ad(\frac{x_{\omega(\alpha^\vee, \beta)}}{-2})|\tau)k(\beta, \ad(\frac{x_{\omega(\beta^\vee, \alpha)}}{-2})|\tau)-(\alpha^\vee, \beta)\frac{k(\alpha, \ad(\frac{x_{\alpha^\vee}}{2})|\tau)-k(\alpha, \ad(\frac{x_{\omega(\alpha^\vee, \beta)}}{-2})|\tau)}{ad x_{\alpha^\vee} +ad x_{\omega(\alpha^\vee, \beta)}}\right)
[t_\alpha, t_\beta]d\alpha\wedge d\beta.
\end{align*}
Let $k(\alpha, \beta)[t_\alpha, t_\beta]$ be the difference of the coefficient of $d\alpha\wedge d\beta$ and the coefficient of $d\beta\wedge d\alpha$ in the formula of $\Omega_1-\Omega_2$, that is
\begin{align*}
k(\alpha, \beta)=k(\alpha, \ad(\frac{x_{\omega(\alpha^\vee, \beta)}}{-2})|\tau)k(\beta, \ad(\frac{x_{\omega(\beta^\vee, \alpha)}}{-2})|\tau)
&-(\alpha^\vee, \beta)\frac{k(\alpha, \ad(\frac{x_{\alpha^\vee}}{2})|\tau)-k(\alpha, \ad(\frac{x_{\omega(\alpha^\vee, \beta)}}{-2})|\tau)}{ad x_{\alpha^\vee} +ad x_{\omega(\alpha^\vee, \beta)}}\\
&-(\beta^\vee, \alpha)\frac{k(\beta, \ad(\frac{x_{\beta^\vee}}{2})|\tau)-k(\beta, \ad(\frac{x_{\omega(\beta^\vee, \alpha)}}{-2})|\tau)}{ad x_{\beta^\vee} +ad x_{\omega(\beta^\vee, \alpha)}}.
\end{align*}

\begin{prop}\label{prop:coeff}
With notations as above, for any $\alpha, \beta\in \Phi$, we have
\begin{equation}
k(\alpha, \beta)d\alpha\wedge d\beta+k(\alpha, \alpha+\beta)d(\alpha+\beta)\wedge d\alpha+k(\beta, \alpha+\beta)d\beta \wedge d(\alpha+\beta)=0.
\end{equation}
\end{prop}
\begin{proof}
It is clear that the statement is equivalent to $-k(\alpha, \beta)+k(\alpha, \alpha+\beta)+k(\beta, \alpha+\beta)=0$. 
We now check this equality. 
By definition, we have
\begin{align}
&-k(\alpha, \beta)+k(\alpha, \alpha+\beta)+k(\beta, \alpha+\beta) \notag\\
=&-k(\alpha, \ad(\frac{x_{\omega(\alpha^\vee, \beta)}}{-2})|\tau)k(\beta, \ad(\frac{x_{\omega(\beta^\vee, \alpha)}}{-2})|\tau)+k(\alpha, \ad(\frac{x_{\omega(\alpha^\vee, \alpha+\beta)}}{-2})|\tau)k(\alpha+\beta, \ad(\frac{x_{\omega((\alpha+\beta)^\vee, \alpha)}}{-2})|\tau) \notag\\
&+k(\beta, \ad(\frac{x_{\omega(\beta^\vee, \alpha+\beta)}}{-2})|\tau)k(\alpha+\beta, \ad(\frac{x_{\omega((\alpha+\beta)^\vee, \beta)}}{-2})|\tau) \notag\\
&+(\alpha^\vee, \beta)\frac{k(\alpha, \ad(\frac{x_{\alpha^\vee}}{2})|\tau)-k(\alpha, \ad(\frac{x_{\omega(\alpha^\vee, \beta)}}{-2})|\tau)}{ad x_{\alpha^\vee} +ad x_{\omega(\alpha^\vee, \beta)}}
+(\beta^\vee, \alpha)\frac{k(\beta, \ad(\frac{x_{\beta^\vee}}{2})|\tau)-k(\beta, \ad(\frac{x_{\omega(\beta^\vee, \alpha)}}{-2})|\tau)}{ad x_{\beta^\vee} +ad x_{\omega(\beta^\vee, \alpha)}} \label{eq:K1}\\
&-(\alpha^\vee, \alpha+\beta)\frac{k(\alpha, \ad(\frac{x_{\alpha^\vee}}{2})|\tau)-k(\alpha, \ad(\frac{x_{\omega(\alpha^\vee, \alpha+\beta)}}{-2})|\tau)}{ad x_{\alpha^\vee} +ad x_{\omega(\alpha^\vee, \alpha+\beta)}}
-((\alpha+\beta)^\vee, \alpha)\frac{k(\alpha+\beta, \ad(\frac{x_{(\alpha+\beta)^\vee}}{2})|\tau)-k(\alpha+\beta, \ad(\frac{x_{\omega((\alpha+\beta)^\vee, \alpha)}}{-2})|\tau)}{ad x_{(\alpha+\beta)^\vee} +ad x_{\omega((\alpha+\beta)^\vee, \alpha)}} \label{eq:K2}\\
&-(\beta^\vee, \alpha+\beta)\frac{k(\beta, \ad(\frac{x_{\beta^\vee}}{2})|\tau)-k(\beta, \ad(\frac{x_{\omega(\beta^\vee, \alpha+\beta)}}{-2})|\tau)}{ad x_{\beta^\vee} +ad x_{\omega(\beta^\vee, \alpha+\beta)}}
-((\alpha+\beta)^\vee, \beta)\frac{k(\alpha+\beta, \ad(\frac{x_{(\alpha+\beta)^\vee}}{2})|\tau)-k(\alpha+\beta, \ad(\frac{x_{\omega((\alpha+\beta)^\vee, \beta)}}{-2})|\tau)}{ad x_{(\alpha+\beta)^\vee} +ad x_{\omega((\alpha+\beta)^\vee, \beta)}}. \label{eq:K3}
\end{align}
We use the following identities in Lemma \ref{lemma:omeg}
\begin{align*}
&\frac{\alpha^\vee+\omega(\alpha^\vee, \beta)}{(\alpha^\vee, \beta)}=\frac{\alpha^\vee+\omega(\alpha^\vee, \alpha+\beta)}{(\alpha^\vee, \alpha+\beta)},\,\  \frac{\beta^\vee+\omega(\beta^\vee, \alpha)}{(\beta^\vee, \alpha)}=\frac{\beta^\vee+\omega(\beta^\vee, \alpha+\beta)}{(\beta^\vee, \alpha+\beta)},\\
& \frac{(\alpha+\beta)^\vee+\omega((\alpha+\beta)^\vee, \alpha)}{((\alpha+\beta)^\vee, \alpha)}=-\frac{(\alpha+\beta)^\vee+\omega((\alpha+\beta)^\vee, \beta)}{((\alpha+\beta)^\vee, \beta)}.\end{align*} and the fact that $x:\h\to \Aell$ is a linear function. 
We could rewrite the terms \eqref{eq:K2}, and \eqref{eq:K3} using the identities. 
For example, we have
\[
(\alpha^\vee, \alpha+\beta)\frac{k(\alpha, \ad(\frac{x_{\alpha^\vee}}{2})|\tau)-k(\alpha, \ad(\frac{x_{\omega(\alpha^\vee, \alpha+\beta)}}{-2})|\tau)}{ad x_{\alpha^\vee} +ad x_{\omega(\alpha^\vee, \alpha+\beta)}}
=(\alpha^\vee, \beta)\frac{k(\alpha, \ad(\frac{x_{\alpha^\vee}}{2})|\tau)-k(\alpha, \ad(\frac{x_{\omega(\alpha^\vee, \alpha+\beta)}}{-2})|\tau)}{ad x_{\alpha^\vee} +ad x_{\omega(\alpha^\vee, \beta)}}. \]
After cancellations with \eqref{eq:K1}, we have
\begin{align*}
&\eqref{eq:K1}+\eqref{eq:K2}+\eqref{eq:K3}\\
=&(\alpha^\vee, \beta)\frac{k(\alpha, \ad(\frac{x_{\omega(\alpha^\vee, \alpha+\beta)}}{-2})|\tau)-k(\alpha, \ad(\frac{x_{\omega(\alpha^\vee, \beta)}}{-2})|\tau)}{ad x_{\alpha^\vee} +ad x_{\omega(\alpha^\vee, \beta)}}
+(\beta^\vee, \alpha)\frac{k(\beta, \ad(\frac{x_{\omega(\beta^\vee, \alpha+\beta)}}{-2})|\tau)-k(\beta, \ad(\frac{x_{\omega(\beta^\vee, \alpha)}}{-2})|\tau)}{ad x_{\beta^\vee} +ad x_{\omega(\beta^\vee, \alpha)}}\\
&-((\alpha+\beta)^\vee, \alpha)\frac{k(\alpha+\beta, \ad(\frac{x_{\omega((\alpha+\beta)^\vee, \beta)}}{-2})|\tau)-k(\alpha+\beta, \ad(\frac{x_{\omega((\alpha+\beta)^\vee, \alpha)}}{-2})|\tau)}{ad x_{(\alpha+\beta)^\vee} +ad x_{\omega((\alpha+\beta)^\vee, \alpha)}}. 
\end{align*}
For simplicity, we make a change of variable. Set $u:=\frac{\omega(\alpha^\vee, \beta)}{-2}$ and $v:=\frac{\omega(\beta^\vee, \alpha+\beta)}{-2}$.
Using again the identities from Lemma \ref{lemma:omeg}, we get
\begin{align*}
&u=\frac{\omega(\alpha^\vee, \beta)}{-2}=\frac{\beta^\vee+\omega(\beta^\vee, \alpha)}{(\beta^\vee, \alpha)}=\frac{\omega((\alpha+\beta)^\vee, \beta)}{-2}, \\
&  v=\frac{\omega(\beta^\vee, \alpha+\beta)}{-2}=-\frac{\omega(\alpha^\vee, \alpha+\beta)}{-2}=-\frac{(\alpha+\beta)^\vee+\omega((\alpha+\beta)^\vee, \alpha)}{((\alpha+\beta)^\vee, \alpha)}, \\
&u+v=\frac{\omega(\beta^\vee, \alpha)}{-2}=\frac{\alpha^\vee+\omega(\alpha^\vee, \beta)}{(\alpha^\vee, \beta)}=\frac{\omega((\alpha+\beta)^\vee, \alpha)}{-2}.
\end{align*}
Plugging $u, v$ and $u+v$ into $-k(\alpha, \beta)+k(\alpha, \alpha+\beta)+k(\beta, \alpha+\beta)$, we have
\begin{align*}
&-k(\alpha, \beta)+k(\alpha, \alpha+\beta)+k(\beta, \alpha+\beta)\\
=&-k(\alpha, u|\tau)k(\beta, u+v|\tau)+k(\alpha, -v|\tau)k(\alpha+\beta, u+v|\tau)+k(\beta, v|\tau)k(\alpha+\beta, u|\tau)\\
&+\frac{k(\alpha, -v|\tau)-k(\alpha, u|\tau)}{u+v}
+\frac{k(\beta, v|\tau)-k(\beta, u+v|\tau)}{u}
-\frac{k(\alpha+\beta, u|\tau)-k(\alpha+\beta, u+v|\tau)}{-v}. 
\end{align*}
Using the following well-known theta function identity (see also \cite{CEE})
\begin{align*}
&k(z, -v)k(z', u+v)-k(z, u)k(z'-z, u+v)+k(z', u)k(z'-z, v)\\
&+\frac{k(z'-z, v)-k(z'-z, u+v)}{u}+\frac{k(z', u)-k(z', u+v)}{v}+\frac{k(z, u)-k(z, -v)}{u+v}=0, 
\end{align*}
we conclude the claim. 
\end{proof}

\begin{lemma}
Let $\Phi$ be a root system. For any two roots $\alpha$, $\beta \in \Phi$ such that $\alpha \neq \pm \beta$, then
$\Psi_{\alpha, \beta}:=(\Z\alpha+\Z\beta)\cap \Phi$ is a root subsystem of rank 2.
\end{lemma}
\begin{proof}
It is straightforward to check all the axioms of root system for $\Psi_{\alpha, \beta}$.
\end{proof}
For a root subsystem $\Psi\subset \Phi$, write
\begin{align*}
&(\Omega_1-\Omega_2)^{\Psi}\\
=&\sum_{\{\alpha, \beta\mid \alpha,\beta \in \Psi^+\}}\left(\frac{1}{2}k(\alpha, \ad(\frac{x_{\omega(\alpha^\vee, \beta)}}{-2})|\tau)k(\beta, \ad(\frac{x_{\omega(\beta^\vee, \alpha)}}{-2})|\tau)-(\alpha^\vee, \beta)\frac{k(\alpha, \ad(\frac{x_{\alpha^\vee}}{2})|\tau)-k(\alpha, \ad(\frac{x_{\omega(\alpha^\vee, \beta)}}{-2})|\tau)}{ad x_{\alpha^\vee} +ad x_{\omega(\alpha^\vee, \beta)}}\right)
[t_\alpha, t_\beta]d\alpha\wedge d\beta.\end{align*}
\begin{lemma} \cite[Proposition 2.19]{TL1}\label{reduction}
Assume $(\Omega_1-\Omega_2)^{\Psi_{\alpha, \beta}}=0$, for any $\alpha, \beta \in \Phi$ such that $\alpha \neq \pm \beta$. 
Then, $(\Omega_1-\Omega_2)^{\Phi}=0$. 
\end{lemma}
\begin{proof}
There exists a collection of rank 2 subsystems $\Psi_1, \dots, \Psi_m$, such that
$\Phi=\sqcup_{i=1}^m \Psi_i. $
Therefore, $(\Omega_1-\Omega_2)^{\Phi}=\displaystyle{\sum_{i=1}^m\sum_{\alpha, \beta\in \Psi_i^+}(\Omega_1-\Omega_2)^{\Psi_i}}$. This completes the proof.
\end{proof}
Thanks to Lemma \ref{reduction}, it suffices to show 
$(\Omega_1-\Omega_2)^{\Psi}=0$, for any  rank 2 root system $\Psi$. In the rest of this section, we check the vanishing of $(\Omega_1-\Omega_2)^{\Psi}$ case by case. The tool is Proposition \ref{prop:coeff}. 
\subsection{Rank 2 calculations}
\subsubsection{Case $A_1\times A_1$}
Let $\Phi=\{\pm\alpha, \pm\beta\}$ be the root system of type $A_1\times A_1$, then $[t_\alpha, t_\beta]=0$. 
Therefore, $(\Omega_1-\Omega_2)^{\Phi}=0$ in this case. 
\subsubsection{Case $A_2$}
Let $\epsilon_1, \epsilon_2, \epsilon_3$ be the standard orthonormal basis of $\R^3$.
Let $\alpha_1=\epsilon_1-\epsilon_2, \alpha_2=\epsilon_2-\epsilon_3$ be the simple roots for $A_2$. Then, we have $\Phi^+=\{\alpha_1, \alpha_2, \alpha_1+\alpha_2\}$.
By definition, 
\begin{align}
\Omega_1-\Omega_2 & = k(\alpha_1, \alpha_2)[t_{\alpha_1}, t_{\alpha_2}]d\alpha_1\wedge d\alpha_2+k(\alpha_1, {\alpha_1+\alpha_2})[t_{\alpha_1}, t_{\alpha_1+\alpha_2}]d\alpha_1\wedge d(\alpha_1+\alpha_2) \notag\\
&+k(\alpha_2, {\alpha_1+\alpha_2})[t_{\alpha_2}, t_{\alpha_1+\alpha_2}]d\alpha_2\wedge d(\alpha_1+\alpha_2).  \label{OmegaA2}
\end{align}
By Proposition \ref{prop:coeff}, we have
\[k(\alpha_1, \alpha_1+\alpha_2)d\alpha_1\wedge d(\alpha_1+\alpha_2)=k(\alpha_1, \alpha_2)d\alpha_1\wedge d\alpha_2+k(\alpha_2, \alpha_1+\alpha_2)d\alpha_2\wedge d(\alpha_1+\alpha_2).\]
Plugging the above expression into \eqref{OmegaA2}, we have
\begin{align*}
\Omega_1-\Omega_2 &=
k(\alpha_1, \alpha_2)[t_{\alpha_1}, t_{\alpha_2}+t_{\alpha_1+\alpha_2}]d\alpha_1\wedge d\alpha_2+k(\alpha_2, {\alpha_1+\alpha_2})[t_{\alpha_1}+t_{\alpha_2}, t_{\alpha_1+\alpha_2}]d\alpha_2\wedge d(\alpha_1+\alpha_2).
\end{align*}
Using (tt)-relations $[t_{\alpha_1}, t_{\alpha_2}+t_{\alpha_1+\alpha_2}]=[t_{\alpha_1}+t_{\alpha_2}, t_{\alpha_1+\alpha_2}]=0$, we conclude $\Omega_1-\Omega_2=0$ in the case of $A_2$.
The above procedure can be represented as the graph
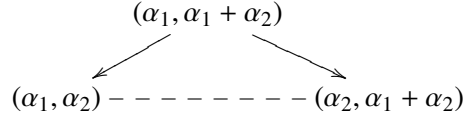
\begin{equation*}
\xymatrix@R=1.5em @C=0.5em{ & (\alpha_1, \alpha_1+\alpha_2)\ar[ld]\ar[rd] &\\
(\alpha_1, \alpha_2)\ar@{--}[rr]&&(\alpha_2, \alpha_1+\alpha_2)
}
\end{equation*}
\subsubsection{Case $B_2$}
Let $\epsilon_1, \epsilon_2$ be the standard orthonormal basis of $\R^2$. Let $\alpha_1=\epsilon_1-\epsilon_2, \alpha_2=\epsilon_2$ be the simple roots for $B_2$.
Then, we have $\Phi^+=\{\alpha_1, \alpha_1+\alpha_2, \alpha_1+2\alpha_2,\alpha_2\}$.
By definition, 
\begin{align*}
&\Omega_1-\Omega_2  \\
= &k(\alpha_1, \alpha_2)[t_{\alpha_1}, t_{\alpha_2}]d\alpha_1\wedge d\alpha_2+k(\alpha_1, z_{\alpha_1+\alpha_2})[t_{\alpha_1}, t_{\alpha_1+\alpha_2}]d\alpha_1\wedge d(\alpha_1+\alpha_2)\\
+& k(z_{\alpha_1+\alpha_2}, \alpha_2)[t_{\alpha_1+\alpha_2}, t_{\alpha_2}]d(\alpha_1+\alpha_2)\wedge d\alpha_2+ k(\alpha_1, \alpha_1+2\alpha_2)[t_{\alpha_1}, t_{\alpha_1+2\alpha_2}]d\alpha_1\wedge d(\alpha_1+2\alpha_2)\\
 +& k(z_{\alpha_1+\alpha_2}, \alpha_1+2\alpha_2)[t_{\alpha_1+\alpha_2}, t_{\alpha_1+2\alpha_2}]d(\alpha_1+\alpha_2)\wedge d(\alpha_1+2\alpha_2)+ k(\alpha_1+2\alpha_2, \alpha_2)[t_{\alpha_1+2\alpha_2}, t_{\alpha_2}] d(\alpha_1+2\alpha_2)\wedge d\alpha_2.
\end{align*}
We use Proposition \ref{prop:coeff} and split $k(\alpha, \beta)d \alpha \wedge d\beta$ according to the following graph
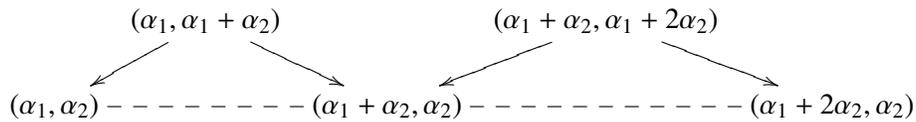
\[
\xymatrix@R=1.5em @C=0.5em{
&(\alpha_1, \alpha_1+\alpha_2) \ar[ld]\ar[rd]&& (\alpha_1+\alpha_2, \alpha_1+2\alpha_2)\ar[ld]\ar[rd]&\\
(\alpha_1, \alpha_2)\ar@{--}[rr]&&(\alpha_1+\alpha_2, \alpha_2)\ar@{--}[rr]&& (\alpha_1+2\alpha_2, \alpha_2)
}
\]
After simplification, we have
\begin{align*}
\Omega_1-\Omega_2
 =&k(\alpha_1, \alpha_2)[t_{\alpha_1}, t_{\alpha_2}+t_{\alpha_1+\alpha_2}]d\alpha_1\wedge d\alpha_2 +k(\alpha_2, \alpha_1+\alpha_2)[t_{\alpha_1}+t_{\alpha_2}+t_{\alpha_1+2\alpha_2}, t_{\alpha_1+\alpha_2}]d\alpha_2\wedge d(\alpha_1+\alpha_2)\\
& +k(\alpha_2, \alpha_1+2\alpha_2)[t_{\alpha_2}+t_{\alpha_1+\alpha_2}, t_{\alpha_1+2\alpha_2}]d\alpha_2\wedge d(\alpha_1+2\alpha_2). 
\end{align*}
Using the (tt) relations $[t_\alpha, \sum_{\beta\in \Psi^+}t_\beta]$ for $B_2$ root system, 
each summand of the above formula equals to zero. Therefore, we conclude
$\Omega_1-\Omega_2=0$ in the case of $B_2$.
\subsubsection{Case $G_2$}
Let $\epsilon_1, \epsilon_2, \epsilon_3$ be the standard orthonormal basis of $\R^3$. 
The simple roots of $G_2$ root system are $\alpha_1=2\epsilon_2-\epsilon_1-\epsilon_3, \alpha_2=\epsilon_1-\epsilon_2$ and we have $\Phi^+=\{\alpha_1, \alpha_2, \alpha_1+\alpha_2, \alpha_1+2\alpha_2,\alpha_1+3\alpha_2, 2\alpha_1+3\alpha_2\}$. We use Proposition \ref{prop:coeff} and split $k(\alpha, \beta)d\alpha \wedge d\beta$ according to the following graph
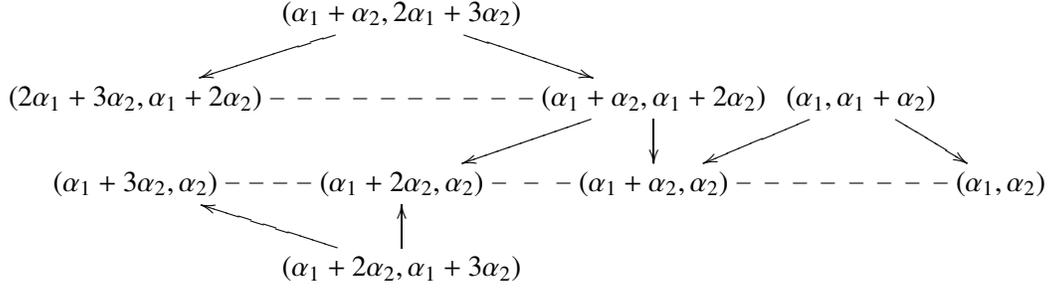
\[
\xymatrix@R=1.5em @C=0.1em{
&(\alpha_1+\alpha_2, 2\alpha_1+3\alpha_2) \ar[ld]\ar[rd]&&&\\
 (2\alpha_1+3\alpha_2, \alpha_1+2\alpha_2)  \ar@{--}[rr]&
& (\alpha_1+\alpha_2, \alpha_1+2\alpha_2) \ar[ld]\ar[d]&
  (\alpha_1, \alpha_1+\alpha_2) \ar[ld]\ar[dr]&\\
 (\alpha_1+3\alpha_2, \alpha_2) \ar@{--}[r]&
(\alpha_1+2\alpha_2, \alpha_2) \ar@{--}[r]&
(\alpha_1+\alpha_2, \alpha_2) \ar@{--}[rr]&&
(\alpha_1, \alpha_2)\\
&(\alpha_1+2\alpha_2, \alpha_1+3\alpha_2) \ar[lu]\ar[u]&&&}
\]
After simplification, we have
\begin{align*}
\Omega_1-\Omega_2=&
k(\alpha_1,\alpha_2)[t_{\alpha_1},t_{\alpha_2}+t_{\alpha_1+\alpha_2}]
d\alpha_1\wedge d\alpha_2
+k(\alpha_2, \alpha_1+\alpha_2)[t_{\alpha_1}+t_{\alpha_2}, t_{\alpha_1+\alpha_2}]d{\alpha_2}\wedge d(\alpha_1+\alpha_2)\\
&+k(\alpha_1+2\alpha_2, \alpha_2)([t_{\alpha_1+2\alpha_2}, t_{\alpha_2}+t_{\alpha_1+\alpha_2}+t_{\alpha_1+3\alpha_2}]-[t_{\alpha_1+\alpha_2}, t_{2\alpha_1+3\alpha_2}])d(\alpha_1+2\alpha_2)\wedge d\alpha_2\\
&+k(2\alpha_1+3\alpha_2, \alpha_1+2\alpha_2)[t_{2\alpha_1+3\alpha_2}, t_{\alpha_1+2\alpha_2}+t_{\alpha_1+\alpha_2}]d(2\alpha_1+3\alpha_2)\wedge d(\alpha_1+2\alpha_2)\\
&+k(\alpha_1+3\alpha_2, \alpha_2)[t_{\alpha_1+3\alpha_2}, t_{\alpha_2}+t_{\alpha_1+2\alpha_2}]d(\alpha_1+3\alpha_2)\wedge d\alpha_2
\end{align*}
Using the (tt) relations $[t_\alpha, \sum_{\beta\in \Psi^+}t_\beta]$ for $G_2$ root system, 
each summand of the above formula equals to zero, where the second line follows from the equality
\[
[t_{\alpha_1+2\alpha_2}, t_{\alpha_2}+t_{\alpha_1+\alpha_2}+t_{\alpha_1+3\alpha_2}]-[t_{\alpha_1+\alpha_2}, t_{2\alpha_1+3\alpha_2}]
=[t_{\alpha_1+2\alpha_2}, t_{\alpha_2}+t_{\alpha_1+\alpha_2}+t_{\alpha_1+3\alpha_2}+t_{2\alpha_1+3\alpha_2}]
-[t_{\alpha_1+\alpha_2}+t_{\alpha_1+2\alpha_2}, t_{2\alpha_1+3\alpha_2}].
\]
Therefore, we conclude
$\Omega_1-\Omega_2=0$ in the case of $G_2$.
\section{Degeneration of the elliptic connection}\label{sec:degen}
In this section, we show that as the imaginary part of $\tau$ tends 
$\infty$, the connection
$\nabla_{\KZB, \tau}$ \eqref{conn any type} degenerates to a trigonometric connection of the form considered in \cite{TL1}. This gives a map from the trigonometric Lie algebra $A_{\trig}$ to the completion of $\Aell$. 

\subsection{The trigonometric connection}
In \cite{TL1}, Toledano Laredo introduced the trigonometric connection, which we recall it here.
Let $H=\Hom_\Z(Q,\C^*)$ be the complex algebraic torus with Lie
algebra $\h$ and coordinate ring given by the group algebra $\C Q$.
We denote the function corresponding to $\lambda\in Q$ by $e^
\lambda\in\C[H]$, and set
\begin{equation}\label{eq:Hreg}
H\rreg=H\setminus\bigcup_{\alpha\in\Phi}\{e^\alpha=1\}
\end{equation}

Let $A_{\trig}$ be an algebra endowed with the following data:
\begin{itemize}
\item a set of elements $\{t_\alpha\}_{\alpha\in\Phi}\subset A_{\trig}$ such
that $t_{-\alpha}=t_{\alpha}$, 
\item a linear map $X:\h\to A_{\trig}$. 
\end{itemize}
Consider the $A_{\trig}$--valued connection on $H\rreg$ given by
\begin{equation}\label{eq:trigo}
\nabla_{\trig}=d-
\sum_{\alpha\in\Phi_+}\etaa{\alpha}\,t_{\alpha}-
du_i\,X(u^i), 
\end{equation}
where $\Phi_+\subset\Phi$ is a chosen system of positive roots,
$\{u_i\}$ and $\{u^i\}$ are dual bases of $\h^*$ and $\h$ respectively,
the differentials $du_i$ are regarded as translation invariant
one forms on $H$ and the summation over $i$ is implicit.
\Omit{
In \cite{TL1}, there is another form of the trigonometric connection
\[
\nabla_{\trig}=d-\frac{1}{2} \sum_{\alpha\in \Phi^+} \frac{e^{\alpha}+1}{e^{\alpha}-1} t_{\alpha} d\alpha -\delta(u^i) du_i,
\]
where $\delta(v)=X(v)-\frac{1}{2} \sum_{\alpha\in \Phi^+} \alpha(v) t_{\alpha}.$
}

\begin{theorem}[\cite{TL1}]
\label{thm:trio tt}
  The connection \eqref{eq:trigo} is flat, if and only if the following relations hold
\begin{description}
  \item[$(tt)$] For any rank 2 root subsystem $\Psi\subset \Phi$, and $\alpha\in \Psi$,
  $[t_\alpha, \sum_{\beta\in \Psi_+}t_\beta]=0.$
  \item[$(XX)$] For any $u, v\in \h$,
  $[X(u), X(v)]=0.$
  \item[($tX$)] For any $\alpha\in \Phi_+$, $w\in W$ such that $w^{-1}\alpha$ is a simple root and $u\in \h$, such that $\alpha(u)=0$,
      \[[t_\alpha, X_w(u)]=0,\]
      where $ X_w(u)=X(u)-\sum_{\beta\in \Phi_+\cap w \Phi_-}\beta(u)t_\beta$.
\end{description}
\Omit{
 \item[(2)] Modulo the relations $(tt)$, the relations $(tX)$ are equivalent to
 \begin{description}
   \item[($t\delta$)] $[t_\alpha, \delta(v)]=0, $ for any $\alpha\in \Phi$ and $v\in \h$ such that $\alpha(v)=0$.
 \end{description}
 \end{description}}
\end{theorem}
Assume now that the algebra $A_{\trig}$ is acted
upon by the Weyl group $W$ of $\Phi$.
\begin{prop}[\cite{TL1}]
\label{prop:equ} The connection $\nabla_{\trig}$ is $W-$ equivariant if and only if
for any $\alpha\in \Phi$, simple reflection $s_i \in W$ and $x \in \h$,
\begin{align}
&s_i(t_\alpha)=t_{s_i(\alpha)}, \label{st}\\
&s_i(X(x))-X(s_ix)=(\alpha_i, x)t_{\alpha_i}. \label{stau}
\end{align}
\Omit{
  \item[(2)] Modulo (\ref{st}), the relation (\ref{stau}) is equivalent to $W-$ equivariance
of the linear map $\delta: \h\to A_{\trig}$.
\end{description}}
 \end{prop}
\subsection{The degeneration of elliptic configuration space}
In this subsection, we show as $\Im \tau\to \infty$. The elliptic configuration space $T_{\reg}$ degenerates to $H\rreg$. We start with analyzing the elliptic curve $\mathcal{E}_\tau$ as $\Im \tau\to \infty$. 
Let $\wp(z)$ be the Weierstrass function with respect to the lattice $\Z+\tau\Z$,
\[
\wp(z)=\frac{1}{z^2}+\sum_{m, n}\left(\frac{1}{(z-m\tau +n)^2}-\frac{1}{(m\tau+n)^2}\right).
\]
Differentiating $\wp(z)$ term by term, we get $\wp'(z)$. For $q=e^{2\pi i \tau}$, the points $(\wp(z), \wp'(z))$ lie on the curve $E$ defined by the equation
\[
y^2=4x^3-g_2 x-g_3,
\]
where 
\begin{align*}
&g_2=\frac{(2\pi i)^4}{12}\left(1+240\sum_{n=1}^\infty\frac{n^3q^n}{(1-q^n)}\right), \,\
g_3=\frac{(2\pi i)^6}{6^3}\left(-1+504\sum_{n=1}^\infty\frac{n^5q^n}{(1-q^n)}\right).
\end{align*}
The cubic polynomial $4x^3-g_2 x-g_3$ has a discriminant given by
$\Delta=g_2^3 -27 g_3^2.$

The map \[\mathcal{E}_\tau=\C/(\Z+\tau \Z)\to E\subset \mathbb{P}(C), \,\ z\mapsto (1, \wp(z), \wp'(z))\]
is an isomorphism of complex Lie groups.
As $\Im\tau\rightarrow+\infty$, we have $q\to 0$.
The elliptic curve $E$ degenerates to $\tilde{E}$ whose defining equation is
\[y^2=4x^3-\frac{(2\pi i)^4}{12} x+\frac{(2\pi i)^6}{6^3}=\frac{-4}{27}(2\pi^2-3x)(3x+\pi^2)^2,\]
with discriminant $\Delta=0$.
Therefore, $\tilde{E}$ has a singular point at $x_0=(x, y)=(\frac{2\pi^2}{3}, 0)$.

Removing the singular point $x_0$ of $\widetilde{\mathcal{E}}_\tau$, topologically, the open subset
$\widetilde{\mathcal{E}}_\tau\setminus x_0$ is homeomorphic to the complex torus $\C^*$.
Thus, we have a continues map on topological spaces $\phi: \C^*\otimes P^\vee\to \widetilde{\mathcal{E}}_\tau\otimes P^\vee$.

The pullback the bundle $\mathcal{P}_{\tau, n}$ under the map $\phi$ is a trivial principal bundle on $\C^*\otimes P^\vee$ with structure group $\exp(\widehat{\Aell})$. The section of this trivial bundle can be described as
\[
\{f(z): \pi^{-1}(U)\to \exp(\widehat{\Aell})\mid f(z+1)=f(z)\},
\] where $U\subset \C^*\otimes P^\vee$ and $\pi=\exp: \C\to \C^*$ be the natural exponential map.
\subsection{The degeneration of KZB connection}
We describe the degeneration of the connection $\nabla_{\KZB, \tau}$ \eqref{conn any type} as $\Im \tau \to +\infty$ in this subsection.
\Omit{Let  $B_n$ be the \textit{Bernoulli numbers}, which are given by the power series expansion
\[
\frac{z}{e^z-1}=\sum_{n=0}^\infty \frac{B_n}{n!} z^n.
\]
Therefore, we have
\[
\frac{2i\pi z}{e^{2\pi i z}-1}+i\pi z=
\sum_{n=0}^{\infty} \frac{(2\pi i)^{2n}}{(2n)!}B_{2n}z^{2n}=1+\sum_{n=1}^{\infty} \frac{(2\pi i)^{2n}}{(2n)!}B_{2n}z^{2n}.
\]
The above equality is obtained by some special values of the Bernoulli numbers\[B_0=1, B_1=-\frac{1}{2}, \,\ \text{and $B_{2n+1}=0$, for $n\geq 1$}.  \]}
As $\Im\tau\rightarrow+\infty$, using the product formula \eqref{prod formula} of theta function, the theta function $\theta(z|\tau)$ tends to
\begin{equation}\label{eq:theta sine}
\theta(z| \tau)
\longrightarrow \frac{u^{\frac{1}{2}}(1-u^{-1})}{2\pi i}
=\frac{e^{\pi iz}-e^{-\pi iz}}{2\pi i}.
\end{equation}
Thus, by \eqref{eq:theta sine}, as $q\to 0$, we have $k(\alpha, \ad(\frac{x_{\alpha^\vee}}{2})|\tau)$ tends to
\begin{align*}
k(\alpha, \ad(\frac{x_{\alpha^\vee}}{2})|\tau)
 \longrightarrow &
 2\pi i \frac{e^{\pi i (\alpha+\ad(\frac{x_{\alpha^\vee}}{2}))}-e^{-\pi i (\alpha+\ad(\frac{x_{\alpha^\vee}}{2}))} }{(e^{\pi i \alpha}-e^{-\pi i \alpha})\Big(e^{\pi i \ad(\frac{x_{\alpha^\vee}}{2}) }-e^{-\pi i \ad(\frac{x_{\alpha^\vee}}{2})}\Big)}
 -\frac{1}{\ad(\frac{x_{\alpha^\vee}}{2})}
 \\
 =&2\pi i \left(\frac{1}{e^{2\pi i \alpha}-1}+\frac{e^{2\pi i \ad(\frac{x_{\alpha^\vee}}{2})}}{e^{2\pi i \ad(\frac{x_{\alpha^\vee}}{2})}-1 }\right) -\frac{1}{\ad(\frac{x_{\alpha^\vee}}{2})}. 
\end{align*} 
Therefore, as $\Im\tau\to+\infty$, the connection $\nabla_{\KZB}$ \eqref{conn any type} degenerates to a flat connection $\nabla^{\deg}$ on the space
\[(\C^*\otimes P^\vee)\setminus \{\cup_{\alpha\in \Phi^+}(e^{2\pi i \alpha}-1)\}.\]
The flat connection $\nabla^{\deg}$ is defined on the trivial principal bundle with fibers isomorphic to $\exp(\widehat{\Aell})$.
\begin{prop}\label{prop:degeneration}
The flat connection $\nabla^{\deg}$ on $H_{\reg}$ takes the following form
\begin{align*}
\nabla^{\deg}=&d-\sum_{\alpha \in \Phi^+} 
 \left(\frac{2\pi i}{e^{2\pi i \alpha}-1}+\frac{2\pi i e^{2\pi i \ad(\frac{x_{\alpha^\vee}}{2})}}{e^{2\pi i \ad(\frac{x_{\alpha^\vee}}{2})}-1 }  -\frac{1}{\ad(\frac{x_{\alpha^\vee}}{2})} \right) t_\alpha d\alpha+\sum_{i=1}^{n}y(u^i)du_i\\
=&d-\sum_{\alpha \in \Phi^+} \frac{2\pi i t_{\alpha}}{e^{2\pi i \alpha}-1} d \alpha
+\sum_{i=1}^{n} \left(y(u^i) -\sum_{\alpha \in \Phi^+}(\alpha, u^i) 
\Big(\frac{2\pi i e^{2\pi i \ad(\frac{x_{\alpha^\vee}}{2})}}{e^{2\pi i \ad(\frac{x_{\alpha^\vee}}{2})}-1 }  
-\frac{1}{\ad(\frac{x_{\alpha^\vee}}{2})}\Big) t_\alpha \right) du_i. 
\end{align*}
\end{prop}
We modify the trigonometric connection \eqref{eq:trigo} slightly by changing of variables.  On the torus $(\C^*\otimes P^\vee)\setminus\bigcup_{\alpha\in\Phi}\{e^{2\pi i\alpha}=1\}$, we consider an $A_{\trig}$--valued flat trigonometric connection
\begin{equation}\label{eq:trigo2}
\nabla_{\trig}=d-
\sum_{\alpha\in\Phi_+}\frac{2\pi i d\alpha}{e^{2\pi i \alpha}-1}t_{\alpha}-
du_i\,X(u^i).
\end{equation}
By universality of the trigonometric Lie algebra $A_{\trig}$, Proposition \ref
{prop:degeneration} gives rise to a map $A_{\trig}\to \widehat{\Aell}$. The map is given by
\begin{align*}
t_\alpha	\mapsto t_\alpha, \,\ \,\ 
X(u)	&\mapsto -y(u) +\sum_{\alpha \in \Phi^+}(\alpha, u) 
\Big(\frac{2\pi i e^{2\pi i \ad(\frac{x_{\alpha^\vee}}{2})}}{e^{2\pi i \ad(\frac{x_{\alpha^\vee}}{2})}-1 }  -\frac{1}{\ad(\frac{x_{\alpha^\vee}}{2})}\Big)t_\alpha.
\end{align*}
Note that this map does not preserve the gradings of $A_{\trig}$ and $\Aell$, but only
the corresponding descending filtrations.

\section{The Malcev Lie algebra of the pure elliptic braid group}


The following definition can be found in \cite[Definition 2.4]{DPS}. 
\begin{definition}
A finitely presented group $\Gamma$ is \textit{1--formal} if its Malcev Lie algebra
$\mm_\Gamma$ is isomorphic to the completion of its holonomy Lie algebra as
filtered Lie algebras. Equivalently, $\Gamma$ is 1--formal if, and only it $\mm_
\Gamma$ is isomorphic, as a filtered Lie algebra, to the graded completion of a
quadratic Lie algebra.
\end{definition}

\subsection{} 

Let $\Pell^\Phi$ be the pure elliptic braid group $\Pell^\Phi:=\pi_1(T_{\reg}, x_o)$. 
Then, we have the short exact sequence of groups (see \cite[Example 2.20]{ALR})
\begin{equation*}
1 \to \Pell^\Phi \to \Bell \to W\to 1,
\end{equation*} where $\Bell:=\pi_1^{\orb}(T_{\reg}/W, x_o)$ is the elliptic braid group and $W$ is the Weyl group. 
The flatness of the universal KZB connection $\nabla_{\KZB, \tau}$ \eqref{conn any type} gives rise to the monodromy map
\[\mu: \Pell^\Phi \to \exp\widehat{\Aell},\]
where $\widehat{\Aell}$ is the completion of $\Aell$ with respect to the grading $\deg(x(u))=\deg(y(u))=1$ and $\deg(t_\alpha)=2$.

Let $J \subseteq \C \Pell^\Phi$ be the augmentation ideal of group ring $\C \Pell^\Phi$, and let
$\widehat{\C \Pell^\Phi}$ be the completion with respect to the augmentation ideal $J$.
We denote by $U(\Aell)$ the universal enveloping algebra of the Lie algebra $\Aell$. Then, the monodromy map $\mu$ induces the following map on the completions
\[\hat{\mu}: \widehat{\C \Pell^\Phi} \to \widehat{U(\Aell)}.\]

\begin{theorem}\label{thm:formal}
The induced map $\hat{\mu}: \widehat{\C \Pell^\Phi} \to \widehat{U(\Aell)}$ is an isomorphism of Hopf algebras.
\end{theorem}
Taking the primitive elements of the Hopf algebras $\widehat{\C \Pell^\Phi}$ and $\widehat{U(\Aell)}$, we get an isomorphism of the Malcev Lie algebra of $\Pell^\Phi$ and the Lie algebra $\Aell$ defined in Definition \ref{def:holonomy Aell}.

\begin{corollary}\label{cor:ness}
The defining relations of the Lie algebra $\Aell$ are necessary to make connection of the form \eqref{conn any type} flat.
\end{corollary}
\begin{proof}
Assume $\nabla$ is a flat connection defined on a principal bundle with structure group $\exp(\widehat{A})$ with the form \eqref{conn any type}. The monodromy of $\nabla$ induces a map $\widehat{\C \Pell^\Phi} \to \widehat{U(A)}$.  By Theorem \ref{thm:formal}, $\widehat{\C \Pell^\Phi}\cong  \widehat{U(\Aell)}$. 
Therefore, we have a well-defined map $\Aell\to A$. This implies the claim. 
\end{proof}

The rest of the section is devoted to prove Theorem \ref{thm:formal}.
Note that both $\widehat{\C \Pell^\Phi}$ and $\widehat{\Aell}$ are $\N-$filtered and $\hat \mu$ preserves the grading. It suffices to show the associated graded
\[\gr(\hat{\mu}): \gr(\widehat{\C \Pell^\Phi}) \to \gr(\widehat{U(\Aell)})=\widehat{U(\Aell)}\] is an isomorphism.

We first describe the generators of the pure elliptic braid group $ \Pell^\Phi$. 
The following proposition can be found in \cite[Proposition A.1]{BMR}.
\begin{prop}
Let $i$ be the injection of an irreducible divisor $D$ in a smooth connected complex variety $Y$ and base point $x_0\in Y-D$.
Then, the kernel of the morphism
\[
\pi_1(i): \pi_1(Y-D, x_0)\to \pi_1(Y, x_0)
\]
is generated by all the generators of the monodromy around $D$.
\end{prop}
Using induction on the number of irreducible divisors, we have the following. 
\begin{corollary}
Suppose $X=Y-\bigcup_{i=1}^n D_i$, where $D_i$ are irreducible divisors of $Y$. Then, the fundamental group $\pi_1(X, x_0)$ is generated
by all the generators of the monodromy around the divisors $D_i$ and generators of $\pi_1(Y, x_0)$.
\end{corollary}
As a consequence, the generators of $\Pell^\Phi$ can be chosen as $T_\alpha$, $X_1, \dots, X_n$, $Y_1, \dots, Y_n$, where $T_{\alpha}$ is the path around the divisor $H_{\alpha}$, for $\alpha\in \Phi^+$ and $\{X_i, Y_i\}$ are the standard generators of the torus $E\otimes P^\vee \cong E^n$. 
The presentation of $\widehat{\C \Pell^\Phi}$ can be described as follows.
Suppose $\Pell^\Phi$ is presented by generators $g_1, \dots, g_n$ and relations $R_i(g_1, \dots, g_n)$, $i=1, \dots, p$.
Then $\widehat{\C \Pell^\Phi}$ is the quotient of the free Lie algebra generated by $\gamma_1, \dots, \gamma_n$ by the ideal generated by
$\log(R_i(e^{\gamma_1}, \dots, e^{\gamma_n}))$, $i=1, \dots, p$.
We will abuse notation and denote the generators of $\widehat{\C \Pell^\Phi}$ by 
\[
T_\alpha, X_1, \dots, X_n, Y_1, \dots, Y_n\,\  \text{for $\alpha\in \Phi^+$}.
\]
We construct a map \[ p: \widehat{U(\Aell)} \to \gr(\widehat{\C \Pell^\Phi}),\] by sending
$t_\alpha\mapsto T_\alpha$, $x(\lambda_i^\vee) \mapsto X_i$, and $y(\lambda_i^\vee) \mapsto Y_i$. By definition, it is clear that $p$ is surjective. 
We now check the following in the next two subsections. 
\begin{enumerate}
\item $p$ is a homomorphism.
\item $\gr(\hat{\mu})\circ p$ is an isomorphism.
\end{enumerate}
This implies that $\gr(\hat{\mu})$ is an isomorphism, which in turn proves Theorem \ref{thm:formal}. 

\subsection{$p$ is a homomorphism}
\subsubsection{The map $p$ respects the (tt) relations}
We follow the approach in \cite{TL1} to show that $p$ preserves the (tt) relations. 

Let $T:=\Hom_\Z(Q, \E)=P^\vee\otimes_\Z\E$ be the torus. 
Denote $\ker(\chi_\alpha)$ by $T_\alpha \subseteq T$.
\begin{lemma}\label{lem:subtori adj}
There exist a component $(T_\alpha \cap T_\beta)_\chi$, such that $(T_\alpha \cap T_\beta)_\chi$ contained in a subtorus $T_\gamma$ if and only if $\gamma \in \Z\alpha+\Z\beta$.
\end{lemma}
\begin{proof}
Without loss of generality, we assume $\alpha$ is a simple root. Note $\Z\alpha+\Z\beta$ may not be a primitive sublattice in $Q$.
Pick a vector $e_2$ such that $\alpha, e_2$ generate the primitive closure $\overline{\Z\alpha+\Z\beta}$ of $\Z\alpha+\Z\beta$. Write $\beta=a\alpha+b e_2$. Extend the set $\{\alpha, e_2\}$ to a basis $\{e_1:=\alpha, e_2, e_3, \dots, e_n\}$ of $Q$.
Choose the corresponding dual basis $\{f_1, f_2, f_3, \dots, f_n\}$ of $P^\vee$. Recall that the coweight lattice $P^\vee$ is dual to root lattice $Q$.

Identify $T=P^\vee\otimes_\Z \E$ with $(\E)^n$ using the basis $\{f_1, f_2, f_3, \dots, f_n\}$ of $P^\vee$. Then, the maps $\chi_\alpha, \chi_\beta$ are giving by
\begin{align*}
&\chi_\alpha: T=(\E)^n \to \E,  \,\ \sum_{i=1}^n z_i f_i \mapsto z_1.\\
&\chi_\beta: T=(\E)^n \to \E,\,\  \sum_{i=1}^n z_i f_i \mapsto a z_1+b z_2.
\end{align*}
Thus, $T_\alpha\cong \{0\} \times (\E)^{n-1}\subseteq (\E)^n$, and $T_\alpha\cap T_\beta=\{(0, z_2, \dots, z_n)\subseteq(\E)^n \mid b z_2=0\}$. It is clear that the number of connected components of $T_\alpha \cap T_\beta$ is $b^2$.
Taking the component  $(T_\alpha\cap T_\beta)_\chi$ to be $\{(0, \frac{1}{b}, z_3 \dots, z_n)\subseteq(\E)^n\}$.
Write $\gamma=\sum_{i=1}^n m_i e_i$, then, the following holds.
\begin{align*}
(T_\alpha \cap T_\beta)_\chi \subseteq T_\gamma &\Longleftrightarrow
\frac{1}{b}m_2+\sum_{i=3}^n m_i z_i\in \Z+\tau \Z, \forall z_i\in \C, i=3, \dots n.\\
&\Longleftrightarrow b\mid m_2, \,\ \text{and } \,\ m_3=\cdots =m_n=0.\\
&\Longleftrightarrow \gamma \in \Z\alpha+\Z\beta.
\end{align*}
This completes the proof. 
\end{proof}

\begin{prop}
The (tt)-relation $[t_\alpha, \sum_{\beta \in \Psi^+}t_\beta]=0$ holds in $\gr(\widehat{\C \Pell^\Phi})$.
\end{prop}
\begin{proof}
Fix two roots $\alpha, \beta\in \Phi^+$, let $(T_\alpha \cap T_\beta)_\chi$ be the 
 component as in Lemma \ref{lem:subtori adj}.
Choose a point $x_0\in (T_\alpha \cap T_\beta)_\chi$, such that $x_0 \notin T_\gamma$, for any $T_\gamma$ satisfying
$(T_\alpha \cap T_\beta)_\chi \nsubseteqq T_\gamma$.

Take an open disc neighborhood $D$ of $x_0$, such that $D\cap T_\gamma =\emptyset$, for those $T_\gamma$, satisfying
$(T_\alpha \cap T_\beta)_\chi \nsubseteqq T_\gamma$. By Lemma \ref{lem:subtori adj}, we have
\[D\cap (\bigcup_{\gamma\in \Phi^+} T_{\gamma})=\bigcup_{\gamma\in \Z\alpha+\Z\beta}(D\cap T_{\gamma}).\]
\cite[Lemma 1.45]{TL2} implies the $(tt)$ relation in $\gr(\widehat{\C \Pell^\Phi})$.
\end{proof}
\subsubsection{The map $p$ respects the (yx) relation}
\begin{prop}\label{prop:yx}
The map $p$ sends the (yx) relation
$[y(u), x(v)]=\sum_{\gamma\in \Phi^+}\langle v, \gamma\rangle \langle u, \gamma\rangle t_\gamma$ to zero in $\gr(\widehat{\C \Pell^\Phi})$.
\end{prop}
\begin{proof}
Note that the loop $[y(\lambda_j^\vee), x(\lambda_i^\vee)]$ of $T_{\reg}$ is a loop wrapping around the hypertori $T_\gamma$, $\gamma\in \Phi^+$.
How it wrap those hypertori maybe complicated. But in $\gr(\widehat{\C \Pell^\Phi})$, by degree reason, we have
\[[y(\lambda_j^\vee), x(\lambda_i^\vee)]=\sum_{\gamma\in \Phi^+}C_{i,j,\gamma} t_\gamma, \] for some coefficient $C_{i,j,\gamma} \in \Q$. 
We now determine the coefficients $C_{i,j,\gamma}$. 
Consider the map 
\[\chi_\alpha: (P^\vee\otimes_\Z \E)\setminus \ker \chi_\alpha\to \E\setminus \{0\}.\]
Denote by $\varphi_\alpha$ the composition
\[\varphi_\alpha: T_{\reg}:=(P^\vee\otimes_\Z \E)\setminus (\cup_{\gamma\in \Phi^+}\ker \chi_\gamma) \inj (P^\vee\otimes_\Z \E)\setminus \ker \chi_\alpha\to \E\setminus \{0\}\]
Let $x(\lambda_i^\vee)=(t+\frac{1}{2}+\frac{1}{2}\tau)\lambda_i^\vee$, where $0 \leq t \leq 1$, and $y(\lambda_j^\vee)=(s\tau+\frac{1}{2}+\frac{1}{2}\tau)\lambda_j^\vee$, where $0 \leq s \leq 1$ be two loops in $T_{\reg}$. Using \cite[Corollary 5.1]{B}, one has the following relation in the fundamental group of the punctured torus $\E\setminus \{0\}$: 
\[
\Big[\frac{1}{(\lambda_j^\vee, \alpha)}\varphi_\alpha (y(\lambda_j^\vee)), \frac{1}{(\lambda_i^\vee, \alpha)}\varphi_\alpha(x(\lambda_i^\vee))\Big]=t_\alpha,
\]
where $t_\alpha$ is the loop around the puncture $\{0\}$.
The above relation can be rewritten as $[\varphi_\alpha(y(u)), \varphi_\alpha(x(v))]=(u, \alpha)(v, \alpha)t_\alpha$. 
This determines the coefficient $C_{i,j,\gamma}=\langle \lambda_j^\vee, \gamma\rangle \langle \lambda_i^\vee, \gamma\rangle$.
This completes the proof. 
\end{proof}
\subsubsection{The map $p$ respects the (xx) (yy) and (tx) (ty) relations}
We show the relations $[x(u), x(v)]=0$, $[y(u), y(v)]=0$, for any $u, v\in \h$ and
the relations $[t_\alpha, x(u)]=0$, $[t_\alpha, y(u)]=0$ for $(\alpha, u)=0$ in the pure elliptic braid group $\Pell^\Phi$. 
There is an embedding of groups $\Pell^\Phi\inj \Bell$. It is suffices to show the relations in the elliptic braid group $\Bell$. 

We first recall some results from \cite{C} about the topological interpretation of the double affine braid group $\widetilde{\Bell}$.

Let $\iota^2=-1$, and set
\[
U=\{z\in \C^n\mid (\alpha, z)\notin \Z+\iota\Z, \,\ \text{for any root $\alpha\in \Phi$}\}, \,\
\bar{U}=\C^*\times U.
\]

Let $\bar{W}$ be the 2-extended Weyl group $\bar{W}:=W \ltimes (P^\vee \oplus \iota P^\vee)$.
For any element $\bar{w}=(w, a, b)\in \bar{W}$, the action of $\bar{w}$ on $\bar{z}=(z_*, z)\in \bar{U}$ is giving by
\[
\bar{w}(\bar{z})=( (-1)^{l'(w)}, w(z)+a+\iota b), \,\ \text{where $l'(w)$ is the modified length of $w$.}
\]

We fix a base point $\bar{z}^0=(z_*^0, z^0)$ such that the real and imaginary parts of $\{z_*^0, z_j^0, 1\leq j\leq n\}$ are positive and sufficiently small.
Note that the action of $\bar{W}$ on $\bar{U}$ is not free.
\begin{definition}\cite[Definition 2.3]{C}
The double affine braid group $\widetilde{\Bell}$ is formed by
the paths $\gamma\subset \bar{U}$ joining $\bar{z}^0$ with points from $\{\bar{w}(\bar{z}^0), \bar{w}\in \bar{W} \}$ modulo the homotopy and the action of $\bar{W}$.
\end{definition}

We introduce the elements $t_j:=t_{\alpha_j}$, $x_j$, $y_j$ for $1\leq j\leq n$ and $c \in \widetilde{\Bell}$ by the following paths, for $1\leq \psi \leq 1$,
\begin{align*}
&t_j(\psi)=\Big( z^0_*\exp (l'(j) \pi \iota \psi), \,\  z^0+(\exp(\pi \iota \psi)-1)\frac{(z^0, \alpha_j)\alpha_j}{2}\Big),\\
&x_j(\psi)=(z^0_*, z^0+\psi b_j), \,\ \,\  y_j(\psi)=(z^0_*, z^0+\iota \psi b_j),
\\
&c(\psi)=(z^0_* \exp(-2\pi \iota \psi), z^0).
\end{align*}

\begin{theorem}\cite[Theorem 2.4]{C}
The group $\widetilde{\Bell}$ is generated by $\{t_i, 1\leq i\leq n \}$, the two sets $\{x_i\}$, $\{y_i\}$ of pairwise commutative elements and central $c$ with the following relations
\begin{enumerate}
  \item $t_i t_j t_i \dots =t_j t_i t_j \dots$, $m_{ij}$ factors on each side.
  \item $t_i x_i t_i= x_i x^{-1}_{a_i}c_i^{-1}$ and
        $t_i^{-1} y_i t_i^{-1}= y_i y^{-1}_{a_i}c_i$,
  \item $p_r x_i p_r^{-1}=x_j x_i^{-(b_i, \theta)}$ and
        $p_r x_{r^*} p_r^{-1}=x_r^{-1}$,
\end{enumerate}
where $c_i=c^{l'(i)}$, $p_r:=y_r t_{w}^{-1}$ for $w=\sigma_r^{-1}$.
\end{theorem}
As discussed in \cite{C}, the map $T_i\mapsto t_i$, $X_b\mapsto x_b c^{(\rho ', b)}$ and
$Y_b\mapsto y_b c^{-(\rho ', b)}$, $\delta\mapsto c^d$, where $\rho':=\sum_{\alpha\in \Phi^+}\frac{l'(\alpha)\alpha}{2}$, $d=\frac{(\rho', \theta)+1}{m}\in \Z$, $b\in P^\vee$, identifies the double affine Hecke group (see \cite{C} for the definition) with a subgroup of $\Bell$.
In particular, we have the relations
\[ T_i X_b=X_b T_i, \,\ \,\ T_i Y_b=Y_b T_i,  \,\ \text{if $(b, \alpha_i)=0$}. \]

\begin{corollary}
The orbifold fundamental group $\Bell$ of the space $T_{\reg}/W$ is isomorphic to the double affine braid group $\widetilde{\Bell}$ modulo the central element $c$.
\end{corollary}
Therefore, the (xx) (yy) and (tx) (ty) relations hold in $\Pell^\Phi$.

\subsection{$\gr(\hat{\mu})\circ p$ is an isomorphism}
\begin{lemma}
The map $\gr(\hat{\mu}): \gr(\widehat{\C \Pell^\Phi}) \to \widehat{U(\Aell)}$ is given by
$X_{u^i} \mapsto -y(u^i)$, $Y_{u^i} \mapsto 2\pi i x(u^i) -\tau y(u^i)$, and
$T_\alpha \mapsto 2\pi i t_{\alpha}$.
\end{lemma}
\begin{proof}
Recall the KZB connection \eqref{conn any type} is
\[\nabla_{\KZB, \tau}=d-\sum_{\alpha \in \Phi^+} k(\alpha, \ad(\frac{x_{\alpha^\vee}}{2})|\tau)(t_\alpha)d\alpha+\sum_{i=1}^{n}y(u^i)d u_i,\]
where $k(\alpha, \ad(\frac{x_{\alpha^\vee}}{2})|\tau)(t_\alpha)=\frac{1}{\alpha}t_\alpha+$ terms of degree $\geq 3$.
Let $F_{\textbf{u}^0}(\textbf{u})$ be the horizontal section of $\nabla_{\KZB, \tau}$. Then, $\log F_{\textbf{u}_0}(\textbf{u})=-\sum_{i=1}^n (u_i-u_i^0)y(u^i)+$ terms of degree $\geq 2$.

\begin{corollary}\label{cor:solution}
For some open sets $U$ and  $V$, we have
\begin{align*}
&F_{\textbf{u}^0}^{U}(\textbf{u}+\delta_i)=F_{\textbf{u}^0}^{U}(\textbf{u})\mu_{\textbf{u}_0}(x(u^i)),\\
&F_{\textbf{u}^0}^{V}(\textbf{u}+\tau\delta_i)=e^{-2 \pi i x(u^i)}F_{\textbf{u}^0}^{V}(\textbf{u})\mu_{\textbf{u}_0}(y(u^i)),
\end{align*}
where $\mu$ is the monodromy representation. 
\end{corollary}
\begin{proof}
This follows directly from Proposition \ref{quasi peri}.
\end{proof}
By Corollary \ref{cor:solution}, we get 
\begin{align*}
&\log \mu_{\textbf{u}_0}(x(u^i))=-y(u^i)+\text{ terms of degree $\geq 2$}, \\
&\log \mu_{\textbf{u}_0}(y(u^i))=2\pi i x(u^i)-\tau y(u^i)+\text{ terms of degree $\geq 2$.}
\end{align*}
Therefore, $\gr(\hat{\mu})(x(u^i))=-y(u^i)$ and $\gr(\hat{\mu})(y(u^i))=2\pi i x(u^i)-\tau y(u^i)$.
Note that \[
\hat{\mu}(T_\alpha)=\int_{T_\alpha}\! (\text{$\frac{1}{\alpha}t_\alpha+$ terms of degree $\geq 3$} \,) \mathrm{d} \alpha=
\text{$2\pi i t_\alpha$+terms of degree $\geq 3$.}
\]
So we have $\gr(\hat{\mu})(T_\alpha)=2\pi i t_{\alpha}$.
The argument above shows the following. 
\end{proof}
\begin{corollary}
The composition $\gr(\hat{\mu})\circ p: \widehat{U(\Aell)}\to \widehat{U(\Aell)}$ is given by
$x(u^i) \mapsto -y(u^i)$, $y(u^i) \mapsto 2\pi i x(u^i) -\tau y(u^i)$, and
$t_\alpha \mapsto 2\pi i t_{\alpha}$. In particular, $\gr(\hat{\mu})\circ p$ is an isomorphism.
\end{corollary}

\subsection{}

In this section, we prove the following.

\begin{theorem}\label{th:not quadratic}\hfill
\begin{enumerate}
\item The Lie algebra $\Aell$ is not quadratic. 
\item The elliptic pure braid group is not $1$--formal.
\end{enumerate}
\end{theorem}
\begin{proof}
By Theorem \ref{thm:formal}, (2) is a direct consequence of (1). Set $\deg(x(u))=\deg(y(u))=1$,
and $\deg(t_\gamma)=2$, for $\gamma\in \Phi^+$.  Denote by $(\Aell)^{(2)}$ the vector space
of degree 2 elements in $\Aell$. If $\Aell$ is quadratic, which we henceforth assume, it is generated
by $x(u),y(v)$ and the commutators $[x(u),y(v)]$ span $(\Aell)^{(2)}$. Since $[x(u),y(v)]=[x(v),y(u)]$
by relation (yx) of Definition \ref{def:holonomy Aell}, $(\Aell)^{(2)}$ is therefore of dimension at
most $n(n+1)/2$, where $n=\dim\h$. 

By Proposition \ref{prop:map to H_{h,c}}, there is an algebra homomorphism $\rho$ from $\Aell$
to the rational Cherednik algebra $H_{\hbar, c}$, which maps the element $t_{\gamma}$ to the
reflection $s_{\gamma}$ in the group algebra $\C W$, which is a subalgebra of $H_{\hbar, c}$.
Since the set $\{ s_{\gamma}\mid \gamma\in \Phi^+\}$ is linearly independent in $\C W$, it follows
that $\{t_{\gamma}\mid \gamma\in \Phi^+\}$ is linearly independent in $(\Aell)^{(2)}$. For any root
system $\Phi$ other than type $\mathsf{A}$, we have $|\Phi^+|> \frac{n(n+1)}{2}$, which is a
contradiction. If the root system $\Phi$ is of type $\mathsf{A}$, it is shown in \cite{Bez} that
$\Aell$ is not quadratic, as the degree 3 relations (tx), (ty) in Definition \ref{def:holonomy Aell}
can not be obtained from the degree 2 relations.  
\end{proof}

\section{Derivations of the Lie algebra $\Aell$}
\label{sec:derivation}

Let $\mathfrak{d}$ be the Lie algebra defined in \cite{CEE} with generators $\Delta_0, d, X$, and $\delta_{2m}( m\geq 1)$, and relations
\begin{align*}
&[d, X]=2X,\,\ [d, \Delta_0]=-2\Delta_0, \,\ [X, \Delta_0]=d,\\
&
[\delta_{2m}, X]=0,\,\ [d, \delta_{2m}]=2m\delta_{2m},\,\ (\ad\Delta_0)^{2m+1}(\delta_{2m})=0.
\end{align*}
We define a Lie algebra morphism $ \mathfrak{d} \to \Der(\Aell)$, denoted by $\xi \mapsto \tilde{\xi}$. The image $\tilde{\xi}$ of $\xi$ acts on $\Aell$ by the following formulas.
\begin{align*}
&\tilde{d}(x(u))=x(u), \phantom{1234} \tilde{d}(y(u))=-y(u), \phantom{1234}   \tilde{d}(t_{\alpha})=0, \notag\\
&\tilde{X}(x(u))=0, \phantom{123456}   \tilde{X}(y(u))=x(u), \phantom{12345}  \tilde{X}(t_{\alpha})=0, \notag\\
&   \widetilde{\Delta}_0(x(u))=y(u),  \phantom{123}  \widetilde{\Delta}_0(y(u))=0, \phantom{123456} 
\widetilde{\Delta}_0(t_{\alpha})=0, \notag\\
& \tilde{\delta}_{2m}(x(u))=0, \phantom{12345} 
\tilde{\delta}_{2m}(t_{\alpha})=[t_\alpha, (\ad\frac{x(\alpha^\vee)}{2})^{2m}(t_\alpha)], \notag\\
\text{and} \phantom{1234} \,\ &\tilde{\delta}_{2m}(y(u))=
\frac{1}{2}\sum_{\alpha\in \Phi^+}\alpha(u)\sum_{p+q=2m-1}[(\ad\frac{x(\alpha^\vee)}{2})^{p}
   (t_\alpha), (\ad-\frac{x(\alpha^\vee)}{2})^{q}(t_\alpha)].
\end{align*}
\begin{prop}
\label{prop:d is deriv}
The above map $\mathfrak{d} \to \Der(\Aell)$ is a Lie algebra homomorphism.
\end{prop}
For any generator $\xi$ of $\mathfrak{d}$, we check $\tilde{\xi}$ is a derivation of $\Aell$. 
By definition, it is clear that $\tilde{d}$, $\tilde{X}$ and $\widetilde{\Delta}_0$ are derivations of $\Aell$. It is also clear that  $\widetilde{\delta}_{2m}$ respects the relations
$[x(u), x(v)]=0$ and $[t_\alpha, x(u)]=0$, if $(\alpha, u)=0$.
It remains to show that $\tilde{\delta}_{2m}$ preserves the following relations of $\Aell$. 
\begin{description}
\item[(yx)]
$[y(u), x(v)]=\sum_{\gamma\in \Phi^+}(u, \gamma)(v, \gamma)t_\gamma$.
\item[(tt)]
$[t_\alpha, \sum_{\beta \in \Psi^+}t_\beta]=0$,  for any  rank 2 root subsystem $\Psi\subset \Phi$. 
\item[(ty)]
$[t_\alpha, y(u)]=0$, if $(\alpha, u)=0$.
\item[(yy)]
$[y(u), y(v)]=0$.
\end{description}
In the rest of this section, we prove that  $\tilde{\delta}_{2m}$ preserves those relations. Before checking the relations, we need a technical result of the combinatorics of the root system. 
This following subsection will only be used in the proofs in \S\ref{subsec:delta tt}, \S\ref{subsec:delta ty}, \S\ref{subsec:delta yy}.  
\subsection{Combinatorics of the root system}
\label{sec:S}
In this subsection, we introduced the set $S_{\Psi}^{(\alpha)}$, for a rank 2 root subsystem $\Psi$, and $\alpha\in \Phi$ a root. This section is to study the $(tt)$ relations in detail. 

Let $\Phi$ be a root system, for any root $\alpha \in \Phi$, we define a set $S_\Psi^{(\alpha)}$ by
\[
S_\Psi^{(\alpha)}=\{\beta \in \Phi\mid \langle \alpha, \beta\rangle_{\Z}=\Psi\},
\]where $\Psi$ is a rank 2 subsystem of $\Phi$.
\begin{example}\label{B_2}
Let $\Phi$ be the root system of $B_2$, and $\alpha_1$ be the simple long root, $\alpha_2$ be the simple short root. There are two root subsystems. One is $\Psi=\{\pm\alpha_1, \pm(\alpha_1+2\alpha_2)\}$ consisting of all long roots, and the other is $\Phi$ itself. The set of all short roots is not a root system. 

Therefore, we have the following set 
\begin{align*}
&S_\Psi^{(\alpha_1)}=\{ \pm(\alpha_1+2\alpha_2)\}, \,\  S_\Phi^{(\alpha_1)}=\{ \pm(\alpha_1+\alpha_2), \pm \alpha_2\},\\
&S_\Psi^{(\alpha_2)}=\emptyset, \,\  S_\Phi^{(\alpha_2)}=\{\pm\alpha_1, \pm(\alpha_1+\alpha_2), \pm (\alpha_1+2\alpha_2)\}.
\end{align*}
\end{example}

\begin{example}\label{G_2}
Let $\Phi$ be the root system of $G_2$, and $\alpha_1$ be the simple long root, $\alpha_2$ be the simple short root. There are three types of root subsystems.

Type 1: $\Psi_1=\{\pm\alpha_1, \pm(2\alpha_1+3\alpha_2), \pm(\alpha_1+3\alpha_2)\}$ consists of all long roots. 

Type 2: The root subsystem that consists of one long root, and one short root which are perpendicular to each other. 
They are 
$\Psi_2=\{\pm \alpha_1, \pm(\alpha_1+2\alpha_2)\}$, 
$\Psi_3=\{\pm (2\alpha_1+3\alpha_2), \pm(\alpha_2)\}$, 
$\Psi_4=\{\pm (\alpha_1+3\alpha_2), \pm(\alpha_1+\alpha_2)\}$.

Type 3 is $\Phi$ itself. 
Therefore, we have the following set 
\begin{align*}
&S_{\Psi_1}^{(\alpha_1)}=
\{  \pm(2\alpha_1+3\alpha_2), \pm(\alpha_1+3\alpha_2)\}, \,\  
S_{\Psi_2}^{(\alpha_1)}=\{ \pm(\alpha_1+2\alpha_2)\},\,\
S_{\Psi_3}^{(\alpha_1)}=S_{\Psi_4}^{(\alpha_1)}=\emptyset,\,\
S_\Phi^{(\alpha_1)}=\{ \pm(\alpha_1+\alpha_2), \pm\alpha_2\},\,\
\\
&S_{\Psi_1}^{(\alpha_2)}=S_{\Psi_2}^{(\alpha_2)}=S_{\Psi_4}^{(\alpha_2)}=\emptyset, \,\  
S_{\Psi_3}^{(\alpha_2)}=\{\pm(2\alpha_1+3\alpha_2)\},\,\
S_{\Psi_4}^{(\alpha_2)}=\{\pm\alpha_1, 
\pm(\alpha_1+\alpha_2), 
\pm(\alpha_1+2\alpha_2),
\pm(\alpha_1+3\alpha_2)\}.
\end{align*}
\end{example}
\Omit{
\begin{lemma}
The set $\{S_\Psi^{(\alpha)}\}$ is one to one correspondence to the set
\[
\{\text{$\Psi \subseteq \Phi$ of rank 2 subsystem} \mid \alpha \in \Psi
\}
\]
\end{lemma}
\begin{proof}
Sending $\Psi$ in the second set to $S_\Psi=(\Psi\setminus \{\alpha\})\setminus\cup_{\Psi'\subsetneq \Psi}\Psi'$ gives us the bijection.
\end{proof}
}
\begin{lemma}\label{lem:S}
For $\Psi_1\neq \Psi_2$, we have $S_{\Psi_1}^{(\alpha)}\cap S_{\Psi_2}^{(\alpha)}=\emptyset$. Furthermore, we have the decomposition of disjoint sets
\begin{equation}\label{eq:S}
\Phi=( \sqcup_{\{\Psi\subset \Phi\}} S_\Psi^{(\alpha)})\sqcup\{\pm\alpha\}.
\end{equation}
\end{lemma}
\begin{proof}
For any $\beta\in \Phi$, such that $\beta\neq \alpha$, consider 
$\Psi:=\langle \alpha, \beta \rangle_\Z \cap \Phi$, then $\beta\in S_\Psi^{(\alpha)}$. Thus, $\beta$ lies in the right hand side of \eqref{eq:S}. It is clear that the right hand side of \eqref{eq:S} is a disjoint union. 
\end{proof}

\begin{corollary}\label{cor:tt}\hfill
\begin{enumerate}
\item For any rank 2 root subsystem $\Psi\subset \Phi$, we have the (tt) relation
$
\Big[
t_{\alpha}, \sum_{\beta\in S_\Psi^{(\alpha)} } t_{\beta}\Big]=0. $
\item
For any $\beta\in S_\Psi^{(\alpha)}$, we have
$
[t_{\beta}, \sum_{\gamma\in S_\Psi^{(\alpha)}}t_{\gamma}+t_{\alpha}]=0.
$
\end{enumerate}
\end{corollary}
\begin{proof}
Let $\Psi\subset \Phi$ be a rank 2 root subsystem containing $\alpha$, we have
$S_\Psi^{(\alpha)}=(\Psi\setminus \{\pm \alpha\})\setminus\cup_{\{\Psi'\subsetneq \Psi\mid \alpha\in \Psi'\}}\Psi'$.
The conclusion follows from the (tt) relations
$[t_{\beta}, \sum_{\gamma\in \Psi^+} t_\gamma]=0$, and 
$[t_{\beta}, \sum_{\gamma\in \Psi'^+} t_\gamma]=0$. 
\end{proof}

\Omit{
\begin{lemma}
Let $\Phi$ be a root system of rank 2. Then, for any $\alpha\in\Phi_+$,
the sum
\[\sum_{\beta\in\Phi_+:\langle\alpha,\beta\rangle_{\mathbb Z}=\Phi}
[t_\alpha,t_\beta]\]
lies in the subspace spanned by the (tt) relations.
\end{lemma}
\begin{proof}
The statement is clearly true if $\Phi$ does not properly contain
a root subsystem, for in that case the sum ranges over all $\beta
\in\Phi_+\setminus\{\alpha\}$. In general, write
\[
0=[t_\alpha,\sum_{\beta\neq\alpha}t_\beta]=
\sum_{\Psi\subset \Phi}
[t_\alpha,\sum_{\langle\alpha,\beta\rangle_{\mathbb Z}=\Psi} t_\beta]\]

where $\Psi$ ranges over all subsystems of $\Phi$ containing $\alpha$.
Write above as
\[[t_\alpha,\sum_{\langle\alpha,\beta\rangle_{\mathbb Z}=\Phi}t_\beta]+
\sum_{\Psi\subsetneq\Phi} [t_\alpha,\sum_{\langle\alpha,\beta\rangle_{\mathbb Z}=\Psi} t_\beta]
\]
By induction on $\Psi$, the second summand is zero, and so the first
one must be too.
\end{proof}}
\begin{lemma}\label{lem:omeg const}
For any $u_1, u_2 \in S_\Psi^{(\beta)}$, we have
\[
\omega(u_1^\vee, \beta)=\pm \omega(u_2^\vee, \beta). 
\]\end{lemma}
\begin{proof}
For $u_i\in S_\Psi^{(\beta)}$, $i=1, 2$.
Since $\langle u_i, \beta\rangle_{\Z}=\Psi$, the transition matrix between the two integral basis $\{u_1, \beta\}$ and $\{u_2, \beta\}$ is of the form
$\left(
  \begin{array}{cc}
    1 & y \\
    0 & x \\
  \end{array}
\right)$, thus $x=\pm 1$.
Write $u_1=\pm u_2+y\beta$.
Then, $\omega(u_1^\vee, \beta)=\omega((\pm u_2+y\beta)^\vee, \beta)=\pm\omega(u_2^\vee, \beta)$. So the assertion follows.
\end{proof}

\subsection{The map $\tilde{\delta}_{2m}$ preserves relation (yx)}
\label{subsec:delta yx}
We show
\begin{equation}\label{delta yx}
[\tilde{\delta}_{2m}(y(u)), x(v)]+[y(u), \tilde{\delta}_{2m}(x(v))]=\sum_{\gamma\in \Phi^+}(u, \gamma)(v, \gamma)\tilde{\delta}_{2m}(t_\gamma)
\end{equation}
in this subsection. 

By definition of $\tilde{\delta}_{2m}$, the left hand side of \eqref{delta yx} is
\[
[\tilde{\delta}_{2m}(y(u)), x(v)]+[y(u), \tilde{\delta}_{2m}(x(v))]=
\Big[\frac{1}{2}\sum_{\gamma\in \Phi^+}\gamma(u)
\sum_{p+q=2m-1}[(\ad\frac{x_{\gamma^\vee}}{2})^{p}(t_\gamma), (\ad-\frac{x_{\gamma^\vee}}{2})^{q}(t_\gamma)],  \,\ x(v)\Big]. \]
By linearity of $x(v)$ in $v$, we have $x(v)=\frac{1}{2}(v, \gamma)x_{\gamma^\vee}+x(v')$,
where $(\gamma, v')=0$.  Using the (tx)-relation $[t_{\gamma}, x(v')]=0$, we have
 \[
[[(\ad\frac{x_{\gamma^\vee}}{2})^{p}(t_\gamma), (\ad-\frac{x_{\gamma^\vee}}{2})^{q}(t_\gamma)], \,\ x(v')]=0. 
\]
Taking into account this simplification, we have
\begin{align*}
&[\tilde{\delta}_{2m}(y(u)), x(v)]+[y(u), \tilde{\delta}_{2m}(x(v))]\\
=&[\frac{1}{2}\sum_{\gamma \in \Phi^+}\gamma(u)
\sum_{p+q=2m-1}[(\ad\frac{x_{\gamma^\vee}}{2})^{p}(t_\gamma), (\ad-\frac{x_{\gamma^\vee}}{2})^{q}(t_\gamma)], (v, \gamma)\frac{x_{\gamma^\vee}}{2}]\\
=&\frac{1}{2}\sum_{\gamma\in \Phi^+}(v, \gamma)\gamma(u)
[\sum_{p+q=2m-1}([(\ad\frac{x_{\gamma^\vee}}{2})^{p+1}(t_\gamma), (\ad-\frac{x_{\gamma^\vee}}{2})^{q}(t_\gamma)]
-[(\ad\frac{x_{\gamma^\vee}}{2})^{p}(t_\gamma), (\ad-\frac{x_{\gamma^\vee}}{2})^{q+1}(t_\gamma)])\\
=&\sum_{\gamma \in\Phi^+}(v, \gamma)\gamma(u)
[t_\gamma, (\ad-\frac{x_{\gamma^\vee}}{2})^{2m}(t_\gamma)]
=\sum_{\gamma\in \Phi^+}(v, \gamma)(u, \gamma)\tilde{\delta}_{2m}(t_\gamma).
\end{align*}
This completes the proof that $\tilde{\delta}_{2m}$ preserves relation (yx). 

\subsection{The map $\tilde{\delta}_{2m}$ preserves relation (tt)}
\label{subsec:delta tt}
We show in this subsection that
\begin{equation}\label{eq:delta tt}
[\tilde{\delta}_{2m}(t_\alpha), \sum_{\beta \in \Psi^+}t_\beta]+
[t_\alpha, \sum_{\beta \in \Psi^+}\tilde{\delta}_{2m}(t_\beta)]=0.
\end{equation}
\begin{lemma}\label{lem:delta tt}
Let $\Psi$ be a rank 2 root system, and suppose we have the relation $[t_\alpha, \sum_{\beta \in \Psi^+}t_\beta]=0$.  Then, 
\[
\tilde{\delta}_{2m}(t_\beta)=\sum_{\gamma\in \Psi^+}[t_\beta, (\ad\frac{x_{\gamma^\vee}}{2})^{2m}(t_\gamma)].
\]
\end{lemma}
\begin{proof}
To show the claim, it suffices to show $\sum_{\{\gamma\in \Psi^+\mid \gamma\neq \beta\}}[t_\beta, (\ad\frac{x_{\gamma^\vee}}{2})^{2m}(t_\gamma)]=0$.

Using the $(tx)$ relation, we have 
\begin{align*}
\sum_{\{\gamma\in \Psi^+\mid \gamma\neq \beta\}}[t_\beta, (\ad\frac{x_{\gamma^\vee}}{2})^{2m}(t_\gamma)]
=\sum_{\{\gamma\in \Psi^+\mid \gamma\neq \beta\}}[t_\beta, (\ad\frac{x_{\omega(\gamma^\vee, \beta)}}{-2})^{2m}(t_\gamma)]
=\sum_{\Psi'\subset \Psi} \sum_{ \gamma\in S_{\Psi'}^{(\beta)}\cap\Psi^+}
[t_\beta, (\ad\frac{x_{\omega(\gamma^\vee, \beta)}}{-2})^{2m}(t_\gamma)],
\end{align*}
where the last equality follows from the decomposition in Lemma \ref{lem:S}. 
By Lemma \ref{lem:omeg const}, the term $(\ad\frac{x_{\omega(\gamma^\vee, \beta)}}{-2})^{2m}$ is independent of $\gamma\in S_{\Psi'}^{(\beta)}$. Using the relation $[t_\beta, x_{\omega(\gamma^\vee, \beta)}]=0$, we have
\[
\sum_{ \gamma\in S_{\Psi'}^{(\beta)}\cap\Psi^+}
[t_\beta, (\ad\frac{x_{\omega(\gamma^\vee, \beta)}}{-2})^{2m}(t_\gamma)]
=(\ad\frac{x_{\omega(\gamma^\vee, \beta)}}{-2})^{2m}
[t_{\beta}, \sum_{ \gamma\in S_{\Psi'}^{(\beta)}\cap\Psi^+} t_\gamma]
=0. 
\]
The last equality follows from the (tt) relation in Corollary \ref{cor:tt}. This completes the proof. 
\end{proof}
We use Lemma \ref{lem:delta tt} to show the equality \eqref{eq:delta tt}.
By Lemma \ref{lem:delta tt}, we have
\begin{align*}
[\tilde{\delta}_{2m}(t_\alpha), \sum_{\beta \in \Psi^+}t_\beta]+
[t_\alpha,  \sum_{\beta \in \Psi^+} &\tilde{\delta}_{2m}(t_\beta)]
=-\Big[t_\alpha,\sum_{\beta\in \Psi^+}[t_\beta, (\ad\frac{x_{\alpha^\vee}}{2})^{2m}(t_\alpha)]\Big]
+\Big[t_\alpha, \sum_{\beta\in \Psi^+}\sum_{\gamma\in \Psi^+}[t_\beta, (\ad\frac{x_{\gamma}}{2})^{2m}(t_\gamma)]\Big]\\
=&\Big[ t_\alpha,\sum_{\beta\in \Psi^+}[t_\beta, \sum_{\{\gamma|\gamma\neq\alpha\}}(\ad\frac{x_{\gamma^\vee}}{2})^{2m}(t_\gamma)]\Big]
=\Big[t_\alpha,\sum_{\beta\in \Psi^+}[t_\beta, \sum_{\{\gamma|\gamma\neq\alpha\}}(\ad\frac{x_{\omega(\gamma^\vee, \alpha)}}{2})^{2m}(t_\gamma)]\Big]\\
=&\sum_{\Psi'\subset \Psi}\Big[t_\alpha,[\sum_{\beta\in \Psi^+}t_\beta, \sum_{\gamma\in S^{(\alpha)}_{\Psi'}\cap \Psi^+}(\ad\frac{\omega(\gamma^\vee, \alpha)}{2})^{2m} t_\gamma]\Big],
\end{align*}
where the last equality follows from the decomposition in Lemma \ref{lem:S}. 
By Lemma \ref{lem:omeg const}, the term $(\ad\frac{x_{\omega(\gamma^\vee, \beta)}}{2})^{2m}$ is independent of $\gamma\in S_{\Psi'}^{(\beta)}$. Using the relation $[t_\beta, x_{\omega(\gamma^\vee, \beta)}]=0$,  we have 
\[
\Big[t_\beta, \sum_{\gamma\in S^{(\alpha)}_{\Psi'}\cap \Psi^+}(\ad\frac{\omega(\gamma^\vee, \alpha)}{2})^{2m} t_\gamma\Big]
=(\ad\frac{\omega(\gamma^\vee, \alpha)}{2})^{2m} [t_{\beta}, \sum_{\gamma\in S^{(\alpha)}_{\Psi'}\cap \Psi^+} t_\gamma]=0. 
\]
The last equality follows from the (tt) relation in Corollary \ref{cor:tt}. 
This implies the equality \eqref{eq:delta tt}.
\subsection{The map $\tilde{\delta}_{2m}$ preserves relation (ty)}
\label{subsec:delta ty}
We show in this subsection
\begin{equation}\label{eq:delta yt}
[y(u), \tilde{\delta}_{2m}(t_\alpha)]+[\tilde{\delta}_{2m}(y(u)), t_\alpha]=0,  \,\ \text{if $(\alpha, u)=0$}.
\end{equation}
We first calculate the term $[\tilde{\delta}_{2m}(y(u)), t_\alpha]$. 
\begin{lemma}\label{lem:gamma ij}
We write $\{\Psi\subset \Phi\}$ for the subset of rank 2 root subsystems of $\Phi$.  We have
\begin{equation}\label{delta yt}
[\tilde{\delta}_{2m}(y(u)), t_\alpha]=\frac{1}{2}
\sum_{\Psi\subset\Phi}
\sum_{\{(\gamma, \beta)\in S^{(\alpha)}_{\Psi} \cap \Psi^+\mid \gamma\neq \beta\}}
\gamma(u) 
\Bigg[t_{\alpha}, \,\ 
\frac{ (\ad\frac{x_{\omega(\beta^\vee, \gamma)}}{2})^{2m} 
-(\ad\frac{x_{\omega(\gamma^\vee, \beta)}}{2})^{2m}}{
 -\epsilon(\beta, \gamma)\ad\frac{x_{\omega(\beta^\vee, \gamma)}}{2}
 -\ad\frac{x_{\omega(\gamma^\vee, \beta)}}{2} }   [t_{\beta}, t_\gamma]\Bigg],
\end{equation}
where $\epsilon(\beta, \gamma) =\pm 1$. It is determined by the equality $\omega(\gamma^\vee, \alpha)=\epsilon(\beta, \gamma) \omega(\beta^\vee, \alpha)$.\end{lemma}
\begin{proof}
For fixed $p, q$, such that $p+q=2m-1$. Using the relations 
$[x_{\gamma^\vee}, t_{\gamma}]=-[x_{\omega(\gamma^\vee, \alpha)}, t_{\gamma}]$ and $[x_{\omega(\gamma^\vee, \alpha)}, t_{\alpha}]=0$, we have
\begin{align}
&\Big[[(\ad\frac{x_{\gamma^\vee}}{2})^{p}(t_\gamma), (\ad-\frac{x_{\gamma^\vee}}{2})^{q}(t_\gamma)], t_\alpha\Big] \notag\\
=&\Big[[t_\alpha, (\ad\frac{x_{\gamma^\vee}}{-2})^{q}(t_\gamma)], (\ad\frac{x_{\gamma^\vee}}{2})^{p}(t_\gamma)\Big]
-\Big[[t_\alpha, (\ad\frac{x_{\gamma^\vee}}{2})^{p}(t_\gamma)], (\ad\frac{x_{\gamma^\vee}}{-2})^{q}(t_\gamma)\Big] \notag\\
=&\Big[(\ad\frac{x_{\omega(\gamma^\vee, \alpha)}}{2})^{q}[t_\alpha, t_\gamma], (\ad\frac{x_{\gamma^\vee}}{2})^{p}(t_\gamma)\Big]
-\Big[(\ad\frac{x_{\omega(\gamma^\vee, \alpha)}}{-2})^{p}[t_\alpha, t_\gamma], (\ad\frac{x_{\gamma^\vee}}{-2})^{q}(t_\gamma)\Big]
\notag\\
=&\Big[(\ad\frac{x_{\omega(\gamma^\vee, \alpha)}}{2})^{q}[t_\alpha, t_\gamma], (\ad\frac{x_{\gamma^\vee}}{2})^{p}(t_\gamma)\Big]
+\Big[(\ad\frac{x_{\omega(\gamma^\vee, \alpha)}}{2})^{p}[t_\alpha, t_\gamma], (\ad\frac{x_{\gamma^\vee}}{2})^{q}(t_\gamma)\Big]. \label{eq:gamma pq}
\end{align}
Taking summation over all pairs $(p, q)$ with $p+q=2m-1$, the two summands in \eqref{eq:gamma pq} could be combined. Therefore, we have
\[
\sum_{p+q=2m-1}\Big[[(\ad\frac{x_{\gamma^\vee}}{2})^{p}(t_\gamma), (\ad-\frac{x_{\gamma^\vee}}{2})^{q}(t_\gamma)], t_\alpha\Big]
=2\sum_{p+q=2m-1}\Big[(\ad\frac{x_{\omega(\gamma^\vee, \alpha)}}{2})^{p}[t_\alpha, t_\gamma], (\ad\frac{x_{\gamma^\vee}}{2})^{q}(t_\gamma)\Big].
\]
By \eqref{eq:gamma pq}, together with the (tt) relation 
$[t_\alpha, t_{\gamma}]=-\sum_{\{\beta\in S^{(\alpha)}_{\Psi}\mid  \beta\neq \gamma\}}[t_\alpha, t_{\beta}]$
and the fact $\Phi\backslash \{\pm \alpha\}=\sqcup_{\Psi\subset \Phi} S_{\Psi}^{(\alpha)}$, we have
\begin{align}
&[\sum_{\gamma\in \Phi^+}\gamma(u)\sum_{p+q=2m-1}[(\ad\frac{x_{\gamma^\vee}}{2})^{p}(t_\gamma), (\ad-\frac{x_{\gamma^\vee}}{2})^{q}(t_\gamma)], t_\alpha] \notag\\
=&2\sum_{\gamma\in \Phi^+}\gamma(u) 
\sum_{p+q=2m-1}\Big[(\ad\frac{x_{\omega(\gamma^\vee, \alpha)}}{2})^{p}[t_\alpha, t_\gamma], (\ad\frac{x_{\gamma^\vee}}{2})^{q}(t_\gamma)\Big] \notag\\
=&-2\sum_{\Psi^+\subset\Phi^+}
\sum_{\{\gamma\mid \gamma\in S^{(\alpha)}_{\Psi} \cap \Phi^+\}}
\sum_{\{\beta\mid \beta\in S^{(\alpha)}_{\Psi}\cap \Phi^+,  \,\ \beta\neq \gamma\}}
\gamma(u) 
\sum_{p+q=2m-1}\Big[(\ad\frac{x_{\omega(\gamma^\vee, \alpha)}}{2})^{p}[t_\alpha, t_\beta], (\ad\frac{x_{\gamma^\vee}}{2})^{q}(t_\gamma)\Big] \notag\\
=&-2\sum_{\Psi\subset\Phi}
\sum_{\{(\gamma, \beta)\in S^{(\alpha)}_{\Psi} \cap \Psi^+\mid \gamma\neq \beta\}}
\gamma(u) 
\sum_{p+q=2m-1}\Big[[t_\alpha, (\ad\frac{x_{\omega(\gamma^\vee, \alpha)}}{2})^{p}t_\beta], (\ad\frac{x_{\gamma^\vee}}{2})^{q}(t_\gamma)\Big]. \label{eq:ty relation 1}
\end{align}
For a pair $(\gamma, \beta)\in S^{(\alpha)}_{\Psi}$, as shown in Lemma \ref{lem:omeg const}, 
$\omega(\gamma^\vee, \alpha)=\epsilon(\beta, \gamma) \omega(\beta^\vee, \alpha)$, where $\epsilon(\beta, \gamma) =\pm 1$. Therefore, we simplify \eqref{eq:ty relation 1} further. 
\begin{align}
\Big[[t_\alpha, (\ad\frac{x_{\omega(\gamma^\vee, \alpha)}}{2})^{p}t_\beta], (\ad\frac{x_{\gamma^\vee}}{2})^{q}(t_\gamma)\Big]
&=\epsilon(\beta, \gamma)^p
\Big[[t_\alpha, ( \ad\frac{x_{\omega(\beta^\vee, \alpha)}}{2})^{p}t_\beta], (\ad\frac{x_{\gamma^\vee}}{2})^{q}(t_\gamma)\Big] \notag\\
&=(-\epsilon(\beta, \gamma))^p
\Big[[t_\alpha, ( \ad\frac{x_{\beta^\vee}}{2})^{p}t_\beta], (\ad\frac{x_{\gamma^\vee}}{2})^{q}(t_\gamma)\Big]. \label{eq:ty relation 2}
\end{align}
We plug \eqref{eq:ty relation 2} into \eqref{eq:ty relation 1} and then split \eqref{eq:ty relation 1} into two equal parts. We use the trick that switching the pair $(\beta, \gamma)$, and then the pair of two indices $(p, q)$. 
The identity $\gamma(u)=\epsilon(\beta, \gamma)\beta(u)$ is useful in the simplification, which follows from the fact $(\alpha, u)=0$. 
For $p+q=2m-1$, we have 
$(-\epsilon(\beta, \gamma))^{q+1}=(-\epsilon(\beta, \gamma))^{p}$.  
More precisely, we have
\begin{align*}
2[\tilde{\delta}_{2m}(y(u)), t_\alpha]=&[\sum_{\gamma\in \Phi^+}\gamma(u)\sum_{p+q=2m-1}[(\ad\frac{x_{\gamma^\vee}}{2})^{p}(t_\gamma), (\ad-\frac{x_{\gamma^\vee}}{2})^{q}(t_\gamma)], t_\alpha]\\
=&-2\sum_{\Psi^+\subset\Phi^+}
\sum_{\{(\gamma, \beta)\in S^{(\alpha)}_{\Psi}\cap \Phi^+\mid \gamma\neq \beta\}}
\gamma(u) 
\sum_{p+q=2m-1}(-\epsilon(\beta, \gamma))^p
\Big[[t_\alpha, ( \ad\frac{x_{\beta^\vee}}{2})^{p}t_\beta], (\ad\frac{x_{\gamma^\vee}}{2})^{q}(t_\gamma)\Big]\\
=&-\sum_{\Psi^+\subset\Phi^+}
\sum_{\{(\gamma, \beta)\in S^{(\alpha)}_{\Psi}\cap \Phi^+\mid \gamma\neq \beta\}}
\gamma(u) 
\sum_{p+q=2m-1}(-\epsilon(\beta, \gamma))^p
\Big[[t_\alpha, ( \ad\frac{x_{\beta^\vee}}{2})^{p}t_\beta], (\ad\frac{x_{\gamma^\vee}}{2})^{q}(t_\gamma)\Big]\\
&-\sum_{\Psi^+\subset\Phi^+}
\sum_{\{(\beta, \gamma)\in S^{(\alpha)}_{\Psi}\cap \Phi^+\mid \gamma\neq \beta\}}
\beta(u) 
\sum_{p+q=2m-1}(-\epsilon(\beta, \gamma))^p
\Big[[t_\alpha, ( \ad\frac{x_{\gamma^\vee}}{2})^{p}t_\gamma], (\ad\frac{x_{\beta^\vee}}{2})^{q}(t_\beta)\Big]
\end{align*}
\begin{align*}
=&-\sum_{\Psi^+\subset\Phi^+}
\sum_{\{(\gamma, \beta)\in S^{(\alpha)}_{\Psi}\cap \Phi^+\mid \gamma\neq \beta\}}
\gamma(u) 
\sum_{p+q=2m-1}(-\epsilon(\beta, \gamma))^p
\Big[[t_\alpha, ( \ad\frac{x_{\beta^\vee}}{2})^{p}t_\beta], (\ad\frac{x_{\gamma^\vee}}{2})^{q}(t_\gamma)\Big]
\\
&-\sum_{\Psi^+\subset\Phi^+}
\sum_{\{(\beta, \gamma)\in S^{(\alpha)}_{\Psi}\cap \Phi^+\mid \gamma\neq \beta\}}
\gamma(u)
\sum_{p+q=2m-1}(-\epsilon(\beta, \gamma))^{p}
\Big[(\ad\frac{x_{\beta^\vee}}{2})^{p}(t_\beta), [t_\alpha, ( \ad\frac{x_{\gamma^\vee}}{2})^{q}t_\gamma]\Big]\\
=&-\sum_{\Psi^+\subset\Phi^+}
\sum_{\{(\gamma, \beta)\in S^{(\alpha)}_{\Psi}\cap \Phi^+\mid \gamma\neq \beta\}}
\gamma(u) 
\sum_{p+q=2m-1}(-\epsilon(\beta, \gamma))^p
\Big[t_\alpha, [( \ad\frac{x_{\beta^\vee}}{2})^{p}t_\beta, (\ad\frac{x_{\gamma^\vee}}{2})^{q}(t_\gamma)]\Big]
\end{align*}
\begin{align*}
=&\sum_{\Psi^+\subset\Phi^+}
\sum_{\{(\gamma, \beta)\in S^{(\alpha)}_{\Psi}\cap \Phi^+\mid \gamma\neq \beta\}}
\gamma(u) 
\sum_{p+q=2m-1}(-\epsilon(\beta, \gamma))^p
\Big[t_\alpha, [( \ad\frac{x_{\omega(\beta^\vee, \gamma)}}{2})^{p}t_\beta, (\ad\frac{x_{\omega(\gamma^\vee, \beta)}}{2})^{q}(t_\gamma)]\Big]\\
=&\sum_{\Psi^+\subset\Phi^+}
\sum_{\{(\gamma, \beta)\in S^{(\alpha)}_{\Psi}\cap \Phi^+\mid \gamma\neq \beta\}}
\gamma(u) 
\sum_{p+q=2m-1}
\Big[t_\alpha, ( -\epsilon(\beta, \gamma)\ad\frac{x_{\omega(\beta^\vee, \gamma)}}{2})^{p} (\ad\frac{x_{\omega(\gamma^\vee, \beta)}}{2})^{q}[t_{\beta}, t_\gamma]\Big]\\
=&\sum_{\Psi^+\subset\Phi^+}
\sum_{\{(\gamma, \beta)\in S^{(\alpha)}_{\Psi}\cap \Phi^+\mid \gamma\neq \beta\}}
\gamma(u) 
\Big[t_{\alpha}, \frac{ (\ad\frac{x_{\omega(\beta^\vee, \gamma)}}{2})^{2m} 
-(\ad\frac{x_{\omega(\gamma^\vee, \beta)}}{2})^{2m}}{
 -\epsilon(\beta, \gamma)\ad\frac{x_{\omega(\beta^\vee, \gamma)}}{2}
 -\ad\frac{x_{\omega(\gamma^\vee, \beta)}}{2} }   [t_{\beta}, t_\gamma]\Big]. 
 \end{align*}
 This completes the proof. 
\end{proof}

We then calculate the term $[y(u), \tilde{\delta}_{2m}(t_\alpha)]$. Using the relation $[y(u), t_\alpha]=0$, we have
\begin{align*}
&[y(u), \tilde{\delta}_{2m}(t_\alpha)]\\
=&\Big[t_\alpha, [y(u), (\ad\frac{x_{\alpha^\vee}}{2})^{2m}(t_\alpha)]\Big]&& \text{By $[y(u), t_\alpha]=0$.}\\
=&\Bigg[t_\alpha, \frac{1}{2}\sum_{\{\gamma\in \Phi^+\mid \gamma\neq\alpha\}}(\alpha^\vee, \gamma)(u, \gamma)\sum_{s=0}^{2m}(\ad\frac{x_{\alpha^\vee}}{2})^s
(\ad\frac{x_{\omega(\alpha^\vee,\gamma)}}{-2})^{2m-1-s}[t_\gamma, t_\alpha]\Bigg]&&\text{By Jacobi identity and the relation (yx)}\\
=&\Bigg[t_\alpha,\sum_{\{\gamma\in \Phi^+\mid \gamma\neq\alpha\}}\gamma(u)\frac{(\ad\frac{x_{\alpha^\vee}}{2})^{2m}-
(\ad\frac{x_{\omega(\alpha^\vee,\gamma)}}{-2})^{2m}}
{\ad\frac{x_{\omega(\gamma^\vee,\alpha)}}{-2}}[t_\gamma, t_\alpha]\Bigg].&& \text{By the fact that $\frac{\alpha^\vee+\omega(\alpha^\vee, \gamma)}{(\alpha^\vee, \gamma)}=\frac{\omega(\gamma^\vee, \alpha)}{-2}$.}
\end{align*}
Using the decomposition in Lemma \ref{lem:S}, and the (tt) relation in Corollary \ref{cor:tt}, the above is the same as
\begin{align}\label{more tt}
&\Big[t_\alpha, 
\sum_{\Psi\subset \Phi} 
\sum_{\{\gamma\mid \gamma\in S_{\Psi}^{(\alpha)}\cap \Psi^+ \}}
\gamma(u)\frac{(\ad\frac{x_{\alpha^\vee}}{2})^{2m}-
(\ad\frac{x_{\omega(\alpha^\vee,\gamma)}}{-2})^{2m}}
{\ad\frac{x_{\omega(\gamma^\vee,\alpha)}}{-2}}[t_\gamma, t_\alpha]\Big]
=
\Big[t_\alpha, 
\sum_{\Psi\subset \Phi} 
\sum_{\{\gamma, \beta \in S_{\Psi}^{(\alpha)}\cap \Psi^+ \}}
\gamma(u)\frac{(\ad\frac{x_{\alpha^\vee}}{2})^{2m}-
(\ad\frac{x_{\omega(\alpha^\vee,\gamma)}}{-2})^{2m}}
{\ad\frac{x_{\omega(\gamma^\vee,\alpha)}}{-2}}[t_\beta, t_{\gamma}]\Big].
\end{align}
Switching $\beta$ and $\gamma$, we have
\begin{align}\label{y delta t}
[y(u), \tilde{\delta}_{2m}(t_\alpha)] 
=&\Bigg[t_\alpha, 
\sum_{\Psi\subset \Phi} 
\sum_{\{\gamma, \beta \in S_{\Psi}^{(\alpha)}\cap \Psi^+ \}}
\gamma(u)\frac{(\ad\frac{x_{\alpha^\vee}}{2})^{2m}-
(\ad\frac{x_{\omega(\alpha^\vee,\gamma)}}{-2})^{2m}}
{\ad\frac{x_{\omega(\gamma^\vee,\alpha)}}{-2}}[t_\beta, t_{\gamma}]\Bigg] \notag\\
=&\frac{1}{2}\Bigg[t_\alpha, 
\sum_{\Psi\subset \Phi} 
\sum_{\{\gamma, \beta \in S_{\Psi}^{(\alpha)}\cap \Psi^+ \}}
\gamma(u)\frac{(\ad\frac{x_{\alpha^\vee}}{2})^{2m}-
(\ad\frac{x_{\omega(\alpha^\vee,\gamma)}}{-2})^{2m}}
{\ad\frac{x_{\omega(\gamma^\vee,\alpha)}}{-2}}[t_\beta, t_{\gamma}]\Bigg]\\
&
-\frac{1}{2}\Bigg[t_\alpha, 
\sum_{\Psi\subset \Phi} 
\sum_{\{\gamma, \beta \in S_{\Psi}^{(\alpha)}\cap \Psi^+ \}}
\beta(u)\frac{(\ad\frac{x_{\alpha^\vee}}{2})^{2m}-
(\ad\frac{x_{\omega(\alpha^\vee,\beta)}}{-2})^{2m}}
{\ad\frac{x_{\omega(\beta^\vee,\alpha)}}{-2}}[t_\beta, t_{\gamma}]\Bigg] \notag\\
=&\frac{1}{2}\Bigg[t_\alpha, 
\sum_{\Psi\subset \Phi} 
\sum_{\{\gamma, \beta \in S_{\Psi}^{(\alpha)}\cap \Psi^+ \}}
\gamma(u)\frac{(\ad\frac{x_{\omega(\alpha^\vee,\beta)}}{2})^{2m}-
(\ad\frac{x_{\omega(\alpha^\vee,\gamma)}}{-2})^{2m}}
{\ad\frac{x_{\omega(\gamma^\vee,\alpha)}}{-2}}[t_\beta, t_{\gamma}]\Bigg].
\end{align}
For fixed $\alpha$ and $\Psi$. The set $S_{\Psi}^{(\alpha)}$ is given in Examples \ref{B_2} and \ref{G_2}. 
The relation \eqref{eq:delta yt} follows from direct computation by comparing \eqref{delta yt} and \eqref{y delta t}. 

More precisely, we assume the pair $\beta, \gamma$ satisfy $\beta=\epsilon(\beta, \gamma)\gamma+\alpha$, then, we have
\[
\omega(\beta^\vee, \gamma)=\omega(\alpha^\vee, \gamma), \,\
\omega(\gamma^\vee, \beta)=-\epsilon(\beta, \gamma)\omega(\alpha^\vee, \beta), \,\
\epsilon(\beta, \gamma)
\omega(\beta^\vee, \gamma)+\omega(\gamma^\vee, \beta)
=\omega(\gamma^\vee, \alpha).
\]
Therefore, the relation \eqref{eq:delta yt} holds under this assumption.
 
The assumption $\beta=\epsilon(\beta, \gamma)\gamma+\alpha$ does not hold only when $\alpha$ is the short root of $G_2$.
For the exceptional case, we have $\beta=\epsilon(\beta, \gamma)\gamma+3\alpha$. We modify the equality of \eqref{more tt} to use a more refined (tt) relation for the root system $\Phi_{G_2}$, see Example \ref{G_2}. 
Then, the corresponding term of $[y(u), \tilde{\delta}_{2m}(t_\alpha)]$ in \eqref{y delta t}  will be modified to
\begin{align*}
\frac{1}{2}\Big[t_\alpha, 
\gamma(u)\frac{(\ad\frac{x_{\omega((3\alpha)^\vee,\beta)}}{2})^{2m}-
(\ad\frac{x_{\omega((3\alpha)^\vee,\gamma)}}{-2})^{2m}}
{\ad\frac{x_{\omega(\gamma^\vee,3\alpha)}}{-2}}[t_\beta, t_{\gamma}]\Bigg]. 
\end{align*}
The rest of the proof is exactly the same as before. 
This concludes the relation \eqref{eq:delta yt}.

\subsection{The map $\tilde{\delta}_{2m}$ preserves relation (yy)}
\label{subsec:delta yy}
We show in this subsection
\begin{equation}
[y(u), \tilde{\delta}_{2m} (y(v))]+[\tilde{\delta}_{2m} (y(u)), y(v)]=0, \,\ \text{for any $u, v\in \h$}.
\end{equation}
By definition of $\tilde{\delta}_{2m}$, we have
\begin{align}
[y(u), \tilde{\delta}_{2m} (y(v))]+
[\tilde{\delta}_{2m} (y(u)), y(v)] 
=&\Big[y(u),\frac{1}{2}\sum_{\gamma\in \Phi^+}\gamma(v)
\sum_{p+q=2m-1}[(\ad\frac{x_{\gamma^\vee}}{2})^{p}(t_\gamma), (\ad-\frac{x_{\gamma^\vee}}{2})^{q}(t_\gamma)] \Big] \notag\\
&-\Big[y(v),\frac{1}{2}\sum_{\gamma\in \Phi^+}\gamma(u)
\sum_{p+q=2m-1}[(\ad\frac{x_{\gamma^\vee}}{2})^{p}(t_\gamma), (\ad-\frac{x_{\gamma^\vee}}{2})^{q}(t_\gamma)]\Big] \notag\\
=&\sum_{\gamma\in  \Phi^+}\Big[y_\eta,
\frac{1}{2}\sum_{p+q=2m-1}[(\ad\frac{x_{\gamma^\vee}}{2})^{p}(t_\gamma), (\ad-\frac{x_{\gamma^\vee}}{2})^{q}(t_\gamma)]\Big] \notag\\
=&\sum_{\gamma\in \Phi^+}\sum_{p+q=2m-1}
\Big[[y_\eta, (\ad\frac{x_{\gamma^\vee}}{2})^{p}(t_\gamma)], (\ad-\frac{x_{\gamma^\vee}}{2})^{q}(t_\gamma)\Big], \label{y delta y}
\end{align}
where $y_\eta=\gamma(v)y(u)-\gamma(u)y(v)$ and therefore $[y_\eta, t_\gamma]=0$.

We use the relation (yx), we have
\begin{align}
[y_\eta, (\ad\frac{x_{\gamma^\vee}}{2})^{p}(t_\gamma)] 
=&\sum_{s=0}^{p-1}(\ad\frac{x_{\gamma^\vee}}{2})^{s}
\ad( [y_\eta, \frac{x_{\gamma^\vee}}{2}])
(\ad\frac{x_{\gamma^\vee}}{2})^{p-1-s}(t_\gamma)] \notag\\
=&\sum_{s=0}^{p-1}\sum_{\beta\in \Phi^+}(\beta, \eta)(\beta, \frac{\gamma^\vee}{2})
(\ad\frac{x_{\gamma^\vee}}{2})^{s}
\ad( t_{\beta})
(\ad\frac{x_{\gamma^\vee}}{2})^{p-1-s}(t_\gamma)] \notag\\
=&\sum_{s=0}^{p-1}\sum_{\beta\in \Phi^+}(\beta, \eta)(\beta, \frac{\gamma^\vee}{2})
(\ad\frac{x_{\gamma^\vee}}{2})^{s}
(\ad\frac{x_{\omega(\gamma^\vee, \beta)}}{-2})^{p-1-s}[t_{\beta}, t_\gamma] \notag\\
=&\sum_{s=0}^{p-1}
\sum_{\beta\in \Phi^+}\det(\beta, \gamma)_{u, v}(\beta, \frac{\gamma^\vee}{2})
(\ad\frac{x_{\gamma^\vee}}{2})^{s}
(\ad\frac{x_{\omega(\gamma^\vee, \beta)}}{-2})^{p-1-s}[t_{\beta}, t_\gamma]\notag\\
=&
\sum_{\beta\in \Phi^+}\det(\beta, \gamma)_{u, v}(\beta, \frac{\gamma^\vee}{2})
\frac{(\ad\frac{x_{\gamma^\vee}}{2})^{p}- (\ad\frac{x_{\omega(\gamma^\vee, \beta)}}{-2})^{p}}{
\ad\frac{x_{\gamma^\vee}}{2}+\ad\frac{x_{\omega(\gamma^\vee, \beta)}}{2}
}[t_{\beta}, t_\gamma]\notag\\
=&
\sum_{\beta\in \Phi^+}\det(\beta, \gamma)_{u, v}
\frac{(\ad\frac{x_{\gamma^\vee}}{2})^{p}- (\ad\frac{x_{\omega(\gamma^\vee, \beta)}}{-2})^{p}}{
\ad\frac{x_{\omega( \beta^\vee, \gamma)}}{-2}
}[t_{\beta}, t_\gamma]\label{y_eta xt},
\end{align}
where $
\det(\beta, \gamma)_{u, v}=\beta(u)\gamma(v)-\beta(v)\gamma(u)$. 
The last equality of \eqref{y_eta xt} follows from the linearity of $x(u)$, and the equality $\frac{\gamma^\vee+\omega(\gamma^\vee, \beta)}{(\gamma^\vee, \beta)}=\frac{\omega( \beta^\vee, \gamma)}{-2}$.

We now use the decomposition $\Phi=\sqcup_{\Psi\subset \Phi} S_{\Psi}^{(\gamma)}\sqcup \{\pm \gamma\}$, and the (tt) relation in Corollary \ref{cor:tt}. 
By fixing $\gamma\in \Phi^+$, we have
\begin{align*}
\eqref{y_eta xt}
=&\sum_{\Psi\subset \Phi}
\sum_{\beta\in S_{\Psi}^{(\gamma)}\cap \Psi^+}
\det(\beta, \gamma)_{u, v}
\frac{(\ad\frac{x_{\gamma^\vee}}{2})^{p}-(\ad\frac{x_{\omega(\gamma^\vee, \beta)}}{-2})^{p}}{\ad \frac{x_{\omega(\beta^\vee, \gamma)}}{-2}}[t_\beta, t_\gamma] \notag\\
=&- \sum_{\Psi\subset \Phi}
\sum_{\beta, \alpha\in S_{\Psi}^{(\gamma)}\cap \Psi^+}
\det(\beta, \gamma)_{u, v}
\frac{(\ad\frac{x_{\gamma^\vee}}{2})^{p}-(\ad\frac{x_{\omega(\gamma^\vee, \beta)}}{-2})^{p}}{\ad \frac{x_{\omega(\beta^\vee, \gamma)}}{-2}}[t_\beta, t_\alpha]\notag\\
=&-\frac{1}{2} \sum_{\Psi\subset \Phi}
\sum_{\beta, \alpha\in S_{\Psi}^{(\gamma)}\cap \Psi^+}
\det(\beta, \gamma)_{u, v}
\frac{(\ad\frac{x_{\gamma^\vee}}{2})^{p}-(\ad\frac{x_{\omega(\gamma^\vee, \beta)}}{-2})^{p}}{\ad \frac{x_{\omega(\beta^\vee, \gamma)}}{-2}}[t_\beta, t_\alpha]\notag\\
&+\frac{1}{2} \sum_{\Psi\subset \Phi}
\sum_{\beta, \alpha\in S_{\Psi}^{(\gamma)}\cap \Psi^+}
\det(\alpha, \gamma)_{u, v}
\frac{(\ad\frac{x_{\gamma^\vee}}{2})^{p}-(\ad\frac{x_{\omega(\gamma^\vee, \alpha)}}{-2})^{p}}{\ad \frac{x_{\omega(\alpha^\vee, \gamma)}}{-2}}[t_\beta, t_\alpha]\notag\\
=&-\frac{1}{2} \sum_{\Psi\subset \Phi}
\sum_{\beta , \alpha\in S_{\Psi}^{(\gamma)}\cap \Psi^+}
\det(\beta, \gamma)_{u, v}
\frac{(\ad\frac{x_{\omega(\gamma^\vee, \alpha)}}{-2})^{p}-(\ad\frac{x_{\omega(\gamma^\vee, \beta)}}{-2})^{p}}{\ad \frac{x_{\omega(\beta^\vee, \gamma)}}{-2}}[t_\beta, t_\alpha]. 
\end{align*}
The last equality follow from the fact that 
for $\alpha, \beta\in S_{\Psi}^{(\gamma)}$,\[
\ad \frac{x_{\omega(\beta^\vee, \gamma)}}{-2}/\det(\beta, \gamma)_{u, v}
=\ad \frac{x_{\omega(\alpha^\vee, \gamma)}}{-2}/\det(\alpha, \gamma)_{u, v}. 
\]
\Omit{
The notation $\alpha<\beta$ is any total order of $S_{\Psi}^{(\gamma)}$, since we switched $\alpha$, with $\beta$. }

For $\alpha, \beta \in S_{\Psi}^{(\gamma)}$, we have $\alpha, \beta, \gamma$ are colinear. 
Therefore, there exist three integers $A, B, C\in \Z$, such that $A\alpha+B\beta+C\gamma=0$. From Examples \ref{B_2} and \ref{G_2}, we know for $\alpha, \beta\in S_{\Psi}^{(\gamma)}$, we have either $\pm (\alpha\pm \beta)=\gamma$ or $\pm(\alpha\pm \beta)=3\gamma$. 
We first deal with the case $\pm (\alpha\pm \beta)=\gamma$. Without loss of generality, we assume first that $C=1$, and $A=\pm 1$,  $B=\pm 1$. 
Then we have the equalities
\begin{align*}
&\omega(\gamma^\vee, \alpha)=\frac{-1}{B} \omega(\beta^\vee, \alpha),  \,\
\omega(\gamma^\vee, \beta)=\frac{-1}{A} \omega(\alpha^\vee, \beta), \\
&\omega(\gamma^\vee, \alpha)-\omega(\gamma^\vee, \beta)
=\frac{-1}{B}\omega(\beta^\vee, \gamma). 
\end{align*}
Using the above equalities, we compute the equation \eqref{y_eta xt} further. 
\begin{align}
\eqref{y_eta xt}=
&\frac{1}{2} \sum_{\Psi\subset \Phi}
\sum_{\beta, \alpha\in S_{\Psi}^{(\gamma)}\cap \Psi^+}
\frac{\det(\beta, \gamma)_{u, v}}{B}
\sum_{s+k=p-1}
(\ad\frac{x_{\omega(\gamma^\vee, \alpha)}}{-2})^{s} (\ad\frac{x_{\omega(\gamma^\vee, \beta)}}{-2})^{k}[t_\beta, t_\alpha] \notag\\
=&\frac{1}{2} \sum_{\Psi\subset \Phi}
\sum_{\beta, \alpha\in S_{\Psi}^{(\gamma)}\cap \Psi^+}
\frac{\det(\beta, \gamma)_{u, v}}{B}
\sum_{s+k=p-1}
(\frac{1}{B}\ad\frac{x_{\omega(\beta^\vee, \alpha)}}{2})^{s} (\frac{1}{A}\ad\frac{x_{\omega(\alpha^\vee, \beta)}}{2})^{k}[t_\beta, t_\alpha] \notag\\
=&\frac{1}{2} \sum_{\Psi\subset \Phi}
\sum_{\beta, \alpha\in S_{\Psi}^{(\gamma)}\cap \Psi^+}
\frac{\det(\beta, \gamma)_{u, v}}{B}
\sum_{s+k=p-1} 
\Big[(\frac{1}{B}\ad\frac{x_{\beta^\vee}}{-2})^{s} (t_{\beta}), 
(\frac{1}{A}\ad\frac{x_{\alpha^\vee}}{-2})^{k}(t_\alpha)\Big].  \label{eq: beta alpha}
\end{align}
Plugging the formula \eqref{eq: beta alpha} into the equation \eqref{y delta y}, we have
\begin{align}
&[\tilde{\delta}_{2m} (y(u)), y(v)]
+[y(u), \tilde{\delta}_{2m} (y(v))] \notag\\
=&\frac{1}{2}\sum_{s+k+q=2m-2}
\sum_{\gamma\in \Phi^+}\sum_{\Psi\subset \Phi}
\sum_{\beta, \alpha\in S_{\Psi}^{(\gamma)}\cap \Psi^+}
\frac{\det(\beta, \gamma)_{u, v}}{B} 
\Bigg[ \Big[(\frac{1}{B}\ad\frac{x_{\beta^\vee}}{-2})^{s} (t_{\beta}), 
(\frac{1}{A}\ad\frac{x_{\alpha^\vee}}{-2})^{k}(t_\alpha)\Big], (\ad\frac{x_{\gamma^\vee}}{-2})^{q}(t_\gamma)\Bigg] 
\notag\\
=&\frac{1}{2}\sum_{s+k+q=2m-2}
\sum_{\gamma\in \Phi^+}\sum_{\Psi\subset \Phi}
\sum_{\beta, \alpha\in S_{\Psi}^{(\gamma)}\cap \Psi^+}
\frac{\det(\beta, \gamma)_{u, v}}{B^{s+1} A^k} 
\Bigg[ \Big[(\ad\frac{x_{\beta^\vee}}{-2})^{s} (t_{\beta}), 
(\ad\frac{x_{\alpha^\vee}}{-2})^{k}(t_\alpha)\Big], (\ad\frac{x_{\gamma^\vee}}{-2})^{q}(t_\gamma)\Bigg]. 
\label{derivation4}
\end{align}
Exactly the same proof as before, using $AB\alpha+B\gamma+\beta=0$, we have 
\begin{align}
&[\tilde{\delta}_{2m} (y(u)), y(v)]+[y(u), \tilde{\delta}_{2m} (y(v))] \notag\\
\Omit{=&\frac{1}{2}\sum_{s+k+q=2m-2}
\sum_{\beta\in \Phi^+}\sum_{\Psi\subset \Phi}
\sum_{\gamma, \alpha\in S_{\Psi}^{(\beta)}\cap \Psi^+}
\frac{\det(\gamma, \beta)_{u, v}}{B} 
\Bigg[ \Big[(\frac{1}{B}\ad\frac{x_{\gamma^\vee}}{-2})^{q} (t_{\gamma}), 
(\frac{1}{AB}\ad\frac{x_{\alpha^\vee}}{-2})^{k}(t_\alpha)\Big], (\ad\frac{x_{\beta^\vee}}{-2})^{s}(t_\beta)\Bigg] 
\notag\\}
=&\frac{1}{2}\sum_{s+k+q=2m-2}
\sum_{\beta\in \Phi^+}\sum_{\Psi\subset \Phi}
\sum_{\gamma, \alpha\in S_{\Psi}^{(\beta)}\cap \Psi^+}
\frac{\det(\gamma, \beta)_{u, v}}{B^{q+k+1}A^k} 
\Bigg[ \Big[(\ad\frac{x_{\gamma^\vee}}{-2})^{q} (t_{\gamma}), 
(\ad\frac{x_{\alpha^\vee}}{-2})^{k}(t_\alpha)\Big], (\ad\frac{x_{\beta^\vee}}{-2})^{s}(t_\beta)\Bigg],  \label{eq:derivation4-2}
\end{align}
and using $A\gamma+AB\beta+\alpha=0$, one get
\begin{align}
&[\tilde{\delta}_{2m} (y(u)), y(v)]+[y(u), \tilde{\delta}_{2m} (y(v))] \notag\\
\Omit{=&\frac{1}{2}\sum_{s+k+q=2m-2}
\sum_{\alpha\in \Phi^+}\sum_{\Psi\subset \Phi}
\sum_{\gamma, \beta\in S_{\Psi}^{(\beta)}\cap \Psi^+}
\frac{\det(\beta, \alpha)_{u, v}}{AB} 
\Bigg[ \Big[(\frac{1}{AB}\ad\frac{x_{\beta^\vee}}{-2})^{s} (t_{\beta}), 
(\frac{1}{A}\ad\frac{x_{\gamma^\vee}}{-2})^{q}(t_\gamma)\Big], (\ad\frac{x_{\alpha^\vee}}{-2})^{k}(t_\alpha)\Bigg] 
\notag\\}
=&\frac{1}{2}\sum_{s+k+q=2m-2}
\sum_{\alpha\in \Phi^+}\sum_{\Psi\subset \Phi}
\sum_{\gamma, \beta\in S_{\Psi}^{(\alpha)}\cap \Psi^+}
\frac{\det(\beta, \alpha)_{u, v}}{A^{1+s+q}B^{s+1}} 
\Bigg[ \Big[(\ad\frac{x_{\beta^\vee}}{-2})^{s} (t_{\beta}), 
(\frac{1}{A}\ad\frac{x_{\gamma^\vee}}{-2})^{q}(t_\gamma)\Big], (\ad\frac{x_{\alpha^\vee}}{-2})^{k}(t_\alpha)\Bigg].   \label{eq:derivation4-3}
\end{align}
It is obvious that $\det(\beta, \gamma)_{u, v}=-\det(\gamma, \beta)_{u, v}=-A\det(\beta, \alpha)_{u, v}$. Therefore, taking into account that $s+k+q=2m-2$, the coefficients satisfy the equality
\[
\frac{\det(\beta, \gamma)_{u, v}}{B^{s+1} A^k} 
=-\frac{\det(\gamma, \beta)_{u, v}}{B^{q+k+1}A^k} 
=-\frac{\det(\beta, \alpha)_{u, v}}{A^{1+s+q}B^{s+1}}. 
\]
Taking the summation of \eqref{derivation4}, \eqref{eq:derivation4-2} and \eqref{eq:derivation4-3}, we have
\[
3\Big([\tilde{\delta}_{2m} (y(u)), y(v)]+[y(u), \tilde{\delta}_{2m} (y(v))]\Big)=0
\]
by the Jacobi identity.

When $A\alpha+B\beta+3\gamma=0$, which is the case $G_2$. 
We modify the above proof using the refined (tt) relations. Then the equation \eqref{derivation4} is modified to
\begin{align*}
&[\tilde{\delta}_{2m} (y(u)), y(v)]
+[y(u), \tilde{\delta}_{2m} (y(v))] \notag\\
=&\frac{1}{2}\sum_{s+k+q=2m-2}
\sum_{\gamma\in \Phi^+}\sum_{\Psi\subset \Phi}
\sum_{\beta, \alpha\in S_{\Psi}^{(\gamma)}\cap \Psi^+}
BC\det(\beta, \gamma)_{u, v}
\Bigg[ \Big[(\frac{1}{B}\ad\frac{x_{\beta^\vee}}{-2})^{s} (t_{\beta}), 
(\frac{1}{A}\ad\frac{x_{\alpha^\vee}}{-2})^{k}(t_\alpha)\Big], (\frac{1}{C}\ad\frac{x_{\gamma^\vee}}{-2})^{q}(t_\gamma)\Bigg]. 
\end{align*}
The rest of the proof is similar. 

\section{A principal bundle on the moduli space $\M_{1, n}$}
Let $e, f, h$ be the standard basis of $\mathfrak{sl}_2$. 
There is a Lie algebra morphism
$\mathfrak{d} \to \mathfrak{sl}_2$ defined by $\delta_{2m} \mapsto 0, d \mapsto h, X \mapsto e, \Delta_0\mapsto f.$
Let $\mathfrak{d}_+ \subset \mathfrak{d}$ be the kernel of this homomorphism. 
Since the morphism has a section, which is given by $e\mapsto X, f\mapsto \Delta_0$ and $h \mapsto d$, we have
a semidirect decomposition $\mathfrak{d}=\mathfrak{d}_+ \rtimes \mathfrak{sl}_2.$ As a consequence, we have the decomposition
\[ \Aell \rtimes \mathfrak{d}=(\Aell \rtimes \mathfrak{d}_+)\rtimes \mathfrak{sl}_2.\]
\begin{lemma}\cite[Lemma 8]{CEE}\label{lem:grading of d}
The Lie algebra $\Aell \rtimes \mathfrak{d}_+$ is positively graded.
\end{lemma}
\begin{proof}
Define a $\Z^2-$grading of $\mathfrak{d}$ and $\Aell$ by 
\begin{align*}
&\deg(\Delta_0)=(-1, 1), \,\ \deg(d)=(0, 0), \,\  \deg(X)=(1, -1), \,\ \deg(\delta_{2m})=(2m+1, 1)\\
\text{and}\,\ &\deg(x(u))=(1, 0)\,\ \deg(y(u))=(0, 1), \,\ \deg(t_\alpha)=(1, 1).
\end{align*}
It is straightforward to check this Lie algebra is positively graded. 
\end{proof}
We form the following semidirect product \[G_n:=\exp(\widehat{\Aell\rtimes \mathfrak{d}_+})\rtimes \SL_2(\C), \]
where $\widehat{\Aell\rtimes \mathfrak{d}_+}$ is the completion of $\Aell\rtimes \mathfrak{d}_+$ with respect to the grading in Lemma \ref{lem:grading of d}. 

Let $P^\vee$ be the coweight lattice, the semidirect product $(P^\vee\oplus P^\vee)\rtimes \SL_2(\Z)$ acts on $\h \times \mathfrak{H}$. 
For $(\bold{n}, \tau \bold{m}) \in (P^\vee\oplus P^\vee)$ and $(z, \tau)\in \h \times \mathfrak{H}$, the action is given by translation.
$(\bold{n}, \bold{m})*(z, \tau):=
(z+\bold{n}+\tau\bold{m}, \tau)$. For $\left(\begin{smallmatrix}
a & b \\
c & d
\end{smallmatrix}\right)\in \SL_2(\Z)$, the action is given by
$\left(\begin{smallmatrix}
a & b \\
c & d
\end{smallmatrix}\right)*(z, \tau):=(\frac{z}{c\tau+d}, \frac{a\tau+b}{c\tau+d})$. 
 
 Let $\alpha(-): \h\to \C$ be the map induced by the root $\alpha\in \h^\vee$. 
 We define $\widetilde{H}_{\alpha, \tau}\subset \h\times \mathfrak{H}$ to be
\[
\widetilde{H}_{\alpha, \tau}=\{(z, \tau)\in \h\times\mathfrak{H}\mid \alpha(z) \in \Lambda_\tau=\Z+\tau \Z\}.
\]

\begin{lemma}
The group $(P^\vee \oplus P^\vee) \rtimes \SL_2(\Z)$ preserves the hyperplane complement $\h \times \mathfrak{H}\setminus\bigcup_{\alpha\in \Phi^+, \tau\in \mathfrak{H}} \widetilde{H}_{\alpha, \tau}$.
\end{lemma}
\begin{proof}
It is clear that $\alpha( \bold{n}+\tau \bold{m} )\in \Lambda_\tau$, for any $\alpha\in \Phi$. 
Therefore, for $(z, \tau)\in \h\times \mathfrak{H}$ such that $\alpha( z) \notin \Lambda_\tau$, we have
$\alpha( z +\bold{n}+\tau \bold{m}) ) \notin \Lambda_\tau$. 

Let $(z, \tau)\in \h\times \mathfrak{H}$ be the element such that $\alpha( z) \notin \Lambda_\tau$ for any $\alpha\in \Phi$.  Suppose there exists some $\beta\in \Phi$ and $n, m\in \Z$, such that $\beta(\frac{z}{c\tau+d})=n+m \frac{a\tau+b}{c\tau+d}$. 
It is equivalent to $\beta(z)= n(c\tau+d ) +m (a\tau+b)\in \Lambda_{\tau}$, which is contradicting with the choice of $(z, \tau)$. This completes the proof. 
\end{proof}
We define the elliptic moduli space $\M_{1, n}$ to be the quotient of 
$\h \times \mathfrak{H}\setminus\bigcup_{\alpha\in \Phi^+, \tau\in \mathfrak{H}} \widetilde{H}_{\alpha, \tau}$ by the group $(P^\vee \oplus P^\vee) \rtimes \SL_2(\Z)$ action.  
Let $\pi: \h \times \mathfrak{H}\setminus\bigcup_{\alpha\in \Phi^+, \tau\in \mathfrak{H}} \widetilde{H}_{\alpha, \tau} \to \M_{1, n}$ be the natural projection.
We define a principal $G_n-$bundle $P_n$ on the elliptic moduli space $\M_{1, n}$ in this section.

For $u\in \C^*$, $u^d:=\left(\begin{smallmatrix}
u & 0 \\
0 & u^{-1}
\end{smallmatrix}\right)\in \SL_2(\C)\subset G_n$ and for $v\in \C, e^{vX}:=\left(\begin{smallmatrix}
1 & v \\
0 & 1
\end{smallmatrix}\right)\in \SL_2(\C)\subset G_n$.
\begin{prop}
\label{prop: G_n bundle}
There exists a unique principal $G_n-$bundle $P_n$ over $\M_{1, n},$ such that a section of $U \subset \M_{1, n}$ is a function
$f: \pi^{-1}(U)\to G_n$, with the properties that
\begin{align*}
&f(z+\lambda^\vee_i|\tau)=f(z|\tau), \,\ f(z+\tau\lambda^\vee_i|\tau)=e^{-2\pi ix_{\lambda_i^\vee}}f(z|\tau), \,\ f(z|\tau+1)=f(z|\tau), \\ 
&f(\frac{z}{\tau}|-\frac{1}{\tau})=\tau^d\exp(\frac{2\pi i}{\tau}(\sum_iz_i x_{\lambda_i^\vee}+X))f(z|\tau).
\end{align*}
\end{prop}
\begin{proof}  
In \cite[Proposition 10]{CEE}, \CEE constructed a principal bundle $\widetilde P_n$ on $\h \times \mathfrak{H}/((\Z^n)^2\rtimes \SL_2(\Z))$ with the structure group $G_n$. 
The vector bundle $P_n$ can be obtained by 
restricting this bundle $\widetilde P_n$ on $\M_{1, n} \subset \h \times \mathfrak{H}/((\Z^n)^2\rtimes \SL_2(\Z))$.
\end{proof}

\section{Flat connection on the elliptic moduli space}
In this section, we construct the universal flat connection on the moduli space $\mathcal{M}_{1, n}$ of (pointed) elliptic curves associated to a root system $\Phi$. This connection is an extension of the universal KZB connection $\nabla_{\KZB, \tau}$. 
Recall that in Section \S\ref{sec:conn def} \eqref{function k}, we have the function $ 
k(z, x|\tau)=\frac{\theta(z+x| \tau)}{\theta(z| \tau)\theta(x| \tau)}-\frac{1}{x} \in \Hol(\C-\Lambda_\tau)[\![x]\!]$. Let 
\[g(z, x|\tau):=k_x(z, x|\tau)=
\frac{\theta(z+x|\tau)}{\theta(z|\tau)\theta(x|\tau)} 
\left(
\frac{\theta'}{\theta}(z+x|\tau)-\frac{\theta'}{\theta}(x|\tau)\right)+\frac{1}{x^2}
\]
be the derivative of function $k(z, x|\tau)$ with respect to variable $x$. We have $g(z, x|\tau) \in \Hol(\C-\Lambda_\tau)[\![x]\!]$.

For a power series $\psi(x)=\sum_{n\geq 1}b_{2n}x^{2n}\in \C[\![x]\!]$ with positive even degrees, we consider the following two elements in $\widehat{\Aell\rtimes \mathfrak{d}}$\[
\delta_\psi:=\sum_{n\geq 1}b_{2n}\delta_{2n}, \,\ 
\Delta_\psi:=\Delta_0+\delta_\psi=\Delta_0+\sum_{n\geq 1}b_{2n}\delta_{2n}.\]
As in \cite{CEE}, we define the power series $\varphi(x)$ to be
\[
\varphi(x)=g(0, 0|\tau)-g(0, x|\tau)=
-\frac{1}{x^2}-(\frac{\theta'}{\theta})'(x|\tau)
+\Big(\frac{1}{x^2}+(\frac{\theta'}{\theta})'(x|\tau)\Big)|_{x=0}
\in \C[\![x]\!], \]
 which has positive even degrees. 
We set $a_{2n}:=-\frac{(2n+1)B_{2n+2}(2i\pi)^{2n+2}}{(2n+2)!}$, where $B_n$ are the Bernoulli numbers determined by the expansion $\frac{x}{e^x-1}=\sum_{r\geq 0}\frac{B_r}{r!}x^r$.
Then, the power series $\varphi(x)$ has the expansion $\varphi(x)=\sum_{n\geq 1}a_{2n}E_{2n+2}(\tau)x^{2n}$ for some coefficients $E_{2n+2}(\tau)$ only depend on $\tau$. 
By our convention, we have the following two elements in $\widehat{\Aell\rtimes \mathfrak{d}}$
 \[
 \delta_{\varphi}=\sum_{n\geq 1}a_{2n}E_{2n+2}(\tau)\delta_{2n}, \,\ 
\Delta_{\varphi}=\Delta_0+\delta_{\varphi}=\Delta_0+\sum_{n\geq 1}a_{2n}E_{2n+2}(\tau)\delta_{2n}. \]
Consider the following function on $\h\times \mathfrak{H}$ 
\begin{align*}
\Delta:=\Delta(\underline{\alpha}, \tau)=
&-\frac{1}{2\pi i}\Delta_{\varphi}+\frac{1}{2\pi i}\sum_{\beta \in \Phi^+}g(\beta, \ad\frac{ x_{\beta^\vee}}{2}|\tau)(t_\beta)\\
=&-\frac{1}{2\pi i}\Delta_0-\frac{1}{2\pi i}\sum_{n\geq 1}a_{2n}E_{2n+2}(\tau)\delta_{2n}+\frac{1}{2\pi i}\sum_{\beta \in \Phi^+}g(\beta, \ad\frac{ x_{\beta^\vee}}{2}|\tau)(t_\beta).
\end{align*}
This is a meromorphic function on $\C^n\times \mathfrak{H}$ valued in 
$\widehat{(\Aell\rtimes \mathfrak{d}_+)\rtimes \mathfrak{n}_+}\subset \Lie(G_n)$, where $\mathfrak{n}_+=\C\Delta_0\subset \mathfrak{sl}_2$. It has only poles along the hyperplanes $\bigcup_{\alpha\in \Phi^+, \tau\in \mathfrak{H}} \widetilde{H}_{\alpha, \tau}$.

\begin{theorem}\label{thm:moduli}
The following $\widehat{\Aell\rtimes\mathfrak{d}}$-valued KZB connection on $\M_{1, n}$ is flat.
\begin{equation}\label{conn:extension}
\nabla_{\KZB}=\nabla_{\KZB, \tau}-\Delta d\tau=
d-\Delta d\tau-\sum_{\alpha \in \Phi^+} k(\alpha, \ad\frac{x_{\alpha^\vee}}{2}|\tau)(t_\alpha)d\alpha+\sum_{i=1}^{n}
y(u^i)du_i.
\end{equation}
\end{theorem}
Note we have already shown the flatness of
$\nabla_{\KZB, \tau}$ in Theorem \ref{thm:connection flat}. 
We prove Theorem \ref{thm:moduli} in the rest of this section. 
Set $A:=\Delta d\tau+\sum_{\alpha \in \Phi^+} k(\alpha, \ad\frac{x_{\alpha^\vee}}{2}|\tau)(t_\alpha)d\alpha-\sum_{i=1}^{n}
y(u^i)du_i$, so that $\nabla_{\KZB}=d-A$.  
By definition of flatness, it suffices to show the curvature $dA+A\wedge A$  is zero.
\subsection{Proof of $dA=0$}
In this subsection, we show the differential $dA$ vanishes. We have
\begin{align*}
dA&=\frac{\partial A}{\partial \tau}d\tau+\sum_{i=1}^n\frac{\partial A}{\partial \alpha_i}d\alpha_i\\
&=\sum_{\alpha\in \Phi^+}\frac{\partial }{\partial \tau}k(\alpha, \ad\frac{x_{\alpha^\vee}}{2}|\tau)(t_\alpha)d\tau\wedge d\alpha+\sum_i\frac{\partial }{\partial \alpha_i}\Delta(\underline{\alpha}, \tau)d\alpha_i \wedge d\tau\\
&=\sum_{\alpha\in \Phi^+}\frac{\partial }{\partial \tau}k(\alpha, \ad\frac{x_{\alpha^\vee}}{2}|\tau)(t_\alpha)-\frac{1}{2\pi i}\sum_i\sum_\beta\frac{\partial }{\partial \alpha_i}g(\beta, \ad\frac{ x_{\alpha^\vee}}{2}|\tau)(t_\alpha)d\tau\wedge d\alpha_i\\
&=\sum_{\alpha\in \Phi^+}(\frac{\partial }{\partial \tau}k(\alpha, \ad\frac{x_{\alpha^\vee}}{2}|\tau)(t_\alpha)-\frac{1}{2\pi i}\frac{\partial }{\partial \alpha}g(\alpha, \ad\frac{ x_{\alpha^\vee}}{2}|\tau)(t_\alpha))d\tau\wedge d\alpha,
\end{align*}
which is 0 by the differential equation
$(\partial_\tau k)(z, x|\tau)=\frac{1}{2\pi i}(\partial_z g)(z, x|\tau)$. 
See \cite[Page 190]{CEE} for the proof of this equation.
\subsection{Simplification of $A \wedge A$}
In this subsection, we simplify the term $A\wedge A$. 
By definition, we have,
\begin{align}
A \wedge A
=&\sum_{\alpha\in \Phi^+}[\Delta(\underline{\alpha}, \tau), k(\alpha, \ad\frac{x_{\alpha^\vee}}{2}|\tau)(t_\alpha)]d\tau \wedge d\alpha-\sum_i[\Delta(\underline{\alpha}|\tau), y_{\lambda_i^\vee}]d\tau \wedge d\alpha_i \notag\\
=&-\frac{1}{2\pi i}\Bigg(\sum_{\alpha\in \Phi^+}
\Big[\Delta_\varphi-\sum_{\beta\in \Phi^+}g(\beta, \ad\frac{x_{\beta^\vee}}{2}|\tau)(t_\beta), \,\ k(\alpha, \ad\frac{x_{\alpha^\vee}}{2}|\tau)(t_\alpha)\Big]d\tau \wedge d\alpha\notag \\
&\phantom{123456}-\sum_{i=1}^n \Big[\Delta_\varphi-\sum_{\beta\in \Phi^+}g(\beta, \ad\frac{x_{\beta^\vee}}{2}|\tau)(t_\beta),  \,\ y_{\lambda_i^\vee}\Big]d\tau \wedge d\alpha_i\Bigg) \notag\\
=&-\frac{1}{2\pi i}\Bigg(\sum_{\alpha\in \Phi^+}
\Big[\Delta_\varphi, k(\alpha, \ad\frac{x_{\alpha^\vee}}{2}|\tau)(t_\alpha)\Big]d\tau \wedge d\alpha
-\sum_{\alpha, \beta \in \Phi^+}\Big[g(\beta, \ad\frac{x_{\beta^\vee}}{2}|\tau)(t_\beta) , k(\alpha, \ad\frac{x_{\alpha^\vee}}{2}|\tau)(t_\alpha)\Big]d\tau \wedge d\alpha \label{eq:A wedge A 1}\\
&\phantom{123456}-\sum_{i=1}^n [\Delta_\varphi, y_{\lambda_i^\vee}]d\tau \wedge d\alpha_i
+\sum_{i=1}^n [\sum_{\beta\in \Phi^+}g(\beta, \ad\frac{x_{\beta^\vee}}{2}|\tau)(t_\beta), y_{\lambda_i^\vee}]d\tau \wedge d\alpha_i \Bigg). \label{eq:A wedge A 2}
\end{align}
We now simplify each summand of \eqref{eq:A wedge A 1}, \eqref{eq:A wedge A 2}. In the following lemma, we rewrite the second term of  \eqref{eq:A wedge A 1}. 
\begin{lemma}\label{flat1}
Modulo the relations (tx), (xx) of $\Aell$, the following identity holds for any $\alpha, \beta\in \Phi^+$
\[
[g(\beta, \ad\frac{x_{\beta^\vee}}{2}|\tau)(t_\beta), k(\alpha, \ad\frac{x_{\alpha^\vee}}{2}|\tau)(t_\alpha)]
=g(\beta, (\ad\frac{x_{\omega(\beta^\vee, \alpha)}}{-2})|\tau)k(\alpha, (\ad\frac{x_{\omega(\alpha^\vee, \beta)}}{-2})|\tau)[t_\beta, t_\alpha].
\]
\end{lemma}
\begin{proof}
When $\alpha=\beta$, it is clear that both sides of the identity are zero, hence they are equal. We now assume $\alpha \neq \beta$. 
The desired identity follows from the relations 
\begin{align*}
&[x_{\beta^\vee}+x_{\omega(\beta^\vee, \alpha)}, t_{\beta}]=0, \,\ [x_{\omega(\beta^\vee, \alpha)}, t_{\alpha}]=0,\\
\text{and} \,\ &
[x_{\alpha^\vee}+x_{\omega(\alpha^\vee, \beta)}, t_{\alpha}]=0, \,\ [x_{\omega(\alpha^\vee, \beta)}, t_{\beta}]=0.
\end{align*}
\end{proof}
In the following lemma, 
we rewrite the second term of  \eqref{eq:A wedge A 2}. 
\begin{lemma}\label{flat2}
Modulo the relations (tx), (ty),(yx), (xx) of $\Aell$,
the following identity holds.
\begin{align*}
&\sum_{i=1}^n \Big[y(u^i), \,\ g(\beta, \ad\frac{x_{\beta^\vee}}{2}|\tau)(t_\beta)\Big]d\tau\wedge du_i
-\sum_{\beta\in \Phi^+}\Big[y(\frac{\beta^\vee}{2}), \,\ g(\beta, \ad\frac{x_{\beta^\vee}}{2}|\tau)(t_\beta)\Big]d\tau\wedge d\beta\\
=&\sum_{\gamma \neq \beta\in \Phi^+}\frac{g(\beta, \ad\frac{x_{\beta^\vee}}{2}|\tau)-g(\beta, (\ad\frac{x_{\omega(\beta^\vee, \gamma)}}{-2})|\tau)}{\ad\frac{x_{\omega(\gamma^\vee, \beta)}}{-2}}[t_\gamma, t_\beta]d\tau\wedge d\gamma\\
&-\sum_{\gamma \neq \beta\in \Phi^+}\frac{\gamma(\beta^\vee)}{2}\frac{g(\beta, \ad\frac{x_{\beta^\vee}}{2}|\tau)-g(\beta, (\ad\frac{x_{\omega(\beta^\vee, \gamma)}}{-2})|\tau)}{\ad\frac{x_{\omega(\gamma^\vee, \beta)}}{-2}}[t_\gamma, t_\beta]d\tau\wedge d\beta. 
\end{align*}
\end{lemma}
\begin{proof}
If $\beta(u^i)=0$, by Proposition \ref{prop:Ome2}, we have
\[
[y(u^i), g(\beta, \ad\frac{x_{\beta^\vee}}{2}|\tau)(t_\beta)]=
\sum_{\gamma\in \Phi^+}\gamma(\beta^\vee)\gamma(u^i)\frac{g(\beta, \ad\frac{x_{\beta^\vee}}{2}|\tau)-g(\beta, (\ad\frac{x_{\omega(\beta^\vee, \gamma)}}{-2})|\tau)}{ad x_{\beta^\vee} +ad x_{\omega(\beta^\vee, \gamma)}}[t_\gamma, t_\beta].
\]
For a fixed root $\beta$, extend $\{\beta\}$ to a basis $\{u_1, \dots, u_n\}$ of $\h^*$, with $u_1=\beta$.  Let $\{u^1, \dots, u^n\}$ be the corresponding dual basis. We then have
$u^1=\frac{\beta^\vee}{2}$,  and $\beta(u^i)=0$, for $i\neq 1$. For such choice of $\{u_i\}$, we have
\begin{align*}
&\sum_{i=1}^n[y(u^i), g(\beta, \ad\frac{x_{\beta^\vee}}{2}|\tau)(t_\beta)]d\tau\wedge du_i-[y(\frac{\beta^\vee}{2}), g(\beta, \ad\frac{x_{\beta^\vee}}{2}|\tau)(t_\beta)]d\tau\wedge d\beta
\\
=&\sum_{i=2}^n[y(u^i), g(\beta, \ad\frac{x_{\beta^\vee}}{2}|\tau)(t_\beta)]d\tau\wedge du_i\\
=&\sum_{i=1}^n\sum_{\gamma\in \Phi^+}\gamma(\beta^\vee)\gamma(u^i)\frac{g(\beta, \ad\frac{x_{\beta^\vee}}{2}|\tau)-g(\beta, (\ad\frac{x_{\omega(\beta^\vee, \gamma)}}{-2})|\tau)}{\ad x_{\beta^\vee} +\ad x_{\omega(\beta^\vee, \gamma)}}[t_\gamma, t_\beta]d\tau\wedge du_i\\
&-\sum_{\gamma\in \Phi^+}\gamma(\beta^\vee)\frac{\gamma(\beta^\vee)}{2}\frac{g(\beta, \ad\frac{x_{\beta^\vee}}{2}|\tau)-g(\beta, (\ad\frac{x_{\omega(\beta^\vee, \gamma)}}{-2})|\tau)}{\ad x_{\beta^\vee} +\ad x_{\omega(\beta^\vee, \gamma)}}[t_\gamma, t_\beta]d\tau\wedge d\beta
\\
=&\sum_{\gamma \neq \beta\in \Phi^+}\gamma(\beta^\vee)\frac{g(\beta, \ad\frac{x_{\beta^\vee}}{2}|\tau)-g(\beta, (\ad\frac{x_{\omega(\beta^\vee, \gamma)}}{-2})|\tau)}{\ad x_{\beta^\vee} +\ad x_{\omega(\beta^\vee, \gamma)}}[t_\gamma, t_\beta]d\tau\wedge d\gamma\\
&-\sum_{\gamma\in \Phi^+}\gamma(\beta^\vee)\frac{\gamma(\beta^\vee)}{2} \frac{g(\beta, \ad\frac{x_{\beta^\vee}}{2}|\tau)-g(\beta, (\ad\frac{x_{\omega(\beta^\vee, \gamma)}}{-2})|\tau)}{\ad x_{\beta^\vee} +\ad x_{\omega(\beta^\vee, \gamma)}}[t_\gamma, t_\beta]d\tau\wedge d\beta\\
=&\sum_{\gamma \neq \beta\in \Phi^+}\frac{g(\beta, \ad\frac{x_{\beta^\vee}}{2}|\tau)-g(\beta, (\ad\frac{x_{\omega(\beta^\vee, \gamma)}}{-2})|\tau)}{\ad\frac{x_{\omega(\gamma^\vee, \beta)}}{-2}}[t_\gamma, t_\beta]d\tau\wedge d\gamma\\
&-\sum_{\gamma\in \Phi^+}\frac{\gamma(\beta^\vee)}{2}\frac{g(\beta, \ad\frac{x_{\beta^\vee}}{2}|\tau)-g(\beta, (\ad\frac{x_{\omega(\beta^\vee, \gamma)}}{-2})|\tau)}{\ad\frac{x_{\omega(\gamma^\vee, \beta)}}{-2}}[t_\gamma, t_\beta]d\tau\wedge d\beta. 
\end{align*}
The last equality follow from the identity $ \frac{x_{\beta^\vee} + x_{\omega(\beta^\vee, \gamma)} }{ \gamma(\beta^\vee)}= \frac{x_{\omega(\gamma^\vee, \beta)}}{-2}$ in Lemma \ref{lemma:omeg}. 
Thus, the conclusion follows.
\end{proof}
In the following lemma, we rewrite the first term of  \eqref{eq:A wedge A 2}.
\begin{lemma}\label{flat3}
Modulo the relation $(\delta y)$ of $\Aell\rtimes \mathfrak{d}$, we have
\[
\sum_{i=1}^n[\Delta_\varphi,y(u^i)]d u_i=\sum_{\alpha \in \Phi^+}
\sum_{s}\Big[f_s\big(\ad\frac{x_{\alpha^\vee}}{2}\big)(t_\alpha), \,\ g_s\big(\ad\frac{x_{\alpha^\vee}}{-2}\big)(t_\alpha)\Big]d\tau\wedge d\alpha,
\]
where the functions $f_s$ and $g_s$ are determined by the equality $\frac{1}{2}\frac{\varphi(u)-\varphi(v)}{u-v}=\sum_sf_s(u)g_s(v)$.
\end{lemma}
\begin{proof}
Recall that by definition, we have $[\Delta_0, y(u^i)]=0$, and the action of $\delta_{2m}$ on $y(u^i)$ is given by 
\[[\delta_{2m}, y(u^i)]=\frac{1}{2}\sum_{\alpha\in \Phi^+}\alpha(u^i)\sum_{p+q=2m-1}[(\ad\frac{x_{\alpha^\vee}}{2})^{p}(t_\alpha), (\ad\frac{x_{\alpha^\vee}}{-2})^{q}(t_\alpha)]. \] 
Let $\varphi(x)=\sum_{n\geq 1} b_{2n} x^{2n}$, then,
\[
\frac{1}{2}\frac{\varphi(u)-\varphi(v)}{u-v}
=\frac{1}{2} \sum_{n\geq 1} b_{2n}\sum_{p+q=2n-1} u^p v^q=\sum_sf_s(u)g_s(v). 
\] 
Therefore, 
\begin{align*}
[\Delta_\varphi,y(u^i)]=&\delta_{\varphi}(y(u^i))
=\sum_{n\geq 1} b_{2n} \delta_{2n}(y(u^i))\\
=& \frac{1}{2}\sum_{n\geq 1} b_{2n}
\sum_{\alpha\in \Phi^+}\alpha(u^i)\sum_{p+q=2n-1}[(\ad\frac{x_{\alpha^\vee}}{2})^{p}(t_\alpha), (\ad\frac{x_{\alpha^\vee}}{-2})^{q}(t_\alpha)]\\
=&\sum_{\alpha\in \Phi^+}\alpha(u^i) \sum_{s}[ f_s(\ad\frac{x_{\alpha^\vee}}{2})(t_\alpha), g_s(\ad\frac{x_{\alpha^\vee}}{-2})(t_\alpha)].
\end{align*}
This implies the conclusion. 
\end{proof}
We now compute the first term of  \eqref{eq:A wedge A 1}. First, we have the Lemma. 
\begin{lemma}\label{flat4}
Modulo the relations $(\delta x)$ and $(\delta t)$ of $\Aell\rtimes \mathfrak{d}$, we have
\[
[\delta_{\varphi}, k(\alpha, \ad\frac{x_{\alpha^\vee}}{2}|\tau)(t_\alpha)]=
\sum_s[l_s^\alpha(\ad\frac{x_{\alpha^\vee}}{2})(t_\alpha), m_s^\alpha(\ad\frac{x_{\alpha^\vee}}{2})(t_\alpha)],
\]
where $k(\alpha, u+v)\varphi(v)=\sum_sl_s^\alpha(u)m_s^\alpha(v)$.
\end{lemma}
\begin{proof}
Let $\varphi(x)=\sum_{n\geq 1} b_{2n} x^{2n}$, and $k(x)=\sum_{m>0} a_{m}x^m$ be the expansions of $\varphi(x)$ and $k(x)$, then we have
\[
k(\alpha, u+v)\varphi(v)=\sum_{m>0} a_{m}\sum_{p+q=m}{m\choose p} \sum_{n\geq 1} b_{2n} u^m  v^{2n+q}=\sum_sl_s^\alpha(u)m_s^\alpha(v).\]
Therefore,
\begin{align*}
[\delta_{\varphi}, k(\alpha, \ad\frac{x_{\alpha^\vee}}{2}|\tau)(t_\alpha)]
=&\sum_{n\geq 1} b_{2n} k(\alpha, \ad\frac{x_{\alpha^\vee}}{2}|\tau)\delta_{2n}(t_\alpha)\\
=&\sum_{m>0} a_{m}\sum_{n\geq 1} b_{2n} (\ad\frac{x_{\alpha^\vee}}{2})^{m}[t_\alpha, 
(\ad \frac{x_{\alpha^\vee}}{2})^{2n} (t_\alpha)]\\
=&\sum_{m>0} a_{m}\sum_{n\geq 1} b_{2n} \sum_{p+q=m}{m \choose p}\Big[(\ad\frac{x_{\alpha^\vee}}{2})^{p}(t_\alpha), 
(\ad \frac{x_{\alpha^\vee}}{2})^{2n+q} (t_\alpha)\Big]\\
=&\sum_s[l_s^\alpha(\ad\frac{x_{\alpha^\vee}}{2})(t_\alpha), m_s^\alpha(\ad\frac{x_{\alpha^\vee}}{2})(t_\alpha)].
\end{align*}
This completes the proof. 
\end{proof}
We need the following Lemma for the first term of  \eqref{eq:A wedge A 1}. 
\begin{lemma}\label{flat5}
Modulo the relations $[\Delta_0, x], [\Delta_0, y],[\Delta_0, t]$, and $(yx)$ of $\Aell\rtimes \mathfrak{d}$, we have
\begin{align*}
[\Delta_0, k(\alpha, \ad\frac{x_{\alpha^\vee}}{2}|\tau)(t_\alpha)]
=&\Big[\frac{y_{\alpha^\vee}}{2}, g(\alpha, \ad\frac{x_{\alpha^\vee}}{2}|\tau)(t_\alpha)\Big]
-\sum_s\Big[h_s(\ad\frac{x_{\alpha^\vee}}{2}(t_\alpha), k_s(\ad\frac{x_{\alpha^\vee}}{2}(t_\alpha)\Big]\\
-&\sum_{\gamma\neq \alpha}\frac{k(\alpha, \ad\frac{x_{\alpha^\vee}}{2})-k(\alpha, \ad\frac{x_{\omega(\alpha^\vee, \gamma)}}{-2})-(\ad\frac{x_{\alpha^\vee}}{2}+\ad\frac{x_{\omega(\alpha^\vee, \gamma)}}{2})g(\alpha, \ad\frac{x_{\omega(\alpha^\vee, \gamma)}}{-2})}{(\ad \frac{x_{\omega(\gamma^\vee, \alpha)}}{-2})^2}[t_\gamma, t_\alpha], 
\end{align*}
where the functions $h_s, k_s$ are determined by the equality $\frac{1}{2}\Big(\frac{k(\alpha, u+v)-k(\alpha, u)-vg(\alpha, u)}{v^2}-\frac{k(\alpha, u+v)-k(\alpha, v)-ug(\alpha, v)}{u^2}\Big)=-\sum_sh_s(u)k_s(v)$.
\end{lemma}
\begin{proof}
Using the relations $[\Delta_0, t_{\alpha}]=0$ and $[\Delta_0, x_{\alpha^\vee}]=y_{\alpha^\vee}$, we have
\[
[\Delta_0, (\ad x_{\alpha^\vee})^n(t_\alpha)]
=\sum_{s=0}^{n-1}(\ad x_{\alpha^\vee})^s(\ad[\Delta_0,x_{\alpha^\vee}])(\ad x_{\alpha^\vee})^{n-1-s}(t_\alpha)
=\sum_{s=0}^{n-1}(\ad x_{\alpha^\vee})^s(\ad y_{\alpha^\vee})(\ad x_{\alpha^\vee})^{n-1-s}(t_\alpha).
\]
Using the Jacobi identity, we have the following equality, for $0\leq i\leq n-2$.
\[
(\ad x_{\alpha^\vee})^{i}(\ad [x_{\alpha^\vee}, y_{\alpha^\vee}])(\ad x_{\alpha^\vee})^{n-2-i}(t_\alpha)
=
(\ad x_{\alpha^\vee})^{i+1}(\ad y_{\alpha^\vee})(\ad x_{\alpha^\vee})^{n-2-i}(t_\alpha)
-(\ad x_{\alpha^\vee})^{i}(\ad y_{\alpha^\vee})(\ad x_{\alpha^\vee})^{n-1-i}(t_\alpha).
\]
Taking the summation over $\{i \mid 0\leq i \leq s-1\}$, we have
\begin{align*}
&\sum_{i=0}^{s-1}(\ad x_{\alpha^\vee})^{i}(\ad [x_{\alpha^\vee}, y_{\alpha^\vee}])(\ad x_{\alpha^\vee})^{n-2-i}(t_\alpha)\\
=&\sum_{i=0}^{s-1}
(\ad x_{\alpha^\vee})^{i+1}(\ad y_{\alpha^\vee})(\ad x_{\alpha^\vee})^{n-2-i}(t_\alpha)
-\sum_{i=0}^{s-1}(\ad x_{\alpha^\vee})^{i}(\ad y_{\alpha^\vee})(\ad x_{\alpha^\vee})^{n-1-i}(t_\alpha)\\
=&
(\ad x_{\alpha^\vee})^{s}(\ad y_{\alpha^\vee})(\ad x_{\alpha^\vee})^{n-1-s}(t_\alpha)
-(\ad y_{\alpha^\vee})(\ad x_{\alpha^\vee})^{n-1}(t_\alpha). 
\end{align*}
Therefore, taking the summation over $\{s\mid 0\leq s \leq n-1\}$, we have
\begin{align*}
&[\Delta_0, (\ad x_{\alpha^\vee})^n(t_\alpha)]=
\sum_{s=0}^{n-1}(\ad x_{\alpha^\vee})^s(\ad y_{\alpha^\vee})(\ad x_{\alpha^\vee})^{n-1-s}(t_\alpha)\\
=&n\ad y_{\alpha^\vee}(\ad x_{\alpha^\vee})^{n-1}(t_\alpha)+
\sum_{s=1}^{n-1}\sum_{i=0}^{s-1}(\ad x_{\alpha^\vee})^i\ad[x_{\alpha^\vee}, y_{\alpha^\vee}](\ad x_{\alpha^\vee})^{n-2-i}(t_\alpha)\\
=&n\ad y_{\alpha^\vee}(\ad x_{\alpha^\vee})^{n-1}(t_\alpha)-\sum_{s=1}^{n-1}\sum_{i=0}^{s-1}\sum_{\gamma\in \Phi^+}(\alpha^\vee, \gamma)^2(\ad x_{\alpha^\vee})^i\ad t_\gamma(\ad x_{\alpha^\vee})^{n-2-i}(t_\alpha)\\
=&n\ad y_{\alpha^\vee}(\ad x_{\alpha^\vee})^{n-1}(t_\alpha)
-\sum_{s=1}^{n-1}\sum_{i=0}^{s-1}(\alpha^\vee, \alpha)^2(\ad x_{\alpha^\vee})^i
\ad t_\alpha(\ad x_{\alpha^\vee})^{n-2-i}(t_\alpha)\\
&\phantom{1234567}-\sum_{s=1}^{n-1}\sum_{i=0}^{s-1}\sum_{\{\gamma\in \Phi^+\mid \gamma\neq \alpha\}}(\alpha^\vee, \gamma)^2(\ad x_{\alpha^\vee})^i
(\ad x_{\omega(\alpha^\vee, \gamma)})^{n-2-i}[t_{\gamma}, t_\alpha], 
\end{align*}
where the last equality follows from the (yx) relation 
$[x_{\alpha^\vee}, y_{\alpha^\vee}]=-\sum_{\gamma\in \Phi^+} (\alpha^\vee, \gamma) t_{\gamma}$. 
We now use the following identity which can be shown by induction
\[
\sum_{s=1}^{n-1}\sum_{i=0}^{s-1}a^ib^{n-2-i}=
\frac{a^n-b^n-nb^{n-1}(a-b)}{(a-b)^2}, \,\ \text{for $n\geq 2$.}
\]
The argument similar to the proof of Lemma \ref{flat4} shows
\begin{align*}
[\Delta_0, k(\alpha, &\ad\frac{x_{\alpha^\vee}}{2}|\tau)(t_\alpha)]
=\Big[\frac{y_{\alpha^\vee}}{2}, g(\alpha, \ad\frac{x_{\alpha^\vee}}{2}|\tau)(t_\alpha)\Big]
-\sum_s\Big[h_s(\ad\frac{x_{\alpha^\vee}}{2})(t_\alpha), k_s(\ad\frac{x_{\alpha^\vee}}{2})(t_\alpha)\Big]\\
&-\sum_{\{\gamma\in \Phi^+\mid \gamma\neq \alpha\}}(\alpha^\vee, \gamma)^2\frac{k(\alpha, \ad\frac{x_{\alpha^\vee}}{2})-k(\alpha, \ad\frac{x_{\omega(\alpha^\vee, \gamma)}}{-2})-(\ad\frac{x_{\alpha^\vee}}{2}+\ad\frac{x_{\omega(\alpha^\vee, \gamma)}}{2})g(\alpha, \ad\frac{x_{\omega(\alpha^\vee, \gamma)}}{-2})}{(\ad x_{\alpha^\vee}+\ad x_{\omega(\alpha^\vee, \gamma)})^2}[t_\gamma, t_\alpha].
\end{align*}
The conclusion now follows from the equality 
$\frac{x_{\alpha^\vee}+ x_{\omega(\alpha^\vee, \gamma)}}{(\alpha^\vee, \gamma)}=\frac{x_{\omega(\gamma^\vee, \alpha)}}{-2}$. 
\end{proof}
Plugging the formulas in Lemmas \ref{flat1}, \ref{flat2}, \ref{flat3}, \ref{flat4},and  \ref{flat5} into the formula  \eqref{eq:A wedge A 1}+\eqref{eq:A wedge A 2} of $A\wedge A$, we get
\begin{align}\label{AA}
-2\pi i A\wedge A=\sum_{\alpha\in \Phi^+}\sum_s[F_s^\alpha(\ad\frac{x_\alpha}{2}), G_s^\alpha(\ad\frac{x_\alpha}{2})]d\tau\wedge d\alpha+\sum_{\{\alpha, \beta \in \Phi^+\mid  \gamma\neq \alpha\}}H(\alpha, \gamma)[t_\gamma, t_\alpha] d\tau\wedge d\alpha,
\end{align}
where $\sum_s F_s(u)^\alpha G_s^\alpha(v)=-L(\alpha, u, v)$, and
\begin{align*}
L(z, u, v)=&\frac{1}{2}\frac{\varphi(u)-\varphi(v)}{u+v}+\frac{1}{2}k(z, u+v)(\varphi(u)-\varphi(v))+\frac{1}{2}(g(z, u)k(z, v)-k(z, u)g(z, v))\\
&-\frac{1}{2}\left(\frac{k(z, u+v)-k(z, u)-vg(z, u)}{v^2}-\frac{k(z, u+v)-k(z, v)-vg(z, v)}{u^2}\right).
\end{align*}
As shown in \cite{CEE}, we have $L(z, u, v)=0$. Therefore, the first summand of \eqref{AA} is zero. 
In \eqref{AA}, we have
\begin{align*}
H(\alpha, \gamma)=&-\frac{k(\alpha,\ad \frac{x_{\alpha^\vee}}{2})-k(\alpha,\ad \frac{x_{\omega(\alpha^\vee, \gamma)}}{-2})}{(\ad\frac{x_{\omega(\gamma^\vee, \alpha)}}{-2})^2}
+\frac{\gamma(\alpha^\vee)}{2}\frac{g(\alpha,\ad \frac{x_{\alpha^\vee}}{2})}{\ad\frac{x_{\omega(\gamma^\vee, \alpha)}}{-2}}\\
&+\frac{g(\gamma,\ad \frac{x_{\gamma^\vee}}{2})-g(\gamma,\ad \frac{x_{\omega(\gamma^\vee, \alpha)}}{-2})}{\ad\frac{x_{\omega(\alpha^\vee, \gamma)}}{-2}}
-g(\gamma, \ad \frac{x_{\omega(\gamma^\vee, \alpha)}}{-2})k(\alpha, \ad \frac{x_{\omega(\alpha^\vee, \gamma)}}{-2}).
\end{align*}
To show the vanishing of \eqref{AA}, we first need the following identity of the coefficients $H(\alpha, \gamma)$. 
\begin{prop}\label{prop:H}
The following identity holds
\[
H(\alpha, \gamma)-H(\alpha, \alpha+\gamma)-H(\gamma+\alpha, \gamma)+H(\gamma+\alpha, \alpha)\equiv 0.
\]
\end{prop}
\begin{proof}
We have the equalities
$x_{\omega(\gamma^\vee, \alpha)}=x_{\omega((\alpha+\gamma)^\vee, \alpha)}$,  and $x_{\omega(\alpha^\vee, \gamma)}
=x_{\omega((\alpha+\gamma)^\vee, \gamma)}$. Plugging them into the definition of $H(\alpha, \gamma)$, we can simplify as follows.
\begin{align*}
&H(\alpha, \gamma)-H(\alpha, \alpha+\gamma)-H(\gamma+\alpha, \gamma)+H(\gamma+\alpha, \alpha)\\
=&-\frac{k(\alpha,\ad \frac{x_{\omega(\alpha^\vee, \alpha+\gamma)}}{-2})-k(\alpha,\ad \frac{x_{\omega(\alpha^\vee, \gamma)}}{-2})}{(\ad\frac{x_{\omega(\gamma^\vee, \alpha)}}{-2})^2}
-\frac{g(\alpha,\ad \frac{x_{\omega(\alpha^\vee, \alpha+\gamma})}{-2})}{\ad\frac{x_{\omega(\gamma^\vee, \alpha)}}{-2}}
+\frac{k(\alpha+\gamma,\ad \frac{x_{\omega((\alpha+\gamma)^\vee, \alpha)}}{-2})-k(\alpha+\gamma,\ad \frac{x_{\omega((\alpha+\gamma)^\vee, \gamma)}}{-2})}{(\ad\frac{x_{\omega(\alpha^\vee, \alpha+\gamma)}}{-2})^2}\\
&+\frac{g(\alpha+\gamma,\ad \frac{x_{\omega((\alpha+\gamma)^\vee, \alpha})}{-2})}{\ad\frac{x_{\omega(\alpha^\vee, \alpha+\gamma)}}{-2}}
+\frac{g(\gamma,\ad \frac{x_{\omega(\gamma^\vee, \alpha+\gamma)}}{-2})-g(\gamma,\ad \frac{x_{\omega(\gamma^\vee, \alpha)}}{-2})}{\ad\frac{x_{\omega(\alpha^\vee, \gamma)}}{-2}}\\
&-g(\gamma, \ad \frac{x_{\omega(\gamma^\vee, \alpha)}}{-2})k(\alpha, \ad \frac{x_{\omega(\alpha^\vee, \gamma)}}{-2})
+g(\alpha+\gamma, \ad \frac{x_{\omega((\alpha+\gamma)^\vee, \alpha)}}{-2})k(\alpha, \ad \frac{x_{\omega(\alpha^\vee, \alpha+\gamma)}}{-2})\\
&+g(\gamma, \ad \frac{x_{\omega(\gamma^\vee, \alpha+\gamma)}}{-2})k(\alpha+\gamma, \ad \frac{x_{\omega((\alpha+\gamma)^\vee, \gamma)}}{-2})
-g(\alpha, \ad \frac{x_{\omega(\alpha^\vee, \alpha+\gamma)}}{-2})k(\alpha+\gamma, \ad \frac{x_{\omega((\alpha+\gamma)^\vee, \alpha)}}{-2}).
\end{align*}
We choose $z=\alpha, z'=\alpha+\gamma$, then $z'-z=\gamma$.
Choose $v=\frac{x_{\omega(\gamma^\vee, \alpha)}}{-2}=\frac{x_{\omega((\gamma+\alpha)^\vee, \alpha)}}{-2}$, 
$u=\frac{x_{\omega(\alpha^\vee, \alpha+\gamma)}}{-2}=-\frac{x_{\omega(\gamma^\vee, \alpha+\gamma)}}{-2}$, then,
$u+v=\frac{x_{\omega(\alpha^\vee, \gamma)}}{-2}
=\frac{x_{\omega((\alpha+\gamma)^\vee, \gamma)}}{-2}
$. Note that $k(z, x)=-k(-z, -x)$ and $g(z, x)=g(-z, -x)$.
The above expression is of the following form
\begin{align*}
H(z, z', u, v)=&\frac{k(z, u+v)-k(z, u)-vg(z, u)}{v^2}
-\frac{k(z', u+v)-k(z', v)-ug(z', v)}{u^2}\\
&+\frac{g(z'-z, -u)-g(z'-z, v)}{u+v}
-g(-z', -v)k(-z, -u)+g(-z, -u)k(-z', -v)\\
&-g(z-z', -v)k(z, u+v)+g(z'-z, -u)k(z', u+v). 
\end{align*}
It is shown in \cite{CEE} that $H(z, z', u, v)\equiv 0$. 
This implies the conclusion.
\end{proof}
To show the vanishing of the second term of \eqref{AA}, by Lemma \ref{reduction}, it suffices to show the vanishing of of \eqref{AA} for rank 2 root system $\Phi$. 
We show \eqref{AA} is zero case by case. The mail tool is the identity in Proposition \ref{prop:H}. 
\subsection{Case $A_1\times A_1$}
It is obvious that \eqref{AA} is zero since the (tt) relation $[t_{\alpha_1}, t_{\alpha_2}]=0$. 
\subsection{Case $A_2$}
According to Proposition \ref{prop:H}, we have the equality
\[H(\alpha_1, \alpha_2)=H(\alpha_1, \alpha_1+\alpha_2)-H(\alpha_1+\alpha_2, \alpha_1)-H(\alpha_1+\alpha_2, \alpha_2). \]
We use the graph
$H(\alpha_1, \alpha_2)\rightarrow\left\{
   \begin{array}{ll}
     H(\alpha_1, \alpha_1+\alpha_2) \\
     H(\alpha_1+\alpha_2, \alpha_1)\\
     H(\alpha_1+\alpha_2, \alpha_2)
   \end{array}
 \right.
$ to represent we split the term $H(\alpha_1, \alpha_2)$ according to  Proposition \ref{prop:H}. We plug it into $\sum_{\alpha\in \Phi^+}(\sum_{\gamma\neq \alpha}H(\alpha, \gamma)[t_\alpha, t_\gamma])d\tau\wedge d\alpha$. By computation, the coefficient of $d\tau\wedge d\alpha_1$ is
\begin{align*}
H(\alpha_1, \alpha_1+\alpha_2)[t_{\alpha_1}, t_{\alpha_1+\alpha_2}+t_{\alpha_2}]
+H(\alpha_1+\alpha_2, \alpha_1)[t_{\alpha_1+\alpha_2}+t_{\alpha_2}, t_{\alpha_1}]
+H(\alpha_1+\alpha_2, \alpha_2)[t_{\alpha_1+\alpha_2}+t_{\alpha_1}, t_{\alpha_2}].
\end{align*}
It vanishes because of the (tt) relations. 

We now use the graph $H(\alpha_2, \alpha_1)\rightarrow\left\{
   \begin{array}{ll}
     H(\alpha_2, \alpha_1+\alpha_2) \\
     H(\alpha_1+\alpha_2, \alpha_2)\\
     H(\alpha_1+\alpha_2, \alpha_1)
   \end{array}
 \right.
$. In this case, the coefficient of $d\tau\wedge d\alpha_2$ becomes
\begin{align*}
H(\alpha_2, \alpha_1+\alpha_2)[t_{\alpha_2}, t_{\alpha_1+\alpha_2}+t_{\alpha_1}]
+H(\alpha_1+\alpha_2, \alpha_1)[t_{\alpha_1+\alpha_2}+t_{\alpha_2}, t_{\alpha_1}]
+H(\alpha_1+\alpha_2, \alpha_2)[t_{\alpha_1+\alpha_2}+t_{\alpha_1}, t_{\alpha_2}].
\end{align*}
This also vanishes because of the (tt) relations. 
Therefore,  \eqref{AA} is zero in the case of $A_2$.
\subsection{Case $B_2$}
According to Proposition \ref{prop:H}, we rewrite $H(\alpha, \gamma)$ following the graphs
\[H(\alpha_1, \alpha_2)\rightarrow\left\{
   \begin{array}{ll}
     H(\alpha_1, \alpha_1+\alpha_2) \\
     H(\alpha_1+\alpha_2, \alpha_1)\\
     H(\alpha_1+\alpha_2, \alpha_2)
   \end{array}
 \right.
\text{and} \,\ H(\alpha_1+2\alpha_2, \alpha_2)\rightarrow\left\{
   \begin{array}{ll}
     H(\alpha_1+2\alpha_2, \alpha_1+\alpha_2) \\
     H(\alpha_1+\alpha_2, \alpha_1+2\alpha_2)\\
     H(\alpha_1+\alpha_2, \alpha_2)
   \end{array}
 \right.
\]
We plug them into $\sum_{\alpha\in \Phi^+}(\sum_{\gamma\neq \alpha}H(\alpha, \gamma)[t_\alpha, t_\gamma])d\tau\wedge d\alpha$. 
By computation, the coefficient of $d\tau\wedge d\alpha_1$ is
\begin{align*}
&H(\alpha_1, \alpha_1+\alpha_2)[t_{\alpha_1}, t_{\alpha_1+\alpha_2}+t_{\alpha_2}]
+H(\alpha_1+\alpha_2, \alpha_1)[t_{\alpha_1+\alpha_2}+t_{\alpha_2}, t_{\alpha_1}]
+H(\alpha_1+\alpha_2, \alpha_1+2\alpha_2)[t_{\alpha_1+\alpha_2}+t_{\alpha_2}, t_{\alpha_1+2\alpha_2}]\\
&+H(\alpha_1+\alpha_2, \alpha_2)[t_{\alpha_1+\alpha_2}+t_{\alpha_1}+t_{\alpha_1+2\alpha_2}, t_{\alpha_2}]
+H(\alpha_1+2\alpha_2, \alpha_1+\alpha_2)[t_{\alpha_1+2\alpha_2}, t_{\alpha_1+\alpha_2}+t_{\alpha_2}]. 
\end{align*}
It vanishes because of the (tt) relations of root system $B_2$. 
 
Similarly, we rewrite $H(\alpha, \gamma)$ following the graphs, using the identity in Proposition \ref{prop:H}.
\[
H(\alpha_2, \alpha_1)\rightarrow\left\{
   \begin{array}{ll}
     H(\alpha_2, \alpha_1+\alpha_2) \\
     H(\alpha_1+\alpha_2, \alpha_2)\\
     H(\alpha_1+\alpha_2, \alpha_1)
   \end{array}
 \right.
 , \,\ H(\alpha_2, \alpha_1+2\alpha_2)\rightarrow\left\{
   \begin{array}{ll}
     H(\alpha_1+2\alpha_2, \alpha_2) \\
     H(\alpha_1+2\alpha_2, \alpha_1+\alpha_2)\\
     H(\alpha_2, \alpha_1+\alpha_2)
   \end{array}
 \right.
 \]
and $H(\alpha_1+\alpha_2, \alpha_1+2\alpha_2)\rightarrow\left\{
   \begin{array}{ll}
     H(\alpha_1+2\alpha_2, \alpha_1+\alpha_2) \\
     H(\alpha_1+2\alpha_2, \alpha_2)\\
     H(\alpha_1+\alpha_2, \alpha_2)
   \end{array}
 \right.
$. By computation, the coefficient of $d\tau\wedge d\alpha_2$ is
\begin{align*}
&H(\alpha_2, \alpha_1+\alpha_2)[t_{\alpha_2}, t_{\alpha_1+\alpha_2}+t_{\alpha_1}+t_{\alpha_1+2\alpha_2}]
+H(\alpha_1+2\alpha_2, \alpha_1+\alpha_2)[t_{\alpha_1+2\alpha_2}, t_{\alpha_1+\alpha_2}+t_{\alpha_2}]\\
&+2H(\alpha_1+2\alpha_2, \alpha_2)[t_{\alpha_1+2\alpha_2}, t_{\alpha_2}+t_{\alpha_1+\alpha_2}]
+H(\alpha_1+\alpha_2, \alpha_1)[t_{\alpha_1+\alpha_2}+t_{\alpha_2}, t_{\alpha_1}]\\
&+H(\alpha_1+\alpha_2, \alpha_2)[t_{\alpha_1+2\alpha_2}+t_{\alpha_1+\alpha_2}+t_{\alpha_1},t_{\alpha_2}]. 
\end{align*}
It vanishes because of the (tt) relations of root system $B_2$. This shows \eqref{AA} is zero for $B_2$. 
\subsection{Case $G_2$}
According to Proposition \ref{prop:H}, we rewrite $H(\alpha, \gamma)$ following the graphs in order. 
\begin{align*}
&H(\alpha_1, \alpha_1+3\alpha_2)\rightarrow\left\{
   \begin{array}{ll}
     H(\alpha_1, 2\alpha_1+3\alpha_2) \\
     H(2\alpha_1+3\alpha_2, \alpha_1)\\
     H(2\alpha_1+3\alpha_2, \alpha_1+3\alpha_2)
   \end{array}
 \right.
\,\ \,\ H(\alpha_1, \alpha_2)\rightarrow\left\{
   \begin{array}{ll}
    H (\alpha_1, \alpha_1+\alpha_2) \\
     H(\alpha_1+\alpha_2, \alpha_1)\\
     H(\alpha_1+\alpha_2, \alpha_2)
   \end{array}
 \right.\\
&H(\alpha_1+3\alpha_2, \alpha_1)\rightarrow\left\{
   \begin{array}{ll}
     H(\alpha_1+3\alpha_2, 2\alpha_1+3\alpha_2) \\
     H(2\alpha_1+3\alpha_2, \alpha_1)\\
     H(2\alpha_1+3\alpha_2, \alpha_1+3\alpha_2)
   \end{array}
 \right.
\,\  \,\ H(\alpha_1+3\alpha_2, \alpha_2)\rightarrow\left\{
   \begin{array}{ll}
     H(\alpha_1+3\alpha_2, \alpha_1+2\alpha_2) \\
    H(\alpha_1+2\alpha_2, \alpha_2)\\
     H(\alpha_1+2\alpha_2, \alpha_1+3\alpha_2)
   \end{array}
 \right.
\\
&H(\alpha_1+\alpha_2, 2\alpha_1+3\alpha_2)\rightarrow\left\{
   \begin{array}{ll}
     H(\alpha_1+\alpha_2, \alpha_1+2\alpha_2) \\
     H(2\alpha_1+3\alpha_2, \alpha_1+2\alpha_2)\\
     H(2\alpha_1+3\alpha_2, \alpha_1+\alpha_2)
   \end{array}
 \right.\\
&  H(\alpha_1+2\alpha_2, 2\alpha_1+3\alpha_2)\rightarrow\left\{
   \begin{array}{ll}
    H(\alpha_1+2\alpha_2, \alpha_1+\alpha_2) \\
     H(2\alpha_1+3\alpha_2, \alpha_1+2\alpha_2)\\
     H(2\alpha_1+3\alpha_2, \alpha_1+\alpha_2)
   \end{array}
 \right. \,\ H(\alpha_1+\alpha_2, \alpha_2)\rightarrow\left\{
   \begin{array}{ll}
     H(\alpha_1+\alpha_2, \alpha_1+2\alpha_2) \\
     H(\alpha_1+2\alpha_2, \alpha_2)\\
     H(\alpha_1+2\alpha_2, \alpha_1+\alpha_2)
   \end{array}
 \right.
\end{align*} 
\Omit{and 
$H(\alpha_1+\alpha_2, \alpha_2)\rightarrow\left\{
   \begin{array}{ll}
     H(\alpha_1+\alpha_2, \alpha_1+2\alpha_2) \\
     H(\alpha_1+2\alpha_2, \alpha_2)\\
     H(\alpha_1+2\alpha_2, \alpha_1+\alpha_2)
   \end{array}
 \right.
$. }We plug them into $\sum_{\alpha\in \Phi^+}(\sum_{\gamma\neq \alpha}H(\alpha, \gamma)[t_\alpha, t_\gamma])d\tau\wedge d\alpha$. 
By computation, the coefficient of $d\tau\wedge d\alpha_1$ is
\begin{align*}
&H(\alpha_1, \alpha_1+\alpha_2)[t_{\alpha_1}, t_{\alpha_1+\alpha_2}+t_{\alpha_2}]
+H(\alpha_1, 2\alpha_1+3\alpha_2)[t_{\alpha_1}, t_{2\alpha_1+3\alpha_2}+t_{\alpha_1+3\alpha_2}]\\
&+H(\alpha_1+\alpha_2, \alpha_1)[t_{\alpha_1+\alpha_2}+t_{\alpha_2}, t_{\alpha_1}]
+H(\alpha_1+\alpha_2, \alpha_1+2\alpha_2)[t_{\alpha_1+\alpha_2}, t_{\alpha_1+2\alpha_2}+t_{2\alpha_1+3\alpha_2}+t_{\alpha_1}+t_{\alpha_2}]\\
&+H(\alpha_1+2\alpha_2, \alpha_2)
[t_{\alpha_1+2\alpha_2}+t_{\alpha_1+\alpha_2}+t_{\alpha_1}+t_{\alpha_1+3\alpha_2},t_{\alpha_2}]\\&
+H(\alpha_1+2\alpha_2, \alpha_1+\alpha_2)[t_{\alpha_1+2\alpha_2}+t_{2\alpha_1+3\alpha_2}+t_{\alpha_1}+t_{\alpha_2}, t_{\alpha_1+\alpha_2}]\\
&+H(\alpha_1+2\alpha_2, \alpha_1+3\alpha_2)
[t_{\alpha_1+2\alpha_2}+t_{\alpha_2}, t_{\alpha_1+3\alpha_2}]
+H(\alpha_1+3\alpha_2, \alpha_1+2\alpha_2)
[t_{\alpha_1+3\alpha_2}, t_{\alpha_1+2\alpha_2}+t_{\alpha_2}]\\
&+H(\alpha_1+3\alpha_2, 2\alpha_1+3\alpha_2)
[t_{\alpha_1+3\alpha_2}, t_{2\alpha_1+3\alpha_2}+t_{\alpha_1}]
+2H(2\alpha_1+3\alpha_2, \alpha_1)[t_{\alpha_1+3\alpha_2}+t_{2\alpha_1+3\alpha_2}, t_{\alpha_1}]\\
&+H(2\alpha_1+3\alpha_2, \alpha_1+\alpha_2)[t_{2\alpha_1+3\alpha_2}, t_{\alpha_1+\alpha_2}+t_{\alpha_1+2\alpha_2}]
+H(2\alpha_1+3\alpha_2, \alpha_1+2\alpha_2)[t_{2\alpha_1+3\alpha_2}, t_{\alpha_1+\alpha_2}+t_{\alpha_1+2\alpha_2}]\\
&+2H(2\alpha_1+3\alpha_2, \alpha_1+3\alpha_2)[t_{\alpha_1}+t_{2\alpha_1+3\alpha_2}, t_{\alpha_1+3\alpha_2}].
\end{align*}
The coefficient vanishes because of the (tt) relations of root system $G_2$. 

Similarly, we rewrite $H(\alpha, \gamma)$ according to Proposition \ref{prop:H} by following the graphs in order. 
\begin{align*}
&3H(\alpha_1+3\alpha_2, \alpha_1)\rightarrow\left\{
   \begin{array}{ll}
     3H(\alpha_1+3\alpha_2, 2\alpha_1+3\alpha_2) \\
     3H(\alpha_1+3\alpha_2, \alpha_1)\\
     3H(\alpha_1+3\alpha_2, \alpha_1+3\alpha_2)
   \end{array}
 \right.
\,\ \,\ H(\alpha_2, \alpha_1)\rightarrow\left\{
   \begin{array}{ll}
     H(\alpha_2, \alpha_1+\alpha_2) \\
     H(\alpha_1+\alpha_2, \alpha_2)\\
     H(\alpha_1+\alpha_2, \alpha_1)
   \end{array}
 \right.
\\
&H(2\alpha_1+3\alpha_2, \alpha_1+\alpha_2)\rightarrow\left\{
   \begin{array}{ll}
     H(2\alpha_1+3\alpha_2, \alpha_1+2\alpha_2) \\
     H(\alpha_1+\alpha_2, \alpha_1+2\alpha_2)\\
     H(\alpha_1+\alpha_2, 2\alpha_1+3\alpha_2)
   \end{array}
 \right.\\
&2H(2\alpha_1+3\alpha_2, \alpha_1+\alpha_2)\rightarrow\left\{
   \begin{array}{ll}
     2H(2\alpha_1+3\alpha_2, \alpha_1+2\alpha_2) \\
     2H(\alpha_1+2\alpha_2, \alpha_1+\alpha_2)\\
     2H(\alpha_1+2\alpha_2, 2\alpha_1+3\alpha_2)
   \end{array}
 \right.
\end{align*}
\begin{align*}
&H(\alpha_1+2\alpha_2, \alpha_2)\rightarrow\left\{
   \begin{array}{ll}
     H(\alpha_1+2\alpha_2, \alpha_1+\alpha_2) \\
     H(\alpha_1+\alpha_2, \alpha_2)\\
     H(\alpha_1+\alpha_2, \alpha_1+2\alpha_2)
   \end{array}
 \right.
&&H(\alpha_1+2\alpha_2, \alpha_2)\rightarrow\left\{
   \begin{array}{ll}
     H(\alpha_1+2\alpha_2, \alpha_1+\alpha_2) \\
     H(\alpha_2, \alpha_1+\alpha_2)\\
     H(\alpha_2, \alpha_1+2\alpha_2)
   \end{array}
 \right.
\\
&H(\alpha_1+3\alpha_2, \alpha_2)\rightarrow\left\{
   \begin{array}{ll}
     H(\alpha_1+3\alpha_2, \alpha_1+2\alpha_2) \\
     H(\alpha_1+2\alpha_2, \alpha_2)\\
     H(\alpha_1+2\alpha_2, \alpha_1+3\alpha_2)
   \end{array}
 \right.
&&H(\alpha_1+\alpha_2, \alpha_2)\rightarrow\left\{
   \begin{array}{ll}
     H(\alpha_1+\alpha_2, \alpha_1+2\alpha_2) \\
     H(\alpha_1+2\alpha_2, \alpha_2)\\
     H(\alpha_1+2\alpha_2, \alpha_1+\alpha_2)
   \end{array}
 \right.
\\
&
H(\alpha_2, \alpha_1+3\alpha_2)\rightarrow\left\{
   \begin{array}{ll}
     H(\alpha_2, \alpha_1+2\alpha_2) \\
     H(\alpha_1+2\alpha_2, \alpha_1+3\alpha_2)\\
     H(\alpha_1+2\alpha_2, \alpha_2)
   \end{array}
 \right.
&&H(\alpha_2, \alpha_1+\alpha_2)\rightarrow\left\{
   \begin{array}{ll}
      H(\alpha_1+2\alpha_2, \alpha_1+\alpha_2)\\
     H(\alpha_1+2\alpha_2, \alpha_2)\\
     H(\alpha_2, \alpha_1+2\alpha_2)
   \end{array}
 \right.
\end{align*}
Similar computation shows the coefficient of $d\tau\wedge d\alpha_2$ is
\begin{align*}
&H(\alpha_2, \alpha_1+2\alpha_2)[t_{\alpha_2}, t_{\alpha_1+2\alpha_2}-t_{\alpha_1+2\alpha_2}+t_{\alpha_1+3\alpha_2}-t_{\alpha_1+3\alpha_2}]
+H(\alpha_1+\alpha_2, \alpha_1)[t_{\alpha_1+\alpha_2}+t_{\alpha_2}, t_{\alpha_1}]\\
&+H(\alpha_1+\alpha_2, \alpha_1+2\alpha_2)[t_{\alpha_1+\alpha_2}+t_{2\alpha_1+3\alpha_2}+
t_{\alpha_2}+t_{\alpha_1+3\alpha_2},t_{\alpha_1+2\alpha_2}]\\&
+H(\alpha_1+\alpha_2, 2\alpha_1+3\alpha_2)[t_{\alpha_1+\alpha_2}, t_{2\alpha_1+3\alpha_2}-t_{2\alpha_1+3\alpha_2}]\\
&+2H(\alpha_1+2\alpha_2, \alpha_1+\alpha_2)[t_{\alpha_1+2\alpha_2}, t_{\alpha_1+\alpha_2}+t_{2\alpha_1+3\alpha_2}+
t_{\alpha_2}+t_{\alpha_1+3\alpha_2}]\\&
+2H(\alpha_1+2\alpha_2, \alpha_1+3\alpha_2)[t_{\alpha_1+2\alpha_2}+t_{\alpha_2}, t_{\alpha_1+3\alpha_2}]\\
&+2H(\alpha_1+2\alpha_2, 2\alpha_1+3\alpha_2)[t_{\alpha_1+2\alpha_2}+t_{\alpha_1+\alpha_2}, t_{2\alpha_1+3\alpha_2}]\\
&+3H(\alpha_1+3\alpha_2, \alpha_1+2\alpha_2)[t_{\alpha_1+3\alpha_2}, t_{\alpha_1+2\alpha_2}+
t_{\alpha_2}]
+3H(\alpha_1+3\alpha_2, 2\alpha_1+3\alpha_2)[t_{\alpha_1+3\alpha_2}, t_{2\alpha_1+3\alpha_2}+t_{\alpha_1}]\\
&+3H(2\alpha_1+3\alpha_2, \alpha_1)[t_{2\alpha_1+3\alpha_2}+t_{\alpha_1+3\alpha_2}, t_{\alpha_1}]
+3H(2\alpha_1+3\alpha_2, \alpha_1+2\alpha_2)[t_{2\alpha_1+3\alpha_2}, t_{\alpha_1+2\alpha_2}+ t_{\alpha_1+\alpha_2}]\\
&+3H(2\alpha_1+3\alpha_2, \alpha_1+3\alpha_2)[t_{2\alpha_1+3\alpha_2}+t_{\alpha_1}, t_{\alpha_1+3\alpha_2}].
\end{align*}
This coefficient vanishes because of the (tt) relations of root system $G_2$. 
Thus, \eqref{AA} is zero in the case of $G_2$.

\section{The elliptic connection valued in rational Cherednik algebras}
\label{sec:rca}

In this section, we specialise the coefficients of the connections $\nabla_{\KZB, \tau}$
\eqref{conn any type} and $\nabla_{\KZB}$ \eqref{conn:extension} to the rational Cherednik
algebra of a Weyl group $W$. This specialisation is the elliptic analogue of the Coxeter
KZ connection valued in the group algebra $\IC W$, and Cherednik's affine KZ connection
valued in the degenerate affine Hecke algebra of $W$.

\subsection{The rational Cherednik algebra}\label{ss:rca}

In this subsection, we review some basic facts about the rational Cherednik algebras.
For details, see \cite{EG}, \cite{EM}. 

Let $W$ be a Weyl group and $\h$ be its complexified reflection representation. Let
$S\subset W$ be the set of reflections and, for any $s\in S$, fix $\alpha_s\in \h^*$, such
that $s(\alpha_s)=-\alpha_s$. Let $K$ be the vector space of $W$--invariant functions $S\to\C$,
and $\wt{K}=K\oplus\C$. We denote the standard linear functions on $\wt{K}$ by $\{c_s
\}_{s\in S/W}$ and $\hbar$.

\begin{definition}
The rational Cherednik algebra $H_{\hbar, c}$ is the quotient of the algebra
$\C W \ltimes T(\h\oplus \h^*)[\wt{K}]$ by the ideal generated by the relations
\[[x, x'] = 0, \,\  [y, y'] = 0, \,\ 
[y, x]=\hbar\langle y, x\rangle-\sum_{s\in S}c_s\langle \alpha_s, y\rangle\langle \alpha_s^\vee, x\rangle s,\]
where $x, x'\in \h^*$, $y, y'\in \h$. The algebra is $\IN\times\IN$--graded by
$\deg(x)=(1,0)$, $x\in\h$, $\deg(y)=(0,1)$, $y\in\h^*$, $\deg(w)=0$, $w\in W$,
$\deg(\hbar)=(1,1)$ and $\deg(c_s)=(1,1)$.

\end{definition}
Let $\{y_i \mid 1\leq i \leq n\}$ be the basis of $\h$, and $\{x_i \mid 1\leq i \leq n\}$ be the corresponding dual basis of $\h^*$. We have the following elements of the rational Cherednik algebras.
\begin{align}\label{formula of hef}
\textbf{h}:=\sum_{i=1}^n x_iy_i+\frac{1}{2}\dim \h-\sum_{s\in S}c_s s=\sum_{i=1}^n \frac{x_iy_i+y_ix_i}{2},\,\ 
\textbf{E}:=-\frac{1}{2}\sum_{i=1}^n x_i^2, \,\ \text{and} \,\, \textbf{F}:=\frac{1}{2}\sum_{i=1}^n y_i^2.
\end{align}
The following properties of $\textbf{h}$, $\textbf{E}$ and $\textbf{F}$ can be found in \cite[Prop. 3.18, 3.19]{EM}.

\begin{prop}
\label{prop:RCA hEF}
The following holds. 
\begin{romenum}
\item For any $x\in \h^*, y\in \h$, $[\textbf{h}, x]=\hbar x, [\textbf{h}, y]=-\hbar y$.
\item The elements $\frac{1}{\hbar}\textbf{h}, \frac{1}{\hbar}\textbf{E}, \frac{1}{\hbar}\textbf{F}$ form an $\mathfrak{sl}_2$-triple.
\end{romenum}
\end{prop}

\subsection{Specialising the KZB connection $\nabla_{\KZB, \tau}$ to the rational Cherednik algebra}

In this subsection, we specialize the KZB connection $\nabla_{\KZB, \tau}$ to the rational Cherednik
algebra $H_{\hbar, c}$ by constructing a homomorphism from the Lie algebra $\Aell$ to $H_{\hbar, c}$.

Let $\tilde{\alpha}$ be the highest root of Lie algebra $\g$. Write $\tilde{\alpha}^\vee=\sum_{i=1}^ng_i \alpha_i^\vee$, and let
$h^\vee=1+\sum_{i=1}^ng_i$ be the dual Coxeter number of $\g$.
\begin{prop}\label{prop:map to H_{h,c}}
For any $a, b\in \C$, there is a bigraded Lie algebra homomorphism
$\xi_{a, b}: \Aell \to H_{\hbar, c}$,  defined as follows
\begin{gather*}
x(v)\mapsto a\pi(v), \,\ y(u)\mapsto bu, \,\ 
t_\gamma\mapsto ab\left(\frac{\hbar}{h^\vee}-\frac{2c_{s_\gamma}}{(\gamma|\gamma)}s_{\gamma}\right),
\end{gather*}
for $u, v\in \h$ and $\gamma\in \Phi^+$, 
where $\pi: \h\to \h^*$ is the isomorphism induced by the non-degenerate bilinear form $(\cdot|\cdot)$ on $\h$.
and $s_{\gamma}$ is the reflection corresponding to the root $\gamma$. 
\end{prop}
For simplicity and without lost of generality we assume $a=b=1$. To prove the Proposition, we need to show the map $\xi_{a, b}$ respects all defining relations of $\Aell$. The only non-trivial relation to check is the (yx) relation
\[
[y(u), x(v)]=\sum_{\gamma\in \Phi^+}\langle v, \gamma\rangle\langle u, \gamma\rangle t_\gamma.
\]
In the remainder of this subsection, we check $\xi$ preserves the (yx) relation of $\Aell$.
\begin{lemma}\label{equal:inn}
The following equality holds
\[(\cdot|\cdot)=\frac{1}{h^\vee}\sum_{\gamma\in \Phi^+}\langle \cdot, \gamma\rangle\langle \cdot, \gamma\rangle.\]
\end{lemma}
\begin{proof}
Both $(\cdot|\cdot):\h\times\h\to \C$, and $\sum_{\gamma\in \Phi^+}\langle \cdot, \gamma\rangle\langle \cdot, \gamma\rangle:\h\times\h\to \C$
are two symmetric bilinear forms on $\h$.  They are both positive definite, and invariant under the Weyl group $W$. Note that there is only one such bilinear form up to constants. Therefore, there exists some constant $k\in \Q$, such that 
\[(\cdot|\cdot)=k\sum_{\gamma\in \Phi^+}\langle \cdot, \gamma\rangle\langle \cdot, \gamma\rangle.\]
It remains to show the constant $k$ is the same as $\frac{1}{h^\vee}$. 
It follows from \cite[Lemma1.2]{L} that
$\sum_{\gamma\in \Phi^+}\langle \tilde{\alpha}^\vee, \gamma\rangle\langle \tilde{\alpha}^\vee, \gamma\rangle=2h^\vee$. 
On the other hand, $(\tilde{\alpha}^\vee | \tilde{\alpha}^\vee)=2$. This gives $k=\frac{1}{h^\vee}$. 
Note that in the case of type $A_n$, we have $k=\frac{1}{n+1}$, which coincides with the constant in \cite{CEE}.
\end{proof}
\begin{proof}[Proof of Proposition \ref{prop:map to H_{h,c}}]
We now use Lemma \ref{equal:inn} to prove Proposition \ref{prop:map to H_{h,c}}.
Under the map $\xi_{a, b}$, we have
\begin{align*}
&[y(u), x(v)]\mapsto [u, \pi(v)]
=\hbar\langle u, \pi(v)\rangle-\sum_{s\in S}c_s\langle \alpha_s, u\rangle\langle \alpha_s^\vee, \pi(v)\rangle s
=\hbar(u|v)-\sum_{\gamma\in \Phi^+}c_s\langle \gamma, u\rangle\langle \gamma^\vee, \pi(v)\rangle s_{\gamma}
\\
&\sum_{\gamma\in \Phi^+}\langle v, \gamma\rangle\langle u, \gamma\rangle t_\gamma \mapsto
\sum_{\gamma\in \Phi^+}\langle v, \gamma\rangle\langle u, \gamma\rangle(\frac{\hbar}{h^\vee}-\frac{2c_s}{(\gamma|\gamma)}s_{\gamma})
=\hbar(u|v)-
\sum_{\gamma\in \Phi^+}c_s\langle \pi(v), \gamma^\vee\rangle\langle u, \gamma\rangle s_{\gamma}. 
\end{align*}
Therefore,  $\xi_{a, b}([y(u), x(v)])=\xi_{a, b} (\sum_{\gamma\in \Phi^+}\langle v, \gamma\rangle\langle u, \gamma\rangle t_\gamma)$. 
This completes the proof. 
\end{proof}
\begin{theorem}
\label{thm:KZB for rca}
The universal KZB connection specializes to the following 
elliptic KZ connection valued in the rational Cherednik algebra
\[
\nabla_{H_{\hbar, c}}=d+\sum_{\alpha \in \Phi^+}\frac{2c_{\alpha}}{(\alpha|\alpha)}k(\alpha, \ad(\frac{ \alpha^\vee}{2})|\tau) s_{\alpha} d\alpha
   -\sum_{\alpha \in \Phi^+} \frac{\hbar}{h^\vee} \frac{\theta'(\alpha|\tau)}{\theta(\alpha|\tau)} d\alpha
   +\sum_{i=1}^{n} u^i du_i.
\] The elliptic connection is flat and $W$-equivariant.
\end{theorem}
\begin{proof}
By Proposition \ref{prop:map to H_{h,c}}, the universal KZB connection specializes to the following connection
\[
\nabla_{H_{\hbar, c}}=d-\sum_{\alpha \in \Phi^+} k(\alpha, \ad(\frac{\alpha^\vee}{2})|\tau)\left(\frac{\hbar}{h^\vee}-\frac{2c_{s_\alpha}}{(\alpha|\alpha)}s_{\alpha}\right)d\alpha+\sum_{i=1}^{n} u^idu_i.
\]
We now simplify the above expression. Note that $k(z, 0|\tau)=\frac{\theta'(z| \tau)}{\theta(z| \tau)}$. Therefore, 
\begin{align*}
k(\alpha, \ad(\frac{x_{\alpha^\vee}}{2})|\tau)(\frac{\hbar}{h^\vee}-\frac{2c_s}{(\alpha|\alpha)}s_{\alpha})
=&k(\alpha, \ad(\frac{x_{\alpha^\vee}}{2})|\tau) \Big(\frac{\hbar}{h^\vee}\Big)
   -k(\alpha, \ad(\frac{x_{\alpha^\vee}}{2})|\tau)\Big(\frac{2c_s}{(\alpha|\alpha)}s_{\alpha}\Big)\\
=&\frac{\hbar}{h^\vee} \frac{\theta'(\alpha|\tau)}{\theta(\alpha|\tau)} 
   -k(\alpha, \ad(\frac{x_{\alpha^\vee}}{2})|\tau)\Big(\frac{2c_s}{(\alpha|\alpha)}s_{\alpha}\Big). 
\end{align*}
This completes the proof. 
\end{proof}
\subsection{The monodromy of $\nabla_{H_{\hbar, c}}$}
The connection $\nabla_{H_{\hbar, c}}$ is flat and W-equivariant. Its monodromy yields a one parameter family of monodromy representations of the elliptic braid group
$\Bell=\pi_{1}^{\orb}(T_{\reg}/W)$. The double affine Hecke algebra $\mathbb{H}_\g$ is the quotient of the group algebra $\C[\Bell]$ by the quadratic relations \[
(S_i-q t_i)(S_i+qt_i^{-1})=0,
\] where $q=e^{\pi i\frac{\hbar}{h^\vee}}$, and $t_i=e^{-\pi i \frac{2 c_{i}}{(\alpha_i, \alpha_i)}}$.
\begin{prop}
The monodromy of the elliptic connection $\nabla_{H_{\hbar, c}}$ factors through  the double affine Hecke algebra $\mathbb{H}_\g$. 
\end{prop} 


The monodromy of $\nabla_{H_{\hbar, c}}$ gives an algebra homomorphism
\[
\mathbb{H}_\g\to \widehat{H_{\hbar, c}}, 
\]  
where $\widehat{H_{\hbar, c}}$ is the completion of $H_{\hbar, c}$ with respect to the $\IN$--grading on $H_{\hbar, c}$.
Furthermore, it induces an isomorphism between the completion $\widehat{\mathbb{H}_\g}$ and $\widehat{H_{\hbar, c}}$.
Indeed, $\widehat{\mathbb{H}_\g}$ and $\widehat{H_{\hbar, c}}$ are both flat deformations of $\C[[\h]]\otimes \C[[\h^*]]\otimes \C W$. Thus, the algebra homomorphism 
$\widehat{\mathbb{H}_\g}\to \widehat{H_{\hbar, c}}$ is an isomorphism.

\subsection{Specialisation of $\nabla_{\KZB}$ to the rational Cherednik algebra}
In this subsection, we specialize the KZB connection $\nabla_{\KZB}$ to the rational Cherednik algebra $H_{\hbar, c}$ by constructing a homomorphism from the Lie algebra $\Aell\rtimes \mathfrak{d}$ to $H_{\hbar, c}$.

In the formula \eqref{formula of hef} of $\textbf{E}$ , $\textbf{F}$ and $\textbf{H}$, we assume furthermore that $ \{y_i\}$ is an orthonormal basis, and $\{x_i\}$ is the corresponding dual basis. 
\begin{prop} 
\label{prop:ext for rca}
The homomorphism $\xi_{a, b}: \Aell\to H_{\hbar, c}$ can be extended to the homomorphism
$\tilde{\xi}_{a, b}: U(\Aell\rtimes \mathfrak{d})\rtimes W\to H_{\hbar, c}$ by the following formulas
\begin{align*}
&w\mapsto w,\,\
d\mapsto \frac{\textbf{h}}{\hbar},\,\  X\mapsto ab^{-1} \frac{\textbf{E}}{\hbar}, \,\ \Delta_0\mapsto ba^{-1}\frac{\textbf{F}}{\hbar},\\
&\delta_{2m}\mapsto -2\frac{a^{2m-1}b^{-1}}{\hbar}\sum_{\alpha\in \Phi^+}\frac{c_\alpha^2}{(\alpha, \alpha)}(x_{\alpha^\vee})^{2m}.
\end{align*}
\end{prop}
\begin{proof}
For simplicity and without loss of generality, we assume that $a=b=1$.
We first show the above homomorphism preserves the relations of $\mathfrak{d}$, see Section \S\ref{sec:derivation} for the relations. 
From Proposition \ref{prop:RCA hEF}, we know the triple $\frac{\textbf{h}}{\hbar}, \frac{\textbf{E}}{\hbar},\frac{\textbf{F}}{\hbar}$ form an $\mathfrak{sl}_2$-triple. Thus, the homomorphism
$\tilde{\xi}$ preserves the relations of the triple $d, X, \Delta_0$.

It is obvious that $\tilde{\xi}$ preserves the relation $[\delta_{2m}, X]=0$, since the images of $\delta_{2m}$ and $X$ under $\tilde{\xi}$ lie in $\C[\h]$. The fact that $[\textbf{h}, x]=\hbar x$ implies $\tilde{\xi}$ preserves the relation
$[d, \delta_{2m}]=2m\delta_{2m}$.
Now we check $\tilde{\xi}$ preserves the relation $(\ad\Delta_0)^{2m+1}(\delta_{2m})=0$. We have the following relation
\[
[\textbf{F}, x_j]=\hbar y_j, \,\, \text{for any $1\leq j\leq n$.}
\]
Since
\begin{align*}
2[\textbf{F}, x_j]
&=[\sum_{i=1}^n y_i^2, x_j]
=\sum_{i=1}^n \Big(y_i[y_i, x_j]+[y_i, x_j]y_i\Big)\\
&=2\hbar y_j-\sum_{i}\sum_{s}c_s(y_i, \alpha_s)(x_j, \alpha_s^\vee)(y_is+sy_i)\\
&=2\hbar y_j-\sum_{i}\sum_{s}c_s \langle x_i, \alpha_s\rangle(x_j, \alpha_s^\vee)(y_is+sy_i)&&\text{by $(y_i, \alpha_s)=\langle x_i, \alpha_s\rangle$}\\
&=2\hbar y_j-\sum_{\alpha\in \Phi^+}c_s(x_j, \alpha_s^\vee)\Big(\sum_{i} \langle x_i,\alpha_s\rangle(y_is+sy_i)\Big)\\
&=2\hbar y_j-\sum_{\alpha\in \Phi^+}c_s(x_j, \alpha_s^\vee)(\alpha s_{\alpha}+s_{\alpha}\alpha)=2\hbar y_j. &&\text{by $s_{\alpha}\alpha=-\alpha s_{\alpha}$}
\end{align*}
Using above relation, we obtain
\begin{align*}
&(\ad\textbf{F})^{2m+1}x_{\alpha^\vee}^{2m}
=(\ad\textbf{F})^{2m}[\textbf{F}, x_{\alpha^\vee}^{2m}]
=\hbar(\ad\textbf{F})^{2m}\Big(x_{\alpha^\vee}^{2m-1}y_{\alpha^\vee}+y_{\alpha^\vee} x_{\alpha^\vee}^{2m-1}\Big)\\
=&\hbar \Big((\ad\textbf{F})^{2m}x_{\alpha^\vee}^{2m-1}\Big)y_{\alpha^\vee}+\hbar y_{\alpha^\vee}\Big((\ad\textbf{F})^{2m}x_{\alpha^\vee}^{2m-1}\Big)
=0.
\end{align*}
The last equality follows from the fact  $(\ad\textbf{F})^{2m-1}x_{\alpha^\vee}^{2m-1} \in \C[y_{\alpha^\vee}]$, and $[\textbf{F}, y^n]=0$, for any $n$. Therefore, $(\ad\textbf{F})^{2m}x_{\alpha^\vee}^{2m-1}=0$. This shows $\tilde{\xi}$ preserves the relation $(\ad\Delta_0)^{2m+1}(\delta_{2m})=0$. Therefore, $\tilde{\xi}$ preserves relations of $\mathfrak{d}$. 

It is clear that the map $\xi: \Aell\to H_{\hbar, c}$ is compatible with the Weyl group $W$ action. The Weyl group $W$ acts on $\mathfrak{d}$ trivially. We now check that for any $\eta\in \mathfrak{d}$, we have $w(\tilde{\eta})=0$. 
First, we have $\textbf{h}=w(\textbf{h})$. Indeed, for any $w\in W$, $w(x_i)$ are a basis of $\h^*$, and $w(y_i)$ are the corresponding dual basis of $\h$. Then, we have  $\sum_{i=1}^n \Big(x_i y_i+y_ix_i\Big)=\sum_{i=1}^n \Big(w(x_i) w(y_i)+w(y_i) w(x_i)\Big)$.  Thus, $\textbf{h}=w(\textbf{h}).$
We now show that  $w(\textbf{F})=\textbf{F}$, for any $w\in W$. This follows from the computation
\begin{align*}
w(\sum_{i=1}^n y_i^2)=\sum_{i=1}^n (w(y_i))^2=\sum_{i=1}^n (y_i)^2.
\end{align*}
The last equality follows from the fact that $\{y_i \mid 1\leq i \leq n\}$ is an orthonormal basis. Similarly, $w(\textbf{E})=\textbf{E}$, for any $w\in W$. 
We now show $w(\tilde{\xi}(\delta_{2m}))=\tilde{\xi}(\delta_{2m})$. This follows from the fact that $w$ permutes the root system $\Phi$ and preserves the killing form.

Next, we show that the map $\tilde{\xi}$ preserves the relations between $\Aell$ and $\mathfrak{d}$, see Section \S\ref{sec:derivation} for the relations. 
By Proposition \ref{prop:RCA hEF}, it is clear that $\tilde{\xi}$ preserves the relations  $\tilde{d}(x(u))=x(u)$ and $\tilde{d}(y(u))=-y(u)$. The map $\tilde{\xi}$ preserves the relation
$\tilde{d}(t_{\alpha})=0$ follows from the fact $w(\textbf{h})=\textbf{h}$, for any $w\in W$. Indeed, $[\textbf{h}, w]=(\textbf{h}-w(\textbf{h}))w=0$. 

We check $\tilde{\xi}_{a, b}$ preserves the relation $\tilde{\delta}_{2m}(t_{\alpha})=[t_\alpha, (\ad\frac{x(\alpha^\vee)}{2})^{2m}(t_\alpha)]$. 
We have already shown that $[\tilde{\xi}(\delta_{2m}), w]=0$, for any $w\in W$. It suffices to show $[s_\alpha, (\ad\frac{x(\alpha^\vee)}{2})^{2m}(s_\alpha)]=0$. 
Using the fact $\Big[\frac{x(\alpha^\vee)}{2}, s_\alpha\Big]=x(\alpha^\vee)s_\alpha$, we have \begin{align*}
[s_\alpha, (\ad\frac{x(\alpha^\vee)}{2})^{2m}(s_\alpha)]
=[s_\alpha, x(\alpha^\vee)^{2m}s_\alpha]
=s_\alpha x(\alpha^\vee)^{2m}s_\alpha-x(\alpha^\vee)^{2m}=0.
\end{align*}

Finally, we show that $\tilde{\xi}_{a, b}$ preserves relation
\begin{equation}\label{eqn:elta and y}
\tilde{\delta}_{2m}(y(u))
= \frac{1}{2}\sum_{\alpha\in \Phi^+}\alpha(u)\sum_{p+q=2m-1}[(\ad\frac{x(\alpha^\vee)}{2})^{p}
   (t_\alpha), (\ad\frac{x(\alpha^\vee)}{-2})^{q}(t_\alpha)].
\end{equation}
We first compute
\begin{align*}
&\sum_{p+q=2m-1}[(\ad\frac{x(\alpha^\vee)}{2})^{p}
   (s_\alpha), (\ad\frac{x(\alpha^\vee)}{-2})^{q}(s_\alpha)]\\
=&\sum_{p+q=2m-1}(-1)^q[x(\alpha^\vee)^{p}s_\alpha, x(\alpha^\vee)^{q}s_\alpha] &&\text{by $\Big[\frac{x(\alpha^\vee)}{2}, s_\alpha\Big]=x(\alpha^\vee)s_\alpha$}\\
=&\sum_{p+q=2m-1}(-1)^q\Big((-1)^q-(-1)^p\Big) x(\alpha^\vee)^{2m-1}=4m x(\alpha^\vee)^{2m-1}.&&\text{by $s_\alpha x(\alpha^\vee)^n=(-1)^n x(\alpha^\vee)^ns_\alpha$}\end{align*}
Therefore, under the map $\tilde{\xi}$, the right hand side of \eqref{eqn:elta and y}  maps to \[
\frac{1}{2}\sum_{\alpha\in \Phi^+}\alpha(u)4m \left(\frac{2 c_\alpha}{(\alpha, \alpha)}\right)^2 x(\alpha^\vee)^{2m-1}
=4m\sum_{\alpha \in \Phi^+}\frac{c_\alpha^2}{(\alpha, \alpha)}\alpha^\vee(u) x_{\alpha^\vee}^{2m-1}. \]
We now compute the image of the left hand side of \eqref{eqn:elta and y}.
We have
\begin{align*}
[\delta_{2m}, y(u)]
=&-\frac{2}{\hbar}\sum_{\alpha\in \Phi^+}\frac{c_\alpha^2}{(\alpha, \alpha)}[(x_{\alpha^\vee})^{2m}, y(u)]\\
=&-\frac{2}{\hbar}
\sum_{\alpha\in \Phi^+} \frac{c_\alpha^2}{(\alpha, \alpha)}\Big(\sum_{p+q=2m-1}x_{\alpha^\vee}^p [x_{\alpha^\vee}, y(u)]x_{\alpha^\vee}^{q}\Big)\\
=&-\frac{2}{\hbar}
\sum_{\alpha\in \Phi^+}\frac{c_\alpha^2}{(\alpha, \alpha)} \sum_{p+q=2m-1} x_{\alpha^\vee}^p 
\Big(-\hbar\langle u, \alpha^\vee\rangle+\sum_{\gamma\in \Phi^+}c_\gamma\langle \gamma^\vee, u\rangle\langle \gamma, \alpha^\vee\rangle s_\gamma\Big)x_{\alpha^\vee}^{q}\\
=&4m\sum_{\alpha\in \Phi^+} \frac{c_\alpha^2}{(\alpha, \alpha)} \alpha^\vee(u) x_{\alpha^\vee}^{2m-1}.
\end{align*}
The last equality follows from the following calculation.
\begin{align*}
&\sum_{\alpha\in \Phi^+} \frac{c_\alpha^2}{(\alpha, \alpha)}\sum_{p+q=2m-1}
\langle \gamma, \alpha^\vee \rangle \left( x_{\alpha^\vee}^{p} 
s_\gamma x_{\alpha^\vee}^{q} \right)\\
=&\frac{1}{2}\sum_{\alpha\in \Phi^+} \frac{c_\alpha^2}{(\alpha, \alpha)}\sum_{p+q=2m-1}
\langle \gamma, \alpha^\vee \rangle
 x_{\alpha^\vee}^{p} (x_{s_{\gamma}(\alpha^\vee)})^{q} s_\gamma
+\frac{1}{2}\sum_{\alpha\in \Phi^+} \frac{c_\alpha^2}{(\alpha, \alpha)}\sum_{p+q=2m-1}
\langle \gamma, s_{\gamma}(\alpha^\vee) \rangle
 (x_{s_{\gamma}(\alpha^\vee)})^{q} (x_{\alpha^\vee})^{p} s_\gamma
\\
=&\frac{1}{2}\sum_{\alpha\in \Phi^+} \frac{c_\alpha^2}{(\alpha, \alpha)}\sum_{p+q=2m-1}
\langle \gamma, \alpha^\vee+s_{\gamma}(\alpha^\vee) \rangle
 (x_{\alpha^\vee})^{p} (x_{s_{\gamma}(\alpha^\vee)})^{q}  s_\gamma=0,
\end{align*}
since $\langle \gamma, \alpha^\vee+s_{\gamma}(\alpha^\vee) \rangle$=0. 
Therefore, $\tilde{\xi}_{a, b}$ preserves the relation \eqref{eqn:elta and y}.
This completes the proof. 
\end{proof}

\section{The degenerate affine Hecke algebra and the rational Cherednik algebra}
\label{sec:AKZ part}
In this section, we construct a map from the degenerate affine Hecke algebra ${\mathcal H}'$ to the completion of the rational Cherednik algebra $\wh{H}_{\hbar, c}$ by degenerating the connection in Section \S\ref{sec:rca}.
The completion is with respect to the grading of the rational Cherednik algebra $H_{\hbar, c}$
\[
\deg(x)=\deg(y)=1, \deg(w)=0, \deg(\hbar)=\deg(c)=2,
\] for $x\in \h^*, y\in \h$ and $w\in W$. 
\begin{definition}
The degenerate affine Hecke algebra $\mathcal{H}'$ is the associative
algebra generated by $\C W$ and the symmetric algebra $S\h$, subject to the relations,
  \[ s_ix_u-x_{s_i(u)}s_i=k_i(u, \alpha_i), \] 
  for any simple reflection $s_i\in W$ and linear generator $x_u$, $u\in \h$, and
  $k_{i}\in \C$.
\end{definition}
Cherednik in  \cite{Ch1} constructed the affine KZ connection. It is a flat and $W$-equivariant connection on $H\rreg$ valued in the degenerate affine Hecke algebra ${\mathcal H}'$. The affine KZ connection can be obtained by specializing the trigonometric KZ connection in Section \S\ref{sec:degen}. More precisely, we have a map $A_{\trig}\to \mathcal{H}'$, by $t_{\alpha} \mapsto k_\alpha s_\alpha$, and  $X(u)\mapsto x(u)$, for $\alpha\in \Phi$, and $u\in \h$. This gives the  affine KZ connection
\[
\nabla_{\AKZ}=d-\sum_{\alpha \in \Phi^+} \frac{2\pi i t_{\alpha}}{e^{2\pi i \alpha}-1} d \alpha k_{\alpha}s_\alpha
-\sum_{i} x(u^i) d u_i. \]

As $\Im\tau\to \infty$, by Proposition \ref{prop:degeneration}, the elliptic KZ connection valued in the rational Cherednik algebra degenerates to the following trigonometric connection.
\begin{align*}
\nabla=&d-\sum_{\alpha \in \Phi^+} \frac{2\pi i d \alpha}{e^{2\pi i \alpha}-1} \Big(\frac{\hbar}{h^\vee}-\frac{2c_\alpha}{(\alpha|\alpha)} s_{\alpha}\Big)
-\sum_{\alpha \in \Phi^+}
\Big(\frac{2\pi i e^{2\pi i \ad(\frac{x_{\alpha^\vee}}{2})}}{e^{2\pi i \ad(\frac{x_{\alpha^\vee}}{2})}-1 }  -\frac{1}{\ad(\frac{x_{\alpha^\vee}}{2})}\Big)\Big(\frac{\hbar}{h^\vee}-\frac{2c_\alpha}{(\alpha|\alpha)} s_{\alpha}\Big) d\alpha +\sum_{i=1}^{n} y(u^i)du_i\\
=&d-\sum_{\alpha \in \Phi^+}
-\frac{2c_\alpha}{(\alpha|\alpha)} \frac{2\pi i s_{\alpha}}{e^{2\pi i\alpha}-1}d\alpha
-\sum_{\alpha\in \Phi^+}
\left( 
\frac{e^{2\pi i \alpha}+1}{e^{2\pi i \alpha}-1}\frac{\pi i \hbar}{h^\vee}
+\frac{2c_\alpha}{(\alpha|\alpha)} 
\Big(\frac{2\pi i e^{2\pi i x_{\alpha^\vee}}}{e^{2\pi i x_{\alpha^\vee}}-1 }  -\frac{1}{x_{\alpha^\vee}}\Big) s_{\alpha} \right)d\alpha
+\sum_{i=1}^{n}y(u^i)du_i.
\end{align*}
Note that the constant term of $\frac{2\pi i e^{2\pi i \ad(\frac{x_{\alpha^\vee}}{2})}}{e^{2\pi i \ad(\frac{x_{\alpha^\vee}}{2})}-1 }  -\frac{1}{\ad(\frac{x_{\alpha^\vee}}{2})}$ is $\pi i$. 

By the universality of the affine KZ connection $\nabla_{\AKZ}$, we have an algebra homomorphism $\mathcal{H}'\to \wh{H}_{\hbar, c}$ by
\begin{align*}
k_\alpha &\mapsto -\frac{2c_\alpha}{(\alpha|\alpha)},  \,\ w\mapsto w, \,\ \text{for $w \in W$}, \\
 x(u) &\mapsto -y(u)+\sum_{\alpha\in \Phi^+} \alpha(u)\left( 
\frac{e^{2\pi i \alpha}+1}{e^{2\pi i \alpha}-1}\frac{\pi i \hbar}{h^\vee}
+\frac{2c_\alpha}{(\alpha|\alpha)} 
\Big(\frac{2\pi i e^{2\pi i x_{\alpha^\vee}}}{e^{2\pi i x_{\alpha^\vee}}-1 }  -\frac{1}{x_{\alpha^\vee}}\Big) s_{\alpha} \right).
\end{align*}

\section{The elliptic Dunkl operators}
In \cite{BFV}, Buchstaber-Felder-Veselov defined elliptic Dunkl operators for an arbitrary Weyl groups. Etingof and Ma in \cite{EM2} generalised the elliptic Dunkl operators to an abelian variety with a finite group action. Using the elliptic Dunkl operators, Etingof-Ma defined the elliptic Cherednik algebras as a sheaf of algebras on the abelian variety. They also constructed certain representations of the elliptic Cherednik algebra.

In this section, we show that those representations constructed in \cite{EM2} of the elliptic Cherednik algebra arise from the flat connections valued in the rational Cherednik algebra $H_{0, c}$. 

Let $T=P^\vee \otimes \mathcal{E}_{\tau}$ be the abelian variety and $\mathcal{F}$ be the sheaf on $T$ considered in \cite[Section 5]{EM2}. The sheaf $\mathcal{F}$ can be identified with the vector bundle $\h_{\C}\times_{P^\vee \otimes(\Z+\tau\Z)} \C W$ on $T$, whose fiber is the regular representation of $W$. 
We construct an action of the rational Cherednik algebra  $H_{0, c}$ on the fiber $\C W$. Such action induces a flat connection on the vector bundle $\h_{\C}\times_{P^\vee \otimes(\Z+\tau\Z)} \C W$, see Section \S\ref{sec:rca}. 
\begin{prop}
\label{prop:dunkl operator}
The flat connection of Section \S\ref{sec:rca} specialized on the vector bundle $\h_{\C}\times_{P^\vee \otimes(\Z+\tau\Z)} \C W$ coincides with the elliptic Dunkl operator
 \begin{align*}
\nabla=d-\sum_{w\in W}\sum_{\alpha \in \Phi^+}
\frac{2c_{\alpha}}{(\alpha|\alpha)} \frac{ \theta(\alpha+\alpha^\vee(w\rho)}{\theta(\alpha)\theta(\alpha^\vee(w\rho))} s_{\alpha} d\alpha
\end{align*}
in \cite{BFV} and \cite{EM2}. 
\end{prop}

\subsection{The vector bundle on $T$}
In this subsection, we recall the construction of elliptic Dunkl operators in \cite{EM2}. 
Let $T=P^\vee\otimes \mathcal{E}_\tau$ be the abelian variety. 
In \cite{EM2}, the sheaf $\mathcal{F}$ on $T$ is defined as follows. 
Choose $\rho\in \h^*$, such that $\rho$ is not fixed by any $w\in W$.
Consider the trivial rank 1 bundle $\h \times \C$ on $\h$. We define the $P^\vee\oplus \tau P^\vee$ action on $\h \times \C$ by
\[
\lambda_i^\vee: (z, \xi) \mapsto (z+\lambda_i^\vee, \xi ), \,\ 
\tau\lambda_i^\vee: (z, \xi) \mapsto (z+\tau \lambda_i^\vee, \exp(-2\pi i \rho(\lambda_i^\vee))\xi ).
\]
Here, we assume $\Im(\tau)=1$. 
Denote by $\mathcal{L}_{\rho}$ the line bundle on 
$T$  which is the quotient of $\h\times \C$ by the group $P^{\vee}\oplus \tau P^\vee$ action. We use the notation as in \cite{EM2} that $\mathcal{L}_{w}:=\mathcal{L}_{w\rho}$, for $w\in W$. The sheaf $\mathcal{F}=\oplus_{w\in W} \mathcal{L}_{w}$ is defined by taking direct sum of $\mathcal{L}_{w}$ for all $w\in W$. 
Labelling each fiber of $\mathcal{L}_{w}$ by $\C_w$, we identify the vector bundle $\mathcal{F}$ with the vector bundle $\h_{\C}\times_{P^\vee \otimes(\Z+\tau\Z)} \C W$ with rank $|W|$. Under this identification, the action of $P^\vee\oplus \tau P^\vee$ on the fiber $\C W$ is given by
\[
\lambda_i^\vee: (z, \xi) \mapsto (z+\lambda_i^\vee, \xi ), \,\ \text{and}  \,\ 
\tau\lambda_i^\vee: (z, \xi) \mapsto (z+\tau \lambda_i^\vee, \exp(-2\pi i w(\rho)(\lambda_i^\vee))\xi_w), 
\]
where $\xi=(\xi_w)$, for $\xi_w\in \C_w$.

\subsection{The action of the rational Cherednik algebra $H_{0, c}$}
In this subsection, we define the action of rational Cherednik algebra $H_{0, c}$ on the fiber $\C W$, such that the element $\lambda_i^\vee\in \h$ acts by multiplication of $ w\rho(\lambda_i^\vee)$. 

We first recall some facts of the rational Cherednik algebra $H_{0, c}$, see \cite{EG, EM} for details. By the Sataka isomorphism, the center $Z_{0, c}$ of $H_{0, c}$ is isomorphic to the spherical subalgebra $eH_{0, c}e$, where 
\[e=\frac{1}{|W|} \sum_{w\in W} w\in \C W\] 
is the idempotent element. Any irreducible representation of $H_{0, c}$ as a representation of $W$ is isomorphic to the regular representation $\C W$. In particular, the dimension is $|W|$. The moduli space of irreducible representations of $H_{0, c}$ is the Calogero-Moser space $\Spec(Z_{0, c})$.

We construct an action of $H_{0, c}$ on $\C W$. 
This following action is a  generalization of type $A$ case in \cite[Section 9.6]{EM}. As before, we choose $\rho\in \h$ such that $\rho$ is not fixed by any $w\in W$.  Let $E=E_{\rho, \mu}$ be the space of the complex valued functions on the $W$ orbit of $(\rho, \mu)\in \h\times \h^*$. As a representation of the Weyl group $W$, $E_{\rho, \mu}$ is isomorphic to the regular representation of $W$.

We define the action of $H_{0, c}$ on $E_{\rho, \mu}$ as follows. For $x\in \h^*$, and $y\in \h$,
\begin{align*}
&x\cdot F(\bold{a}, \bold{b})=(x, \bold{a}) F(\bold{a}, \bold{b}), \\
&y\cdot F(\bold{a}, \bold{b})=(y, \bold{b})F(\bold{a}, \bold{b})+\sum_{\alpha\in \Phi^+} c_{\alpha}\frac{\alpha(y)}{\alpha(\bold{a})}(1-s_{\alpha}) F(\bold{a}, \bold{b}).
\end{align*}
We now check that it is a well-defined action. 
We check the action respects the relation $[y, y']=0$. We have
\begin{align*}
yy'F(\bold{a}, \bold{b})=
&(y, \bold{b})(y', \bold{b})F(\bold{a}, \bold{b})
+\sum_{\alpha\in \Phi^+} c_{\alpha}\frac{\alpha(y)}{\alpha(\bold{a})}(1-s_{\alpha})(y', \bold{b})F(\bold{a}, \bold{b})\\
&+\sum_{\alpha\in \Phi^+}c_{\alpha} \frac{\alpha(y')}{\alpha(\bold{a})}(y, \bold{b})(1-s_{\alpha}) F(\bold{a}, \bold{b})
+\sum_{\beta, \alpha\in \Phi^+} c_{\beta}c_{\alpha}\frac{\beta(y)}{\beta(\bold{a})}(1-s_{\beta})\frac{\alpha(y')}{\alpha(\bold{a})}(1-s_{\alpha}) F(\bold{a}, \bold{b})\\
=&(y, \bold{b})(y', \bold{b})F(\bold{a}, \bold{b})
+\sum_{\alpha\in \Phi^+}c_{\alpha} \frac{\alpha(y)}{\alpha(\bold{a})}
\Big(
(y', \bold{b})F(\bold{a}, \bold{b})-(y', \bold{b})F(s_{\alpha}\bold{a}, s_{\alpha}\bold{b})
\Big)\\
&+\sum_{\alpha\in \Phi^+} c_{\alpha}\frac{\alpha(y')}{\alpha(\bold{a})}
\Big(
(y, \bold{b})F(\bold{a}, \bold{b})-(y, s_{\alpha}\bold{b})F(s_{\alpha}\bold{a}, s_{\alpha}\bold{b})
\Big)
+\sum_{\beta, \alpha\in \Phi^+} c_{\alpha}c_{\beta}\frac{\beta(y)}{\beta(\bold{a})}\frac{\alpha(y')}{\alpha(\bold{a})}(1-s_{\beta})(1-s_{\alpha}) F(\bold{a}, \bold{b}). 
\end{align*}
Therefore, the action of the commutator $[y, y']$ is given by
\begin{align*}
[y, y'] F(\bold{a}, \bold{b})=
&-\sum_{\alpha\in \Phi^+} c_{\alpha}\frac{\alpha(y)}{\alpha(\bold{a})}
\Big(
(y', \bold{b})F(s_{\alpha}\bold{a}, s_{\alpha}\bold{b})
\Big)-\sum_{\alpha\in \Phi^+} c_{\alpha}\frac{\alpha(y')}{\alpha(\bold{a})}
\Big(
(y, s_{\alpha}\bold{b})F(s_{\alpha}\bold{a}, s_{\alpha}\bold{b})
\Big)\\
&+\sum_{\alpha\in \Phi^+} c_{\alpha}\frac{\alpha(y')}{\alpha(\bold{a})}
\Big(
(y, \bold{b})F(s_{\alpha}\bold{a}, s_{\alpha}\bold{b})
\Big)+\sum_{\alpha\in \Phi^+} c_{\alpha}\frac{\alpha(y)}{\alpha(\bold{a})}
\Big(
(y', s_{\alpha}\bold{b})F(s_{\alpha}\bold{a}, s_{\alpha}\bold{b})
\Big)\\
=&-\sum_{\alpha\in \Phi^+} c_{\alpha}\frac{\alpha(y)}{\alpha(\bold{a})}
\alpha(y')(\bold{b}, \alpha^\vee)F(s_{\alpha}\bold{a}, s_{\alpha}\bold{b})
+\sum_{\alpha\in \Phi^+} c_{\alpha}\frac{\alpha(y')}{\alpha(\bold{a})}
\alpha(y)(\bold{b}, \alpha^\vee)F(s_{\alpha}\bold{a}, s_{\alpha}\bold{b})
\\
=&0.
\end{align*}
We then check the action respects the relation $[y, x]=-\sum_{\alpha\in \Phi^+} c_{\alpha}\alpha(y)\alpha^\vee(x) s_{\alpha}$. We have 
\begin{align*}
[x, y] F(\bold{a}, \bold{b})=
&\sum_{\alpha\in \Phi^+} c_{\alpha}\frac{\alpha(y)}{\alpha(\bold{a})}\Big((x, \bold{a})F(\bold{a}, \bold{b})
-(x, s_{\alpha}\bold{a})F(s_{\alpha} \bold{a}, s_{\alpha}\bold{b})\Big)
-\sum_{\alpha\in \Phi^+} c_{\alpha}\frac{\alpha(y)}{\alpha(\bold{a})}\Big((x, \bold{a})F(\bold{a}, \bold{b})
-(x, \bold{a})F(s_{\alpha} \bold{a}, s_{\alpha}\bold{b})\Big)\\
=&\sum_{\alpha\in \Phi^+} c_{\alpha}\frac{\alpha(y)}{\alpha(\bold{a})}
\alpha(\bold{a})\alpha^\vee(x)F(s_{\alpha} \bold{a}, s_{\alpha}\bold{b})
=\sum_{\alpha\in \Phi^+}c_{\alpha}\alpha(y)\alpha^\vee(x) s_{\alpha} F(\bold{a}, \bold{b}).
\end{align*}
It is straightforward to show that the action respects to other defining relations of $H_{0, c}$. 

Therefore, we have an irreducible representation $\C W$ of the rational Cherednik algebra $H_{0, c}$. It has the following properties. As a representation of the Weyl group $W$, it is isomorphic to the regular representation of $W$. And the element $\lambda_i^\vee\in \h$ acts by multiplication of $ w\rho(\lambda_i^\vee)$. 
\subsection{The comparison}
The KZB connection valued in $H_{\hbar, c}$ in Section \S\ref{sec:rca} can be simplified as follows, when $\hbar=0$. 
\begin{align}
\nabla=&d-\sum_{\alpha \in \Phi^+}
k(\alpha, \ad(\frac{ \alpha^\vee}{2})|\tau) t_{\alpha} d\alpha
   +\sum_{i=1}^{n} y(u^i) du_i \notag\\
=&d-\sum_{\alpha \in \Phi^+}
\frac{ \theta(\alpha+\ad(\frac{ \alpha^\vee}{2})}{\theta(\alpha)\theta(\ad(\frac{ \alpha^\vee}{2}))} t_{\alpha} d\alpha
+\sum_{i=1}^{n}\left(\sum_{\alpha \in \Phi^+}\frac{ \alpha(u^i)}{\ad (\frac{ \alpha^\vee}{2})}t_\alpha
+ y(u^i) \right)du_i \notag\\
=&d-\sum_{\alpha \in \Phi^+}
\frac{ \theta(\alpha+\ad(\frac{ \alpha^\vee}{2})}{\theta(\alpha)\theta(\ad(\frac{ \alpha^\vee}{2}))} t_{\alpha} d\alpha. \label{conn:rca t=0}
\end{align}
The equality \eqref{conn:rca t=0} follows from the following computations. 
For any $v\in \h$, we apply $[x(v), -]$ to the term $\sum_{\alpha \in \Phi^+}\frac{ \alpha(u^i)}{\ad (\frac{ \alpha^\vee}{2})}t_\alpha
+ y(u^i)$. Using the relation of $H_{0, c}$, we have $[x(v), y(u^i)]=-\sum_{\alpha\in \Phi^+} (\alpha, v)(\alpha, u^i) t_\alpha$. 
On the other hand, we have
\begin{align*}
[x(v), \sum_{\alpha\in \Phi^+}\frac{ (\alpha, u^i) }{\ad(\frac{ \alpha^\vee}{2})}t_{\alpha}]
=&\sum_{\alpha\in \Phi^+}\frac{ (\alpha, u^i) }{\ad(\frac{ \alpha^\vee}{2})}[x(v), t_{\alpha}]
=\sum_{\alpha\in \Phi^+}\frac{(\alpha, u^i)(\alpha, v)  }{\ad(\frac{ \alpha^\vee}{2})}\ad(\frac{ \alpha^\vee}{2}) t_{\alpha}
=\sum_{\alpha\in \Phi^+}(\alpha, u^i)(\alpha, v) t_\alpha,
\end{align*}
where the second equality is obtained by decomposing the vector $v$ as $v=v_{\parallel}+v_{\perp}$, where $v_{\parallel}=(\alpha, v) \frac{\alpha^\vee}{2}$, and we have the relation $[x(v_{\perp}), t_{\alpha}]=0$, since $(v_{\perp}, \alpha)=0$. 
This implies \eqref{conn:rca t=0}. 

We specialize the connection \eqref{conn:rca t=0} to the vector bundle $\h\times_{P^\vee\otimes\Lambda_{\tau}} \C W$. Recall that the action of  $\lambda_i^\vee\in \h$ is given by multiplication of $ w\rho(\lambda_i^\vee)$.
This gives the $W$-equivariant flat connection on the bundle 
$\h\times_{P^\vee\otimes\Lambda_{\tau}} \C W$. It takes the following form
\begin{align*}
\nabla=d-\sum_{w\in W}\sum_{\alpha \in \Phi^+}
\frac{2c_{\alpha}}{(\alpha|\alpha)} \frac{ \theta(\alpha+\alpha^\vee(w\rho)}{\theta(\alpha)\theta(\alpha^\vee(w\rho))} s_{\alpha} d\alpha. 
\end{align*}
This is exactly the elliptic Dunkl operator in \cite{BFV} and \cite{EM2}. The 
flat connection defines a representation of $W\rtimes D_{T_{\reg}}$ on the  bundle $\mathcal{F}|_{T_{\reg}}$, where $D_{T_{\reg}}$ is ring of differential operators on $T_{\reg}$. This restricts to the action of the elliptic Cherednik algebra defined in \cite{EM2}.

\newcommand{\arxiv}[1]
{\texttt{\href{http://arxiv.org/abs/#1}{arXiv:#1}}}
\newcommand{\doi}[1]
{\texttt{\href{http://dx.doi.org/#1}{doi:#1}}}
\renewcommand{\MR}[1]{}

\end{document}